\documentclass[12pt]{article}
\usepackage{fullpage,graphicx,psfrag,amsmath,amsfonts}
\usepackage{url,hyperref,algorithm,algorithmic,amsthm}
\graphicspath{{./figures/}}
\pdfoutput=1

\newcommand{\BEAS}{\begin{eqnarray*}}
\newcommand{\EEAS}{\end{eqnarray*}}
\newcommand{\BEQ}{\begin{equation}}
\newcommand{\EEQ}{\end{equation}}
\newcommand{\BIT}{\begin{itemize}}
\newcommand{\EIT}{\end{itemize}}

\newcommand{\eg}{{\it e.g.}}
\newcommand{\ie}{{\it i.e.}}
\newcommand{\cf}{{\it c.f.}}

\newcommand{\reals}{{\mbox{\bf R}}}

\newcommand{\argmin}{\mathop{\rm argmin}}

\newcommand{\argmax}{\mathop{\rm argmax}}

\newtheorem{theorem}{Theorem}

\newtheorem{lemma}[theorem]{Lemma}
\newtheorem{corollary}[theorem]{Corollary}
\newtheorem{proposition}[theorem]{Proposition}
\theoremstyle{definition}

{\begin{quote}}{\end{quote}}

\makeatletter
\long\def\@makecaption#1#2{
   \vskip 9pt
   \begin{small}
   \setbox\@tempboxa\hbox{{\bf #1:} #2}
   \ifdim \wd\@tempboxa > 5.5in
        \begin{center}
        \begin{minipage}[t]{5.5in}
        \addtolength{\baselineskip}{-0.95pt}
        {\bf #1:} #2 \par
        \addtolength{\baselineskip}{0.95pt}
        \end{minipage}
        \end{center}
   \else
	\hbox to\hsize{\hfil\box\@tempboxa\hfil}
   \fi
   \end{small}\par
}
\makeatother

\newcounter{oursection}

\newcounter{lecture}

\title{Globally Convergent Type-I Anderson Acceleration for Non-Smooth
Fixed-Point Iterations}
\author{ Junzi Zhang
\thanks{ICME, Stanford University. \textbf{Email:} junziz@stanford.edu}
\and Brendan O'Donoghue
\thanks{DeepMind, Google. \textbf{Email:} bodonoghue85@gmail.com}
\and  Stephen Boyd
\thanks{Department of Electrical Engineering, Stanford University.
\textbf{Email:} boyd@stanford.edu} }

\date{\today}

\begin{document}
\maketitle

\begin{abstract}
We consider the application of the type-I Anderson
acceleration (\cite{FangSecant}) to solving general non-smooth
fixed-point problems. By interleaving with safe-guarding steps, and
employing a Powell-type regularization and a re-start checking for
strong linear independence of the updates, we propose the first globally
convergent variant of Anderson acceleration assuming only that the
fixed-point iteration is non-expansive. We show by extensive numerical
experiments that many first order algorithms can be improved,
especially in their terminal convergence,
with the proposed algorithm.  Our proposed method
of acceleration is being implemented in 
SCS 2.0 \cite{SCS-2}, one of the default solvers used in
the convex optimization parser-solver CVXPY 1.0 \cite{cvxpy_rewriting}. 
\end{abstract}

\section{Introduction}

We consider solving the following general fixed-point problem:
\BEQ\label{non_eq}
\text{Find $x\in\reals^n$ such that $x=f(x)$,}
\EEQ 
where $f:\reals^n\rightarrow\reals^n$ is potentially non-smooth.
Unless otherwise stated, we assume throughout the paper that $f$ is
non-expansive (in the $\ell_2$-norm), \ie, $\|f(x)-f(y)\|_2\leq \|x-y\|_2$
for all $x,~y\in \reals^n$, and that the solution set
$X=\{x^\star\;|\;x^\star=f(x^\star)\}$ of (\ref{non_eq}) is nonempty.
With these assumptions, (\ref{non_eq}) can be solved by the
Krasnosel'ski\v{\i}-Mann (KM, or averaged) iteration algorithm, which
updates $x^k$ in iteration $k$ to $x^{k+1}=(1-\alpha) x^k+\alpha f(x^k)$,
where $\alpha\in(0,1)$ is an algotihm parameter.
An elementary proof shows the
global convergence of KM iteration to some fixed-point $x^\star\in X$
\cite{MonoPrimer}. 
In one sense, our goal is to accelerate the vanilla KM algorithm.

Fixed-point problems such as~(\ref{non_eq}) arise 
ubiquitously in mathematics, natural science and
social science. For example, to find a Nash equilibrium in a
multi-player game, one can reformulate it as a monotone inclusion
problem under mild assumptions on the utility functions \cite{NashMono},
which can then be further reformulated as a fixed-point problem of the
corresponding (non-expansive) resolvent or Cayley operator
\cite{MonoPrimer}. In general, the solution of most, if not all, convex
optimization problems falls into the above scenario. In fact, almost all
optimization algorithms are iterative, and the goal is to solve the
corresponding fixed-point problem (\ref{non_eq}), where
$f:\reals^n\rightarrow\reals^n$ is the iteration mapping. When the
optimization problem is convex, $f$ is typically non-expansive, and the
solution set of the fixed-point problem is the same as that of the original
optimization problem, or closed related to it.
Another related example is infinite-horizon
discounted Markov Decision Process \cite{MDP,BertDP}, in which the
optimal policy can be found by solving the fixed-point problem of the
associated Bellman operator, which is not non-expansive (in the
$\ell_2$-norm) but is contractive in the $\ell_{\infty}$-norm. 
Such kind of
scenarios are also discussed in \S \ref{contract_norm} as a variant of
our main setting.

In spite of the robustness of the vanilla KM iteration algorithm, the
convergence can be extremely slow in practice, especially when high or
even just moderate accuracy is needed.
Data pre-conditioning and step-size line-search are the two most commonly
used generic approaches to accelerate the convergence of the KM method
\cite{LSaveop}.
To further accelerate the
convergence, a trade-off between the number of iterations and
per-iteration cost is needed. 
In this case, when $f(x)=x-\alpha\nabla F(x)$ 
is the gradient descent mapping for the minimization of the differentiable 
objective function $F(x)$, Newton, quasi-Newton, and
accelerated gradient descent methods (\eg, Nesterov's \cite{Nesterov}) 
can then be used to reduce the overall iteration complexity at 
the cost of increased cost in each step \cite{LuenYe}. 
For more general $f$, Semi-smooth Newton
\cite{ZicoSemNewton, ZaiwenSemNewton} and B-differentiable (quasi-)Newton
\cite{BdiffNewton, BdiffQN}, which generalize their classical
counterparts, have also been proposed and widely studied. More recently,
some hybrid methods, which interleave vanilla KM iterations with
(quasi-)Newton type acceleration steps, are designed to enjoy smaller
per-iteration cost while maintaining fast convergence in practice
\cite{RNA, SuperMann}.

Nevertheless, to our knowledge, apart from (pure) pre-conditioning and
line-search (which can be superimposed on top of other acceleration
schemes), the (local or global) convergence of most, if not all existing
methods require additional assumptions, \eg, some kind of differentiability
around the solution \cite{MoreTran, MonoBroyden}, symmetry of the
Jacobian of $f$ \cite{LFSymm2, LFSymm1}, or symmetry of the approximate
Jacobians in the algorithm \cite{ZhouLiLBFGS, ZhouLiBFGS}. 
Moreover, line search is
(almost) always enforced in these methods to ensure global convergence,
which can be prohibitive when function evaluations are expensive. Our
main goal in this paper is hence to provide a globally convergent
acceleration method with relatively small per-iteration costs, without resorting to
line search or any further assumptions other than non-expansiveness,
thus guaranteeing improvement of a much larger class of algorithms ruled
out by existing methods.

To achieve this goal, we propose to solve (\ref{non_eq}) using the type-I
(or ``good'') Anderson acceleration (AA-I) \cite{FangSecant}, a natural
yet underdeveloped variant of the original Anderson acceleration (AA),
also known as the type-II Anderson acceleration (AA-II) \cite{AA}. Despite
its elegance in implementation, popularity in chemistry and physics, and
success in specific optimization problems, a systematic treatment of AA,
especially AA-I in optimization-related applications is still lacking. One of
the main purposes of this work is thus to showcase the impressive
numerical performance of AA-I on problems from these fields.

On the other hand, both early experiments in \cite{FangSecant} and our
preliminary benchmark tests of SCS 2.0 \cite{SCS-2} show that although
AA-I outperforms AA-II in many cases (matching its name of ``good"), it
also suffers more from instability.  Moreover, few convergence analysis
of AA and its variants (and none for AA-I) for general nonlinear problems
exist in the literature, and the existing ones all require $f$ to be
continuously differentiable (which excludes most algorithms
involving projections, 
and in general proximal operators), and are either local
\cite{ProjBroyden, ConvDIIS, ConvAA} or assume certain non-singularity
(\eg, contractivity) conditions \cite{RNA, SRNA, RNA-DNN}.  Another goal
of this paper is hence to provide modifications that lead to a
stabilized AA-I with convergencence beyond differentiability, locality
and non-singularity. As a result, we obtain global convergence to a
fixed-point with no additional assumptions apart from non-expansiveness.

We emphasize that our analysis does not provide a rate of convergence.
While it would be nice to formally establish that our
modified AA-I algorithm converges faster than vanilla KM, we do not
do this in this paper.  Instead, we show only that convergence occurs.
The benefit of our method is not an improved theoretical convergence rate;
it is instead (a) a formal proof that the method always converges,
under very relaxed conditions, and (b) empirical studies that 
show that terminal convergence, especially to moderately high 
accuracies, is almost always much better than vanilla methods.

\paragraph{Related work.} As its name suggests, AA is an acceleration
algorithm proposed by D. G. Anderson in 1965 \cite{AA}. The earliest
problem that AA dealt with was nonlinear integral equations. Later, 
 developed by another two different communities
\cite{Pulay1,Pulay2}, AA has enjoyed wide application in material
sciences and computational quantum chemistry for the computation of
electronic structures, where it is also known as Pulay/Anderson mixing
and (Pulay's) direct inversion iterative subspace (DIIS), respectively.
In contrast, its name is quite unpopular in the optimization community.
As far as we know, it was not until \cite{FangSecant} connected it with
Broyden's (quasi-Newton) methods, that some applications of AA to
optimization algorithms, including expectation-maximization (EM),
alternating nonnegative least-squares (ANNLS) and alternating
projections (AP), emerged \cite{WalkerNi,DampAA,Apcm}.  More recently,
applications are further extended to machine learning and control
problems, including K-means clustering \cite{Kmeans-AA}, robot
localization \cite{AA-ICP} and computer vision \cite{RNA-CNN}.

There has been a rich literature on applications of AA to specific
problems, especially within the field of computational chemistry and
physics \cite{Appl1,Appl2,Appl3,Appl4,Appl5,Appl6}. In the meantime, an
emerging literature on applications to optimization-related problems is
also witnessed in recent years, as mentioned above.

Nevertheless, theoretical analysis of AA and its variants is relatively
underdeveloped, and most of the theory literature is focused on the full
memory AA-II, \ie, $m_k=k$ for all $k\geq 0$ in Algorithm \ref{alg:AA} below. 
This deviates from the
original AA-II \cite{AA} and in general most of the numerical and
application literature, where limited memory AA is predominant. For
solving general fixed-point problems (or equivalently, nonlinear
equations), perhaps the most related work to ours is \cite{ProjBroyden}
and \cite{ConvDIIS}, among which the former proves local Q-superlinear
convergence of a full-memory version of AA-I, while the latter
establishes local Q-linear convergence for the original (limited memory)
AA-II, both presuming continuous differentiability of $f$ in
(\ref{non_eq}) around the solutions. A slightly more stabilized version
of full-memory AA-I is introduced in \cite{StableMultiSec} by
generalizing the re-starting strategy in \cite{ProjBroyden}, which is then
later globalized using a non-monotone line search method \cite{LF-LS},
assuming Lipschitz differentiability of $f$ \cite{Multisec-Interp}. In
practice, the generalized re-starting strategy is more computationally
expensive, yet the performance improvement is non-obvious
\cite{Multisec-Interp}. This motivates us to keep to the original
re-starting strategy in \cite{ProjBroyden} in \S \ref{restart}. By
assuming in addition contractivity of $f$, a slightly stronger and
cleaner local linear convergence of the original AA-II can also be
obtained \cite{ConvAA}. A similar analysis for noise-corrupted $f$ is
later conducted in \cite{ConvAA_rand}.

Interestingly, the three papers on full-memory AA-I, which to our
knowledge are the only papers analyzing the convergence of AA-I
(variants) for general nonlinear fixed-point problems, are not aware of
the literature stemming from \cite{AA}, and the algorithms there are
termed as projected Broyden's methods. We will discuss the connection between AA
and the Broyden's methods in \S \ref{AA1} following a similar treatment as
in \cite{FangSecant, WalkerNi, ConvDIIS}, which also paves the way for the analysis of
our modified (limited-memory) AA-I.

On the other hand, stronger results have been shown for more special
cases. When $f$ is restricted to affine mappings, finite-step
convergence of full-memory AA is discussed by showing its essential
equivalence to GMRES and Arnoldi method \cite{WalkerNi, ConvLinear}.
More recently, a regularized variant of full-memory AA-II is
rediscovered as regularized nonlinear acceleration (RNA) in \cite{RNA},
in which $f$ is the gradient descent mapping of a strongly convex and
strongly smooth real-valued function. Global linear convergence with
improved rates similar to Nesterov's accelerated gradient descent is
proved using a Chebyshev's acceleration argument. The results are then
extended to stochastic \cite{SRNA} and momentum-based \cite{RNA-DNN}
algorithms.

We are not aware of any previous work on convergence
of the limited memory AA-I or its variants, let alone global
convergence in the absence of (Fr\'echet continuous) differentiability
and non-singularity (or contractivity), which is missing from the entire
AA literature.

\paragraph{Outline.} In \S \ref{AA1}, we introduce the original AA-I
\cite{FangSecant}, and discuss its relation to quasi-Newton methods. In
\S \ref{mod-AA1}, we propose a stabilized AA-I with Powell-type
regularization, re-start checking and safe-guarding steps. A
self-contained convergence analysis of the stabilized AA-I is given in
\S \ref{global_conv}. Finally, we demonstrate the effectiveness of our
proposed algorithms with various numerical examples in \S \ref{exam_aa}.
Extensions and variants to our results are discussed in \S
\ref{ext_var}, followed by a few conclusive remarks in \S
\ref{conclusion}.

\subsection{Notation and definitions}
We list some basic definitions and notation to be used in the rest of
the paper.
We denote the set of real numbers as $\reals$; $\reals_+$ the set of
non-negative real numbers; $\bar{\reals}=\reals\cup\{+\infty\}$ is the
extended real line, and $\reals^n$ the $n$-dimensional Euclidean space
equipped with the inner product $x^Ty$ for $x,y\in\reals^n$ and 
the $\ell_2$-norm $\|\cdot\|_2$. For notational compactness, 
we will alternatively use 
\[ (x_1,\dots,x_n)\quad \mbox{and} \quad
\left[\begin{array}{c}
x_1\\
\vdots\\
x_n
\end{array}\right]
\] to denote a vector in $\reals^n$.

The \textit{proximal operator} of a convex, closed and proper function
$F:\reals^n\rightarrow \bar{\reals}$ is given by
\[
\text{prox}_F(x)=\argmin\nolimits_y\{F(y)+\tfrac{1}{2}\|y-x\|_2^2\}.
\]
For a nonempty, closed and convex set $\mathcal{C}\subseteq \reals^n$,
the indicator function of $\mathcal{C}$ is denoted as
\[
\mathcal{I}_{\mathcal{C}}(x)=\left\{
\begin{array}{ll}
0 & \text{if $x\in \mathcal{C}$}\\
+\infty & \text{otherwise.}
\end{array}
\right.
\] Similarly, we denote the \textit{projection} on $\mathcal{C}$ as
\[
\Pi_{\mathcal{C}}(x)=\argmin\nolimits_{y\in\mathcal{C}}\|x-y\|_2,
\] and the \textit{normal cone}  of $\mathcal{C}$ as
\[ N_{\mathcal{C}}(x)=\{y\in
\reals^n\;|\;\sup\nolimits_{x'\in\mathcal{C}} y^T(x'-x)\leq 0\}.
\] 
The projection $\Pi_{\mathcal{C}}$ is the proximal operator
of $\mathcal{I}_{\mathcal{C}}$.

A mapping $f:\reals^n\rightarrow\reals^n$ is said to be
\textit{non-expansive} if for all $x,~y\in\reals^n$,
\[
\|f(x)-f(y)\|_2\leq \|x-y\|_2.
\] It is said to be \textit{$\gamma$-contractive} in an (arbitrary) norm
$\|\cdot\|$ if for all $x,~y\in\reals^n$,
\[
\|f(x)-f(y)\|\leq \gamma\|x-y\|.
\]

A relation $G:\reals^n\rightarrow2^{\small \reals^n}$ is said to be
\textit{monotone}, if for all $x,~y\in\reals^n$,
\[ (u-v)^T(x-y)\geq 0 ~~\mbox{for all}~~ u\in G(x),~v\in G(y).
\] It is said to be \textit{maximal monotone} if there is no monotone
operator that properly contains it (as a relation, \ie, subset of
$\reals^n\times\reals^n$). We refer interested readers to
\cite{MonoPrimer} for a detailed explanation of relations.
When a relation $G$ is single-valued, it becomes a usual
mapping from $\reals^n$ to $\reals^n$, and the same definition of
(maximal) monotonicity holds.

For a matrix $A=(a_{ij})_{n\times n}\in \reals^{n\times n}$, its
$\ell_2$-norm (or spectral/operator norm) is denoted as
$\|A\|_2=\sup_{\|x\|_2=1}\|Ax\|_2$. The Frobenius norm of $A$ is denoted
as $\|A\|_F=\sqrt{\sum_{i,j=1}^na_{ij}^2}$. 
The spectral radius of a square matrix $A$
is the maximum absolute value eigenvalue, \ie,
\[
\rho(A)=\max\{|\lambda_1|,\dots,|\lambda_n|\},
\]
where $\lambda_1,\dots,\lambda_n$ are eigenvalues of $A$ (with repetitions counted).

For a description of strong convexity and strong smoothness of a
function $F:\reals^n\rightarrow\reals$, see \cite{MonoPrimer}. They will
only be used when it comes to the examples in \S \ref{exam_aa}.

\section{Type-I Anderson acceleration} \label{AA1}
In this section we introduce the original AA-I, with a focus on its
relation to quasi-Newton methods. Following the historical development
from \cite{AA} to \cite{FangSecant}, we naturally motivate it by
beginning with a brief introduction to the original AA-II, making
explicit its connection to the type-II Broyden's method, and then move on
to AA-I as a natural counterpart of the type-I Broyden's method.

\subsection{General framework of AA}\label{AA_frame}
As illustrated in the prototype Algorithm \ref{alg:AA}, the main idea is
to maintain a memory of previous steps, and update the iteration as a
linear combination of the memory with dynamic weights. It can be seen as
a generalization of the KM iteration algorithm, where the latter uses only the most
recent two steps, and the weights are pre-determined, which leads to
sub-linear convergence for non-expansive mappings in general
\cite{MonoPrimer}, and linear convergence under certain additional
assumptions \cite{AveOP}.
\begin{algorithm}[h]
   \caption{\textbf{Anderson Acceleration Prototype (AA)}}
   \label{alg:AA}
\begin{algorithmic}[1]
   \STATE {\bfseries Input:} initial point $x_0$, fixed-point mapping 
   $f:\reals^n\rightarrow\reals^n$.
   \FOR{$k=0, 1, \dots$}
   \STATE Choose $m_k$ (\eg, $m_k=\min\{m, k\}$ for some integer $m\geq 0$). 
   \STATE Select weights $\alpha_j^k$ based on the last $m_k$ iterations 
   satisfying $\sum_{j=0}^{m_k}\alpha_j^k=1$.
   \STATE $x^{k+1}=\sum_{j=0}^{m_k}\alpha_j^kf(x^{k-m_k+j})$.
   \ENDFOR
\end{algorithmic}
\end{algorithm}

The integer $m_k$ is the memory in iteration $k$, since the next
iterate is a linear combination of the images of the last $k$ iterates
under the map $f$.
Based on the choices of the weights $\alpha_j^k$ in line 4 of Algorithm
\ref{alg:AA}, AA is classified into two subclasses \cite{FangSecant},
namely AA-I and AA-II. The terminology indicates a close relationship
between AA and quasi-Newton methods, as we will elaborate in more
details below. While existing literature is mainly focused on AA-II, our
focus is on the less explored AA-I.

\subsection{The original AA: AA-II}\label{AA2}
Define the residual $g:\reals^n\rightarrow\reals^n$ of $f$ to be
$g(x)=x-f(x)$. In AA-II \cite{AA}, for each iteration $k\geq 0$, we
solve the following least squares problem with a normalization
constraint:
\begin{equation}\label{aa2sub}
\begin{array}{ll}
\mbox{minimize} & \|\sum_{j=0}^{m_k}\alpha_j g(x^{k-m_k+j})\|_2^2\\
\mbox{subject to} & \sum_{j=0}^{m_k}\alpha_j=1,
\end{array}
\end{equation}
with variable $\alpha=(\alpha_0,\dots,\alpha_{m_k})$. The weight
vector $\alpha^k=(\alpha_0^k,\dots,\alpha_{m_k}^k)$ in line 4 of
Algorithm \ref{alg:AA} is then chosen as the solution to (\ref{aa2sub}).
The intuition is to minimize the norm of the weighted residuals of the previous
$m_k+1$ iterates. In particular, when $g$ is affine, it is not difficult
to see that (\ref{aa2sub}) directly finds a normalized weight vector $\alpha$, 
minimizing the residual norm $\|g(x^{k+1/2})\|_2$ among all $x^{k+1/2}$ 
that can be represented as $x^{k+1/2}=\sum_{j=0}^{m_k}\alpha_jx^{k-m_k+j}$, 
from which $x^{k+1}=f(x^{k+1/2})$ is then computed with an additional 
fixed-point iteration in line 5 of Algorithm \ref{alg:AA}.

\paragraph{Connection to quasi-Newton methods.} To reveal the connection
between AA-II and quasi-Newton methods, we begin by noticing that the
inner minimization subproblem (\ref{aa2sub}) can be efficiently solved
as an unconstrained least squares problem by a simple variable
elimination \cite{WalkerNi}. More explicitly, we can reformulate
(\ref{aa2sub}) as follows:
\BEQ\label{inAA}
\begin{array}{ll}
\mbox{minimize} & \|g_k-Y_k\gamma\|_2,
\end{array}
\EEQ with variable $\gamma=(\gamma_0,\dots,\gamma_{m_k-1})$. Here
$g_i=g(x^i)$, $Y_k=[y_{k-m_k}~\dots~y_{k-1}]$ with $y_i=g_{i+1}-g_i$ for
each $i$, and $\alpha$ and $\gamma$ are related by $\alpha_0=\gamma_0$,
$\alpha_i=\gamma_i-\gamma_{i-1}$ for $1\leq i\leq m_k-1$ and
$\alpha_{m_k}=1-\gamma_{m_k-1}$.

Assuming for now that $Y_k$ is full column rank, the solution $\gamma^k$ to
(\ref{inAA}) is given by $\gamma^k=(Y_k^TY_k)^{-1}Y_k^Tg_k$, and hence
by the relation between $\alpha^k$ and $\gamma^k$, the next iterate of
AA-II can be represented as
\BEAS
\begin{split}
x^{k+1}&=f(x^k)-\sum_{i=0}^{m_k-1}\gamma_i^k\left(f(x^{k-m_k+i+1})-
f(x^{k-m_k+i})\right)\\
&=x^k-g_k-(S_k-Y_k)\gamma^k\\
&=x^k-(I+(S_k-Y_k)(Y_k^TY_k)^{-1}Y_k^T)g_k\\
&=x^k-H_kg_k,
\end{split}
\EEAS where $S_k=[s_{k-m_k}~\dots~s_{k-1}]$, $s_i=x^{i+1}-x^i$ for each
$i$, and $H_k=I+(S_k-Y_k)(Y_k^TY_k)^{-1}Y_k^T$. It has been observed
that $H_k$ minimizes $\|H_k-I\|_F$ subject to the inverse multi-secant
condition $H_kY_k=S_k$ \cite{FangSecant,WalkerNi}, and hence can be
regarded as an approximate inverse Jacobian of $g$. The update of $x^k$
can then be considered as a quasi-Newton-type update, with $H_k$ being
some sort of generalized second (or type-II) Broyden's update \cite{Broyden} 
of $I$ satisfying the inverse multi-secant condition.

It's worth noticing that a close variant of AA-II, with an additional
non-negative constraint $\alpha\geq 0$, is also widely used to
accelerate the SCF (self-consistent field) iteration in the electronic
structure computation. Such methods are typically referred to as
``energy DIIS'' in literature \cite{EDIIS}. However, the inner
minimization problem of energy DIIS has to be solved as a generic convex
quadratic program. Moreover, our preliminary experiments suggest that it
may not work as well for many optimization algorithms (\eg, SCS).

\subsection{AA-I}\label{aa1}
In the quasi-Newton literature, the type-II Broyden's update is often
termed as the ``bad Broyden's method''. In comparison, the so-called
``good Broyden's method", or type-I Broyden's method, which directly
approximates the Jacobian of $g$, typically seems to yield better
numerical performance \cite{GB-Broyden}.

In the same spirit, we define the type-I AA (AA-I) \cite{FangSecant}, in
which we find an approximate Jacobian of $g$ minimizing $\|B_k-I\|_F$
subject to the multi-secant condition $B_kS_k=Y_k$. Assuming for now
that $S_k$ is full column rank, we obtain (by symmetry) that
\BEQ\label{Bk} B_k=I+(Y_k-S_k)(S_k^TS_k)^{-1}S_k^T,
\EEQ and the update scheme is defined as
\BEQ\label{xupdate} x^{k+1}=x^k-B_k^{-1}g_k,
\EEQ assuming $B_k$ to be invertible. We will deal with the potential
rank deficiency of $S_k$ and singularity of $B_k$ shortly in the next
sections.

A direct application of Woodbury matrix identity shows that
\BEQ\label{Bkinv} B_k^{-1}=I+(S_k-Y_k)(S_k^TY_k)^{-1}S_k^T,
\EEQ where again we have assumed for now that $S_k^TY_k$ is invertible.
Notice that this explicit formula of $B_k^{-1}$ is preferred in that the
most costly step, inversion, is implemented only on a small $m_k\times
m_k$ matrix.

Backtracking the derivation in AA-II, (\ref{xupdate}) can be rewritten as
\BEQ\label{xupdate-aa}
x^{k+1}=x^k-g_k-(S_k-Y_k)\tilde{\gamma}^k=f(x^k)-\sum_{i=0}^{m_k-1}
\tilde{\gamma}_i^k\left(f(x^{k-m_k+i+1})-f(x^{k-m_k+i})\right),
\EEQ where $\tilde{\gamma}^k=(S_k^TY_k)^{-1}S_k^Tg_k$. Now we can see
how AA-I falls into the framework of Algorithm \ref{alg:AA}: here the
weight vector $\alpha^k$ in line $4$ is defined as
$\alpha_0^k=\tilde{\gamma}_0^k$,
$\alpha_i^k=\tilde{\gamma}_i^k-\tilde{\gamma}_{i-1}^k$ for $1\leq i\leq
m_k-1$ and $\alpha_{m_k}^k=1-\tilde{\gamma}_{m_k-1}^k$. Note that
although not as intuitive as the weight vector choice in AA-II, the
computational complexity is exactly the same whenever matrix-vector
multiplication is done prior to matrix-matrix multiplication.

For easier reference, we detail AA-I in the following Algorithm
\ref{alg:AA-I}. As our focus is on the more numerically efficient limited-memory 
versions, we also specify a maximum-memory parameter $m$ in the algorithm.
\begin{algorithm}[H]
   \caption{\textbf{Type-I Anderson Acceleration (AA-I-m)}}
   \label{alg:AA-I}
\begin{algorithmic}[1]
   \STATE {\bfseries Input:} initial point $x_0$, fixed-point 
   mapping $f:\reals^n\rightarrow\reals^n$, max-memory $m>0$.
   \FOR{$k=0, 1, \dots$}
   \STATE Choose $m_k\leq m$ (\eg, $m_k=\min\{m, k\}$ for some integer $m\geq 0$). 
   \STATE Compute $\tilde{\gamma}^k=(S_k^TY_k)^{-1}(S_k^Tg_k)$. 
   \STATE Compute $\alpha_0^k=\tilde{\gamma}_0^k$, 
   $\alpha_i^k=\tilde{\gamma}_i^k-\tilde{\gamma}_{i-1}^k$ for 
   $1\leq i\leq m_k-1$ and $\alpha_{m_k}^k=1-\tilde{\gamma}_{m_k-1}^k$.
   \STATE $x^{k+1}=\sum_{j=0}^{m_k}\alpha_j^kf(x^{k-m_k+j})$.
   \ENDFOR
\end{algorithmic}
\end{algorithm}

Note that in the above algorithm, the iteration may get stuck or suffer
from ill-conditioning if $B_k$, or equivalently either $S_k$ or $Y_k$ is
(approximately) rank-deficient. This is also a major source of numerical
instability in AA-I. We will solve this issue in the next section.

\section{Stabilized type-I Anderson acceleration}\label{mod-AA1}
In this section, we propose several modifications to the vanilla AA-I
(Algorithm \ref{alg:AA-I}) to stabilize its convergence. We begin by
introducing a Powell-type regularization to ensure the non-singularity
of $B_k$. We then introduce a simple re-start checking strategy that
ensures certain strong linear independence of the updates $s_k$. These
together solve the stagnation problem mentioned at the end of the last
section. Finally, we introduce safe-guarding steps that check the
decrease in the residual norm, with which the modifications altogether
lead to global convergence to a solution of (\ref{non_eq}), as we will
show in \S \ref{global_conv}.

\paragraph{Rank-one update.} To motivate the modifications,  we take a
step back to the update formula (\ref{xupdate}) and formalize a closer
connection between AA-I and the type-I Broyden's method in terms of
rank-one update. The counterpart result has been proved for AA-II in
\cite{ConvDIIS}.
\begin{proposition}\label{r1-update}
Suppose that $S_k$ is full rank, then $B_k$ in (\ref{Bk}) can be computed
 inductively from $B_k^0=I$ as follows: 
\BEQ\label{Bkind}
B_k^{i+1}=B_k^{i}+\dfrac{(y_{k-m_k+i}-B_k^{i}s_{k-m_k+i})\hat{s}_{k-m_k+i}^T}
{\hat{s}_{k-m_k+i}^Ts_{k-m_k+i}}, \quad i=0,\dots,m_k-1
\EEQ
with $B_k=B_k^{m_k}$. Here $\{\hat{s}_i\}_{i=k-m_k}^{k-1}$ is the Gram-Schmidt 
orthogonalization of $\{s_i\}_{i=k-m_k}^{k-1}$, \ie,
\BEQ\label{sorth}
\hat{s}_i=s_i-\sum_{j=k-m_k}^{i-1}\dfrac{\hat{s}_j^Ts_i}{\hat{s}_j^T\hat{s}_j}\hat{s}_j,
 \quad i=k-m_k,\dots,k-1.
\EEQ
\end{proposition}

We remark that another similar rank-one update formula for full-memory 
AA-I is presented in \cite{ProjBroyden}. However, the result there corresponds 
to successive minimization of $\|B_k^{i+1}-B_k^i\|_F$ with the multi-secant
constraints, instead of the direct minimization of $\|B_k-I\|_F$. It's
thus non-obvious how we can apply their result here, and we instead
provide a self-contained proof in the Appendix. The basic idea is to
prove by induction, and to fix $B_k$ by its restrictions to
$\text{span}(S_k)$ and its orthogonal complement, respectively.

\subsection{Powell-type regularization}\label{powell}

To fix the potential singularity of $B_k$, we introduce a Powell-type
regularization to the rank-one update formula (\ref{Bkind}). The idea is
to specify a parameter $\bar{\theta}\in(0,1)$, and simply replace
$y_{k-m_k+i}$ in (\ref{Bkind}) with
\BEQ\label{ytilde}
\tilde{y}_{k-m_k+i}=\theta_k^iy_{k-m_k+i}+(1-\theta_k^i)B_k^is_{k-m_k+i},
\EEQ where $\theta_k^i=\phi_{\bar{\theta}}(\eta_k^i)$ is defined with
\BEQ\label{thetak}
\phi_{\bar{\theta}}(\eta)=
\left\{
\begin{array}{ll}
1 & \text{if $|\eta|\geq \bar{\theta}$}\\
\frac{1-\textbf{sign}(\eta)\bar{\theta}}{1-\eta} 
&\text{if $|\eta|<\bar{\theta}$}
\end{array}
\right.
\EEQ and
$\eta_k^i=\frac{\hat{s}_{k-m_k+i}^T(B_k^i)^{-1}y_{k-m_k+i}}
{\|\hat{s}_{k-m_k+i}\|_2^2}$.
Here we adopt the convention that $\textbf{sign}(0)=1$. The formulation
is almost the same as the original Powell's trick used in \cite{Powell},
but we redefine $\eta_k$ to take the orthogonalization into
considerations. Similar ideas have also been introduced in \cite{RNA}
and \cite{DampAA} by adding a Levenberg-Marquardt-type regularization.
However, such tricks are designed for stabilizing least-squares problems
in AA-II, which are not applicable here.

We remark that the update remains unmodified when $\bar{\theta}=0$. On
the other hand, when $\bar{\theta}=1$, (\ref{xupdate}) reduces to the
vanilla fixed-point iteration associated with (\ref{non_eq}). Hence
$\bar{\theta}$ serves as a bridge between the two extremes. By
definition, we immediately see that
$\theta_k^i\in[1-\bar{\theta},1+\bar{\theta}]$, which turns out to be a
useful bound in the subsequent derivations.

The following lemma establishes the non-singularity of the modified
$B_k$, which also indicates how $\bar{\theta}$ trades off between
stability and efficiency.

\begin{lemma}\label{invertible}
Suppose $\{s_i\}_{i=k-m_k}^{k-1}$ to be an arbitrary sequence in
 $\reals^n$. Define $B_k=B_k^{m_k}$ inductively from $B_k^0=I$ as
\BEQ\label{mod_Bkind}
\begin{split}
B_k^{i+1}=B_k^{i}+\dfrac{(\tilde{y}_{k-m_k+i}-B_k^{i}s_{k-m_k+i})
\hat{s}_{k-m_k+i}^T}{\hat{s}_{k-m_k+i}^Ts_{k-m_k+i}}, \quad i=0,\dots,m_k-1,
\end{split}
\EEQ
with $\hat{s}_{k-m_k+i}$ and $\tilde{y}_{k-m_k+i}$ defined as in 
(\ref{sorth}) and (\ref{ytilde}), respectively. Suppose that the updates 
above are all well-defined. Then $|\text{det}(B_k)|\geq \bar{\theta}^{m_k}>0$, 
and in particular, $B_k$ is invertible.
\end{lemma}
\begin{proof}
We prove by induction that $|\text{det}(B_k^i)|\geq \bar{\theta}^i$. 
The base case when $i=0$ is trivial. Now suppose that we have proved the 
claim for $B_k^i$. By Sylvester's determinant identity, we have
\BEAS
\begin{split}
|\text{det}(B_k^{i+1})|&=|\text{det}(B_k^i)|\left|\text{det}
\left(I+\theta_k^i\dfrac{((B_k^i)^{-1}{y}_{k-m_k+i}-s_{k-m_k+i})
\hat{s}_{k-m_k+i}^T}{\hat{s}_{k-m_k+i}^Ts_{k-m_k+i}}\right)\right|\\
&=|\text{det}(B_k^i)|\left|1+\theta_k^i\dfrac{\hat{s}_{k-m_k+i}^T
((B_k^i)^{-1}{y}_{k-m_k+i}-s_{k-m_k+i})}{\hat{s}_{k-m_k+i}^Ts_{k-m_k+i}}\right|\\
&=|\text{det}(B_k^i)|\left|1-\theta_k^i(1-\eta_k^i)\right|\geq \bar{\theta}^i\cdot
\left\{\begin{array}{ll}
|\eta_k^i|, & \text{$|\eta_k^i|\geq \bar{\theta}$}\\
|\text{sgn}(\eta_k^i)\bar{\theta}|, & \text{$|\eta_k^i|<\bar{\theta}$}
\end{array}\right.\geq \bar{\theta}^{i+1}.
\end{split}
\EEAS
By induction, this completes our proof.
\end{proof}

Now that we have established the non-singularity of the modified $B_k$,
defining $H_k=B_k^{-1}$, we can directly update $H_k=H_k^{m_k}$ from
$H_k^0=I$ as follows:
\BEQ\label{Hkind}
H_k^{i+1}=H_k^i+\dfrac{(s_{k-m_k+i}-H_k^i\tilde{y}_{k-m_k+i})
\hat{s}_{k-m_k+i}^TH_k^i}{\hat{s}_{k-m_k+i}^TH_k^i\tilde{y}_{k-m_k+i}},
\quad i=0,\dots,m_k-1,
\EEQ again with $\hat{s}_{k-m_k+i}$ and $\tilde{y}_{k-m_k+i}$ defined as
in (\ref{sorth}) and (\ref{ytilde}), respectively. This can be easily
seen by a direct application of the Sherman-Morrison formula. Notice
that the $H_k$ hereafter is different from the one in \S \ref{AA2} for
AA-II.

It's worth pointing out that \cite{Multisec-Interp} also considers
selecting an appropriate $\theta_k^i$ in (\ref{ytilde}) to ensure the
non-singularity of $B_k$, but an explicit choice of $\theta_k^i$ is not
provided to guarantee its existence. Moreover, apart from the
well-defined-ness of the iterations, the modification is neither needed
in the proof, nor in the smooth numerical examples there as claimed by
the authors. In contrast, in our general non-smooth settings the
modification is both significant in theory and practice, as we will see
below.

\subsection{Re-start checking}\label{restart}
In this section, we introduce a re-start checking strategy proposed in
\cite{ProjBroyden}, and use it to establish uniform bounds on the
approximate (inverse) Jacobians, which turns out to be essential to the
final global convergence, as we will see in \S \ref{global_conv}.

Notice that the update formula (\ref{mod_Bkind}) is well-defined as long
as $\hat{s}_{k-m_k+i}\neq 0$, in which case the denominator
$\hat{s}_{k-m_k+i}^Ts_{k-m_k+i}=\|\hat{s}_{k-m_k+i}\|_2^2>0$. However,
unless $g_{k-m_k+i}=0$ for some $i=0,\dots,m_k-1$, in which case the
problem is already solved, we will always have
\[ s_{k-m_k+i}=-B_{k-m_k+i}^{-1}g_{k-m_k+i}\neq 0,
\]  where we used Lemma \ref{invertible} to deduce that $B_{k-m_k+i}$ is
invertible.

This means that the only case when the updates in (\ref{mod_Bkind})
break down is $s_{k-m_k+i}\neq0$ while $\hat{s}_{k-m_k+i}= 0$.
Unfortunately, such a scenario is indeed possible if $m_k$ is chosen as
$\min\{m,k\}$ for some fixed $1\leq m\leq \infty$ (with $m=\infty$
usually called ``full"-memory), a fixed-memory strategy most commonly
used in the literature. In particular, when $m$ is greater than the
problem dimension $n$, we will always have $\hat{s}_{k}=0$ for $k>n$ due
to linear dependence.

To address this issue, we enforce a re-start checking step that clears
the memory immediately before the algorithm is close to stagnation. More
explicitly, we keep $m_k$ growing, until either $m_k=m+1$ for some
integer $1\leq m<\infty$ or $\|\hat{s}_{k-1}\|_2<\tau\|s_{k-1}\|_2$, in
which case $m_k$ is reset to $0$ (\ie, no orthogonalization). The
process is then repeated.  Formally, the following rule is adopted to
select $m_k$ in each iteration $k\geq 0$, initialized from $m_0=0$:
\BEQ\label{mk_rule}
\begin{split}
&\text{Update $m_k=m_{k-1}+1$. If $m_k=m+1$ or $\|\hat{s}_{k-1}\|_2
<\tau\|s_{k-1}\|_2$, then reset $m_k=0$.}
\end{split}
\EEQ

Here $\tau\in(0,1)$ is pre-specified. The main idea is to make sure that
$\hat{s}_k\neq 0$ whenever $s_k$ is so, which ensures that the modified
updates (\ref{mod_Bkind}) and (\ref{Hkind}) won't break down before
reaching a solution. We  actually require a bit more by imposing a
positive parameter $\tau$, which characterizes a strong linear
independence between $s_k$ and the previous updates. This leads to
boundedness of $B_k$, as described in the following lemma.

\begin{lemma}\label{Bkbound}
Assume the same conditions as in Lemma \ref{invertible}, and in addition 
that $m_k$ is chosen by rule (\ref{mk_rule}). Then we have $\|B_k\|\leq 
3(1+\bar{\theta}+\tau)^m/\tau^m-2$ for all $k\geq 0$. 
\end{lemma}
\begin{proof}
Notice that by rule (\ref{mk_rule}), we have $\|\hat{s}_k\|_2\geq 
\tau\|s_k\|_2$ and $m_k\leq m$ for all $k\geq 0$. Hence by (\ref{mod_Bkind}), 
we have that
\[
\begin{split}
\|B_k^{i+1}\|_2&\leq \|B_k^i\|_2+\theta_k^i\dfrac{\|y_{k-m_k+i}
-B_k^is_{k-m_k+i}\|_2}{\|\hat{s}_{k-m_k+i}\|_2}\\
&\leq \|B_k^i\|_2+\dfrac{1+\bar{\theta}}{\tau}\dfrac{\|y_{k-m_k+i}
-B_k^is_{k-m_k+i}\|_2}{\|s_{k-m_k+i}\|_2}.
\end{split}
\]
Noticing that $y_{k-m_k+i}=g(x^{k-m_k+i+1})-g(x^{k-m_k+i})$ and that 
$f(x)(=x-g(x))$ is non-expansive, we see that
\[
\|B_k^{i+1}\|_2\leq \dfrac{1+\bar{\theta}+\tau}{\tau}\|B_k^i\|_2+
\dfrac{2(1+\bar{\theta})}{\tau},
\]
and hence by telescoping the above inequality and the fact that 
$\|B_k^0\|_2=1$, we conclude that
\[
\|B_k\|_2=\|B_k^{m_k}\|_2\leq 3\left(\dfrac{1+\bar{\theta}
+\tau}{\tau}\right)^m-2.
\]
This completes our proof.
\end{proof}

In sum, combining the modified updates with the re-starting choice of
$m_k$, the rank-deficiency problem mentioned at the end of \S \ref{aa1}
is completely resolved. In particular, the full-rank assumption on $S_k$
is no longer necessary. Moreover, the inverse $H_k=B_k^{-1}$ is also
bounded, as described in the following corollary.
\begin{corollary}\label{Hkbound}
Under the same assumptions in Lemma \ref{Bkbound}, we have for all 
$k\geq0$ that 
\BEQ\label{Hkbdineq}
\|H_k\|_2\leq \left(3\left(\dfrac{1+\bar{\theta}+\tau}{\tau}\right)^m
-2\right)^{n-1}/\bar{\theta}^m.
\EEQ
\end{corollary}
\begin{proof}
Denote the singular values of $B_k$ as $\sigma_1\geq \cdots\geq\sigma_n$. 
Then by Lemma \ref{invertible}, we have $\prod_{i=1}^n\sigma_i
\geq\bar{\theta}^{m_k}\geq \bar{\theta}^m$. On the other hand, 
by Lemma \ref{Bkbound}, we have $\sigma_1\leq 3(1+\bar{\theta}+\tau)^m/\tau^m-2$. 
Hence we obtain that
\[
\|H_k\|_2=1/\sigma_n\leq \prod_{i=1}^{n-1}\sigma_i/\bar{\theta}^m\leq 
\left(3\left(\dfrac{1+\bar{\theta}+\tau}{\tau}\right)^m-2\right)^{n-1}/\bar{\theta}^m,
\]
which finishes our proof.
\end{proof}

We remark that for type-II methods as in \cite{ConvDIIS}, the algorithm
can already get stuck if $g(x^{k+1})=g(x^k)$, which is not informative
enough for us to say anything. That's also one of the reasons for
favoring the type-I AA in this paper. It's also worth mentioning that
empirical results in \cite{Restart-Pulay} and \cite{DampAA} have already
suggested that cleaning memories from time to time improves performance
significantly for self-consistent field (SCF) methods and EM-type
algorithms, partially supporting our modification here.

Notice that when $m_k$ is chosen by rule (\ref{mk_rule}) and $B_k$ is
computed as in Lemma \ref{invertible}, we have $B_k^i=B_{k-m_k+i}$. This
means that in iteration $k$, only a rank-one update (\ref{mod_Bkind})
with $i=m_k-1$ is needed, which yields $B_k=B_k^{m_k}$ from
$B_{k-1}=B_k^{m_k-1}$. Moreover, we can remove the necessity of
maintaining updates for $B_k^i$ used in Powell's regularization by
noticing that
$B_k^is_{k-m_k+i}=B_{k-m_k+i}s_{k-m_k+i}=-B_{k-m_k+i}B_{k-m_k+i}^{-1}
g_{k-m_k+i}=-g_{k-m_k+i}$.

\subsection{Safe-guarding steps}\label{safeguard}

We are now ready to introduce the final piece for our modified AA-I
algorithm. The main idea is to interleave AA-I steps with the vanilla KM
iteration steps to safe-guard the decrease in residual norms $g$. In
particular, we check if the current residual norm is sufficiently small,
and replace it with the $\alpha$-averaged (or KM) operator of $f$ in
(\ref{non_eq}) (defined as $f_{\alpha}(x)=(1-\alpha) x+\alpha f(x)$)
whenever not.

The idea of interleaving AA with vanilla iterations has also been
considered in \cite{Periodic-Pulay} with constant periods, and is
observed to improve both accuracy and speed for a certain class of
algorithms (\eg, SCF), despite that no theoretical guarantees for
convergence is provided. Similar ideas have been applied to regularized
AA \cite{RNA} and the classical Broyden's methods \cite{SuperMann} to seek
for smaller per-iteration costs without sacrificing much the
acceleration effects.

The resulting algorithm, combining all the aforementioned tricks, is
summarized as Algorithm \ref{alg:AA-I-safe}. Here, lines 4-8 perform
re-start checking (rule (\ref{mk_rule})) described in \S \ref{restart},
lines 9-11 perform the Powell-type regularization (update
(\ref{mod_Bkind})) described in \S \ref{powell}, and lines 12-14 execute
the safe-guarding strategy described above. As mentioned at the end of
\S \ref{restart}, only a rank-one update of (\ref{mod_Bkind}) from
$i=m_k-1$ is performed in iteration $k$, in which case the subscript
$k-m_k+i$ becomes $k-1$.

Notice that in line 5 of Algorithm \ref{alg:AA-I-safe}, instead of 
defining $s_{k-1}=x^k-x^{k-1}$ and $y_{k-1}=g(x^k)-g(x^{k-1})$ as in 
\S \ref{aa1}, we redefine it using the AA-I trial update $\tilde{x}^k$ 
to ensure that $B_{k-1}s_{k-1}=-B_{k-1}B_{k-1}^{-1}g_{k-1}=-g_{k-1}$ 
still holds as mentioned at the end of \S \ref{restart}, which makes 
it possible to get rid of maintaining an update for $B_k$. 
\begin{algorithm}[h]\label{algorithm3}
   \caption{\textbf{Stablized Type-I Anderson Acceleration (AA-I-S-m)}}
   \label{alg:AA-I-safe}
\begin{algorithmic}[1]
   \STATE {\bfseries Input:} initial point $x_0$, fixed-point mapping 
   $f:\reals^n\rightarrow\reals^n$, regularization constants 
   $\bar{\theta},~\tau,~\alpha\in(0,1)$, safe-guarding constants 
   $D,~\epsilon>0$, max-memory $m>0$.
   \STATE Initialize $H_0=I,~m_0=n_{AA}=0,~\bar{U}=\|g_0\|_2$, 
   and compute $x^1=\tilde{x}^1=f_{\alpha}(x^0)$.
   \FOR{$k=1, ~2, ~\dots$}
       \STATE $m_k=m_{k-1}+1$.
       \STATE Compute $s_{k-1}=\tilde{x}^{k}-x^{k-1}$, $y_{k-1}
       =g(\tilde{x}^{k})-g(x^{k-1})$.
    \STATE Compute $\hat{s}_{k-1}=s_{k-1}-\sum_{j=k-m_k}^{k-2}
    \frac{\hat{s}_j^Ts_{k-1}}{\hat{s}_j^T\hat{s}_j}\hat{s}_j$.
    \STATE {\bf If} $m_k=m+1$  or $\|\hat{s}_{k-1}\|_2<\tau\|s_{k-1}\|_2$
      \STATE \hspace{1cm} reset $m_k=0$, $\hat{s}_{k-1}=s_{k-1}$, and $H_{k-1}=I$.
   \STATE Compute $\tilde{y}_{k-1}=\theta_{k-1}y_{k-1}-(1-\theta_{k-1})g_{k-1}$
   \STATE \hspace{1.7cm} with $\theta_{k-1}=\phi_{\bar{\theta}}(\gamma_{k-1})$ 
   and $\gamma_{k-1}=\hat{s}_{k-1}^TH_{k-1}y_{k-1}/\|\hat{s}_{k-1}\|^2$.
   \STATE Update $H_k=H_{k-1}+\frac{(s_{k-1}-H_{k-1}\tilde{y}_{k-1})
   \hat{s}_{k-1}^TH_{k-1}}{\hat{s}_{k-1}^TH_{k-1}\tilde{y}_{k-1}}$, 
   and $\tilde{x}^{k+1}=x^k-H_kg_k$.
   \STATE {\bf If} $\|g_k\|\leq D\bar{U}(n_{AA}+1)^{-(1+\epsilon)}$
   \STATE \hspace{1cm} $x^{k+1}=\tilde{x}^{k+1}$, $n_{AA}=n_{AA}+1$.
   \STATE {\bf else} $x^{k+1}=f_{\alpha}(x^k)$.
   \ENDFOR
\end{algorithmic}
\end{algorithm}

We remark that the assumptions in Lemma \ref{invertible}, Lemma 
\ref{Bkbound} and Corollary \ref{Hkbound} all hold for Algorithm 
\ref{alg:AA-I-safe} unless a solution is reached and the problem 
is solved, despite that the updates are modified in line 4 and the 
safe-guarding strategy is introduced in lines 12-14. This comes 
immediately from the arbitrariness of the update sequence 
$\{s_i\}_{i=k-m_k}^{k-1}$ (c.f. Lemma \ref{invertible}),  
Formally, we have the following corollary.
\begin{corollary}\label{algHkbound}
In Algorithm \ref{alg:AA-I-safe}, the inequality (\ref{Hkbdineq}) 
holds for all $k\geq 0$. Moreover, the condition number of $H_k$ 
is uniformly bounded by 
\[
\textbf{cond}(H_k)\leq \left(3\left(\dfrac{1+\bar{\theta}+\tau}
{\tau}\right)^m-2\right)^n/\bar{\theta}^m.
\]
\end{corollary}
The proof is a simple combination of the results in Lemma \ref{Bkbound}
and Corollary \ref{Hkbound}.

\section{Analysis of global convergence}\label{global_conv}
In this section, we give a self-contained proof for global convergence
of Algorithm \ref{alg:AA-I-safe}. The proof can be divided into three
steps. Firstly, we prove that the residual $g_k$ converges to $0$.  We
then show that $\|x^k-y\|_2$ converges to some finite limit for any
fixed point $y\in X$ of $f$. Finally, we we show that $x^k$ converges to
some solution to (\ref{non_eq}). We note that some of the arguments are
motivated by the proofs in \cite{Fejer}.

We begin by noticing that $x^{k+1}$ either equals $x^k-H_kg_k$ or
$f_{\alpha}(x^k)$, depending on whether the checking in line 12 of
Algorithm \ref{alg:AA-I-safe} passes or not. We partition the iteration
counts into two subsets accordingly, with $K_{AA}=\{k_0,k_1,\dots\}$
being those iterations that passes line 12, while
$K_{KM}=\{l_0,l_1,\dots\}$ being the rest that goes to line 14.

\paragraph{Step 1: Convergence of $g_k$.} Consider $y\in X$ an arbitrary
fixed point of $f$.

For $k_i\in K_{AA}$ ($i\geq 0$), by Corollary \ref{algHkbound}, we have
$\|H_{k_i}\|_2\leq C$ for some constant $C$ independent of the iteration
count, and hence
\BEQ\label{ineq1}
\begin{split}
\|x^{k_i+1}-y\|_2&\leq \|x^{k_i}-y\|_2+\|H_{k_i}g_{k_i}\|_2\\
&\leq \|x^{k_i}-y\|_2+C\|g_{k_i}\|_2\leq \|x^{k_i}-y\|_2
+CD\bar{U}(i+1)^{-(1+\epsilon)}.
\end{split}
\EEQ

For $l_i\in K_{KM}$ ($i\geq 0$), since $f$ is non-expansive, by Theorem
4.25(iii) in \cite{AveOpContr} or inequality (5) in \cite{MonoPrimer},
we have that
\BEQ\label{ineq2}
\|x^{l_i+1}-y\|_2^2\leq
\|x^{l_i}-y\|_2^2-\alpha(1-\alpha)\|g_{l_i}\|_2^2\leq \|x^{l_i}-y\|_2^2.
\EEQ

By telescoping the above inequalities, we obtain that
\BEQ\label{bounded}
\|x^k-y\|_2\leq
\|x^0-y\|_2+CD\bar{U}\sum\nolimits_{i=0}^{\infty}(i+1)^{-(1+\epsilon)}=E<\infty,
\EEQ and hence $\|x^k-y\|_2$ remains bounded for all $k\geq 0$.

Hence by squaring both sides of (\ref{ineq1}), we obtain that
\BEQ\label{ineq3}
\|x^{k_i+1}-y\|_2^2\leq
\|x^{k_i}-y\|_2^2+\underbrace{(CD\bar{U})^2(i+1)^{-(2+2\epsilon)}
+2CDE\bar{U}(i+1)^{-(1+\epsilon)}}_{=\epsilon_{k_i}}.
\EEQ

Combining (\ref{ineq2}) and (\ref{ineq3}), we see that
\BEQ\label{gli}
\begin{split}
\alpha(1-\alpha)\sum\nolimits_{i=0}^{\infty}\|g_{l_i}\|_2^2\leq& 
\|x^0-y\|_2^2+\sum\nolimits_{i=0}^{\infty}\epsilon_{k_i}<\infty,
\end{split}
\EEQ and hence $\lim_{i\rightarrow\infty}\|g_{l_i}\|_2=0$. Noticing that
$\|g_{k_i}\|_2\leq D\bar{U}(i+1)^{-(1+\epsilon)}$ by line 12 of
Algorithm \ref{alg:AA-I-safe}, we also have
$\lim_{i\rightarrow\infty}\|g_{k_i}\|_2=0$. Hence we see that
\BEQ\label{res_conv}
\lim_{k\rightarrow\infty}\|g_k\|_2=0.
\EEQ

Also notice that by defining $\epsilon_{l_i}=0$, we again see from
(\ref{ineq2}) and (\ref{ineq3}) that
\BEQ\label{fejer}
\|x^{k+1}-y\|_2^2\leq \|x^k-y\|_2^2+\epsilon_k,
\EEQ with $\epsilon_k\geq 0$ and
$\sum_{k=0}^{\infty}\epsilon_k=\sum_{i=0}^{\infty}\epsilon_{k_i}<\infty$.

Notice that in the above derivation of (\ref{gli})-(\ref{fejer}), we
have implicitly assumed that both $K_{AA}$ and $K_{KM}$ are infinite.
However, the cases when either of them is finite is even simpler as one
can completely ignore the finite index set.

\paragraph{Step 2: Convergence of $\|x^k-y\|_2$.} Still consider $y\in
X$ an arbitrary fixed point of $f$. We now prove that $\|x^k-y\|_2$
converges. Since $\|x^k-y\|_2\geq 0$, there is a subsequence
$\{j_0,j_1,\dots\}$ such that
$\lim_{i\rightarrow\infty}\|x^{j_i}-y\|_2=\underline{u}
=\liminf_{k\rightarrow\infty}\|x^k-y\|_2$.
For any $\delta>0$, there exists an integer $i_0$ such that
$\|x^{j_{i_0}}-y\|_2\leq \underline{u}+\delta$ and
$\sum_{k=j_{i_0}}^{\infty}\epsilon_k\leq \delta$. This, together with
(\ref{fejer}), implies that for any $k\geq j_{i_0}$,
\BEQ\label{convergence}
\|x^k-y\|_2^2\leq
\|x^{j_{i_0}}-y\|_2^2+\sum\nolimits_{k=j_{i_0}}^{\infty}\epsilon_k\leq
\underline{u}^2+2\delta \underline{u}+\delta^2+\delta,
\EEQ   and in particular, we have
$\limsup_{k\rightarrow\infty}\|x^k-y\|_2^2\leq
\liminf_{k\rightarrow\infty}\|x^k-y\|_2^2+\delta(2\underline{u}+\delta+1)$.
By the arbitrariness of $\delta>0$, we see that $\|x^k-y\|_2^2$ (and
hence $\|x^k-y\|_2$) is convergent.

\paragraph{Step 3: Convergence of $x^k$.} Finally, we show that $x^k$
converges to some solution $x^\star$ of (\ref{non_eq}), \ie,
$x^\star=f(x^\star)$. To see this, notice that since $\|x^k-y\|_2$ is
bounded for $y\in X$, $x^k$ is also bounded. Hence it must have a
convergent subsequence by Weierstrass theorem.

Suppose on the contrary that $x^k$ is not convergent, then there must be
at least two different subsequences $\{k_0',k_1',\dots\}$ and
$\{l_0',l_1',\dots\}$ converging to two different limits $y_1\neq y_2$,
both of which must be fixed points of $f$. This is because that by
(\ref{res_conv}), we have
\[
0=\lim\nolimits_{i\rightarrow\infty}\|g(x^{k_i'})\|_2=\|g(y_1)\|_2,\quad
0=\lim\nolimits_{i\rightarrow\infty}\|g(x^{l_i'})\|_2=\|g(y_2)\|_2,
\] where we used the fact that $f$ is non-expansive and hence
$g(x)=x-f(x)$ is (Lipschitz) continuous.   Now notice that we have
proved that $\alpha(y)=\lim_{k\rightarrow\infty}\|x^k-y\|_2$ exists for
any $y\in X$. By the simple fact that
$\|x^k-y\|_2^2-\|y\|_2^2=\|x^k\|_2^2-2y^Tx^k$, we have
\[
\begin{split}
&\lim_{i\rightarrow\infty}\|x^{k_i'}\|_2^2=\lim_{k\rightarrow\infty}
\|x^k-y\|_2^2-\|y\|_2^2+2y^T\lim_{i\rightarrow\infty}x^{k_i'}=\alpha(y)
-\|y\|_2^2+2y^Ty_1,\\
&\lim_{i\rightarrow\infty}\|x^{l_i'}\|_2^2=\lim_{k\rightarrow\infty}
\|x^k-y\|_2^2-\|y\|_2^2+2y^T\lim_{i\rightarrow\infty}x^{l_i'}=\alpha(y)
-\|y\|_2^2+2y^Ty_2.
\end{split}
\] Subtracting the above inequalities, we obtain that for any $y\in X$,
\[
2y^T(y_1-y_2)=\lim_{i\rightarrow\infty}\|x^{k_i'}\|_2^2
-\lim_{i\rightarrow\infty}\|x^{l_i'}\|_2^2.
\]  By taking $y=y_1$ and $y=y_2$, we see that
$y_1^T(y_1-y_2)=y_2^T(y_1-y_2)$, which implies that $y_1=y_2$, a
contradiction. Hence we conclude that $x^k$ converges to some $\bar{x}$,
which must be a solution as we have by (\ref{res_conv}) that
$0=\lim_{k\rightarrow\infty}\|g(x^k)\|_2=\|g(\bar{x})\|_2$. Here we
again used the fact that $f$ is non-expansive, and hence $g(x)=x-f(x)$
is (Lipschitz) continuous.

In sum, we have the following theorem.
\begin{theorem}\label{glb_conv_thm}
Suppose that $\{x^k\}_{k=0}^{\infty}$ is generated by Algorithm 
\ref{alg:AA-I-safe}, then we have $\lim_{k\rightarrow\infty}x^k=x^\star$,
 where $x^\star=f(x^\star)$ is a solution to (\ref{non_eq}).
\end{theorem}

\section{Numerical results}\label{exam_aa}
In this section we present examples to demonstrate the power of
AA-I. The major focus is on optimization problems and algorithms, where
$f$ in (\ref{non_eq}) comes from the iterative algorithms used to solve
them. For each example, we specify the concrete form of $f$, verify its
non-expansiveness, and check the equivalence between the fixed-point
problem and the original problem.

We compare the performance of the following three algorithms for each
experiment:
\begin{itemize}
 \item  The vanilla algorithm: \eg, gradient descent;
 \item  AA-I-m: Algorithm \ref{alg:AA-I} with max-memory $m$, choosing 
 $m_k=\min\{m,k\}$;
 \item  AA-I-S-m: Algorithm \ref{alg:AA-I-safe} with max-memory $m$.
\end{itemize}
The convergence curves against both clock time (seconds) and iteration numbers
will be shown.

\subsection{Problems and algorithms}\label{prob_alg} 
We begin by describing some example problems and the corresponding
(unaccelerated) algorithms used to solve them.

\subsubsection{Proximal gradient descent}\label{prox_grad}
Consider the following problem:
\BEQ\label{composite}
\begin{array}{ll}
\mbox{minimize} & F_1(x)+F_2(x),\\
\end{array}
\EEQ where $F_1,~F_2:\reals^n\rightarrow\reals$ are convex closed proper
(CCP), and $F_1$ is $L$-strongly smooth.

We solve it using proximal gradient descent, \ie,
\[
\begin{split}
x^{k+1}=\text{prox}_{\alpha F_2}(x^k-\alpha \nabla F_1(x^k)),
\end{split}
\] where $\alpha\in (0,2/L)$.  In our notation, the fixed-point mapping
is $f(x)=\text{prox}_{\alpha F_2}(x-\alpha\nabla F_1(x))$. For a proof
of non-expansiveness for $f$ and the equivalence between the fixed-point
problem and the original optimization problem (\ref{composite}), see
\cite{Proximal}.

\paragraph{Gradient descent (GD).}  When $F_2= 0$, proximal gradient
descent reduces to vanilla gradient descent for unconstrained problems,
\ie, (denoting $F=F_1$)
\[ x^{k+1}=x^k-\alpha \nabla F(x^k),
\] where $\alpha\in(0,2/L)$, and the fixed-point mapping is
$f(x)=x-\alpha\nabla F(x)$.

\paragraph{Projected gradient descent (PGD).}  When
$F_2(x)=\mathcal{I}_{\mathcal{K}}(x)$, with $\mathcal{K}$ being a
nonempty closed and convex set, problem (\ref{composite}) reduces to a
constrained optimization problem. Accordingly, proximal gradient descent
reduces to projected gradient descent, \ie, (denoting $F=F_1$)
\[
\begin{split}
x^{k+1}=\Pi_{\mathcal{K}}(x^k-\alpha\nabla F(x^k)),
\end{split}
\] where $\alpha\in (0,2/L)$, and the fixed-point mapping is
$f(x)=\Pi_{\mathcal{K}}(x-\alpha\nabla F(x))$.

\paragraph{Alternating projection (AP).} When
$F_1(x)=\frac{1}{2}\text{dist}(x,D)^2$ and $F_2(x)=\mathcal{I}_{C}(x)$,
with $C,~D$ being nonempty closed convex sets and $C\cap D\neq
\emptyset$. The problem (\ref{composite}) then reduces to finding an
element $x$ in the intersection $C\cap D$. Noticing that $F_1$ is
$1$-smooth, by choosing $\alpha=1$, proximal gradient descent reduces to
alternating projection, \ie,
\[ x^{k+1}=\Pi_C\Pi_D(x^k),
\] with $f(x)=\Pi_C\Pi_D(x)$.

\paragraph{ISTA.} When $F_2=\mu\|x\|_1$, the problem reduces to
sparsity-regularized regression (\eg, Lasso, when $F_1$ is quadratic).
Accordingly, proximal gradient descent reduces to Iterative
Shrinkage-Thresholding Algorithm (ISTA), \ie, (denoting $F=F_1$)
\[ x^{k+1}=S_{\alpha\mu}(x^k-\alpha\nabla F(x^k)),
\] where $\alpha\in (0,2/L)$, and
\[ S_{\kappa}(x)_i=\textbf{sign}(x_i)(|x_i|-\kappa)_+, \quad
i=1,\dots,n,
\]  is the shrinkage operator. The fixed-point mapping here is
$f(x)=S_{\alpha\mu}(x-\alpha\nabla F(x))$.

\subsubsection{Douglas-Rachford splitting}
Consider the following problem:
\BEQ\label{composite_op}
\mbox{find $x$ such that}~~ 0\in (A+B)(x),
\EEQ where $A,~B:\reals^n\rightarrow2^{\reals^n}$ are two maximal
monotone relations.

Douglas-Rachford splitting (DRS) solves this problem by the following
iteration scheme:
\BEQ\label{DRS} z^{k+1}=f(z^k)=z^k/2+C_{A}C_{B}(z^k)/2,
\EEQ where $C_G$ is the Cayley operator of $G$, defined as
$C_G(x)=2(I+\alpha G)^{-1}(x)-x$, where $I$ is the identity mapping, and
$\alpha>0$ is an arbitrary constant.  Since the Cayley operator $C_G$ of
a maximal monotone relation $G$ is non-expansive and defined over the
entire $\reals^n$, we see that the fixed-point mapping
$f(x)=x/2+C_{A}C_{B}(x)/2$ is a $\frac{1}{2}$-averaged (and hence
non-expansive) operator. The connection between (\ref{composite_op}) and
(\ref{DRS}) is established by the fact that $x$ solves
(\ref{composite_op}) if and only if $z$ solves (\ref{DRS}) and
$x=R_B(z)$, where $R_B$ is the resolvent operator of $B$,
$R_B(x) = (I+\alpha B)^{-1}$.

Below we will implicitly use the facts that subgradients of CCP
functions, linear mappings $Mx$ with $M+M^T\succeq 0$, and normal cones
of nonempty closed convex sets are all maximal monotone. These facts, as
well as the equivalence between (\ref{composite_op}) and (\ref{DRS}),
can all be found in \cite{MonoPrimer}.

Notice that whenever $z^k$ converges to a fixed-point of (\ref{DRS})
(not necessarily following the DRS iteration (\ref{DRS})),
$x^k=R_{B}(z^k)$ converges to a solution of problem
(\ref{composite_op}), where $R_B(x)=(I+\alpha B)^{-1}(x)$ is the
resolvent of $B$. This comes immediately from the equivalence between
(\ref{composite_op}) and (\ref{DRS}) and the fact that $R_B$ is
non-expansive \cite{MonoPrimer} and hence continuous. Together with
Theorem \ref{glb_conv_thm}, this ensures that the application of
Algorithm \ref{alg:AA-I-safe} to the DRS fixed-point problem (\ref{DRS})
leads to the convergence of $x^k=R_B(z^k)$ to a solution of the original
problem.

\paragraph{Consensus optimization (CO).} In consensus optimization
\cite{MonoPrimer}, we seek to solve
\BEQ\label{consensus}
\begin{array}{ll}
&\mbox{minimize}~~\sum_{i=1}^mF_i(x)\\
\end{array}
\EEQ where $F_i:\reals^n\rightarrow\reals$ are all CCP. Rewriting the
problem as
\BEQ\label{consensus2}
\begin{array}{ll}
&\mbox{minimize}~~\sum_{i=1}^mF_i(x_i)+\mathcal{I}_{\{x_1=x_2=\dots=x_m\}}
(x_1,x_2,\dots,x_m),\\
\end{array}
\EEQ the problem reduces to (\ref{composite_op}) with
\[
\begin{split}
A(x)&=(\partial F_1(x_1),\dots,\partial F_m(x_m))^T,\\
B(x)&=\mathcal{N}_{\{x_1=x_2=\dots=x_m\}}(x_1,\dots,x_m).
\end{split}
\]

Since for a CCP function $F:\reals^n\rightarrow\reals$ and a nonempty
closed convex set $C$, $C_{\partial F}(x)=2\text{prox}_{\alpha F}(x)-x$
and $C_{\mathcal{N}_C}(x)=\Pi_C(x)$, we see that the DRS algorithm
reduces to the following:
\[
\begin{split}
x_i^{k+1}&=\text{argmin}_{x_i}~F_i(x_i)+(1/2\alpha)\|x_i-z_i^k\|_2^2,\\
z_i^{k+1}&=z_i^k+2\bar{x}^{k+1}-x_i^{k+1}-\bar{z}^k, \quad i=1,\dots,m.
\end{split}
\] where $\bar{x}^k=\frac{1}{m}\sum_{i=1}^mx_i$, and the fixed-point
mapping $f$ is the mapping from $z^k$ to $z^{k+1}$. As discussed above,
$x^{k+1}$ converges to the solution of (\ref{consensus}) if $z^k$
converges to the fixed-point of $f$, and hence can be deemed as
approximate solutions to the  original problem.

\paragraph{SCS.} Consider the following generic conic optimization
problem:
\BEQ\label{co}
\begin{array}{ll}
\mbox{minimize} & c^Tx\\
\mbox{subject to} & Ax+s=b,\quad s\in \mathcal{K}.
\end{array}
\EEQ where $A\in\reals^{m\times n}$, $b\in\reals^m$, $c\in\reals^n$, and
$\mathcal{K}$ is a nonempty, closed and convex cone. Our goal here is to
find both primal and dual solutions when they are available, and provide
a certificate of infeasibility or unboundedness otherwise \cite{SCS}. To
this end, one seeks to solve the associated self-dual homogeneous
embedding (SDHE) system \cite{SDHE},
\BEQ\label{sdhe} Qu=v,\quad (u,v)^T\in\mathcal{C}\times\mathcal{C}^*,
\EEQ where $u=(x,y,\tau)^T\in\reals^n\times\reals^m\times \reals$,
$v=(r,s,\kappa)^T\in\reals^n\times\reals^m\times \reals$,
$\mathcal{C}=\reals^n\times\mathcal{K}^*\times\reals_+$,
$\mathcal{C}^*=\{0\}^n\times \mathcal{K}\times\reals_+$ is the dual cone
of $\mathcal{C}$, and the SDHE embedding matrix
\[ Q=\left[\begin{array}{lll} 0 & A^T & c\\ -A & 0 & b\\ -c^T &-b^T & 0
\end{array}\right].
\] The SDHE system can then be further reformulated into
(\ref{composite_op}) (\cite{SuperMann}), with
$A(u)=\mathcal{N}_{\mathcal{C}}(u)$, $B(u)=Qu$.     Accordingly, DRS
reduces to splitting conic solver (SCS) \cite{SCS}, \ie,
\[
\begin{split}
\tilde{u}^{k+1}&=(I+Q)^{-1}(u^k+v^k)\\
u^{k+1}&=\Pi_{\mathcal{C}}(\tilde{u}^{k+1}-v^k)\\
v^{k+1}&=v^k-\tilde{u}^{k+1}+u^{k+1},
\end{split}
\]  Notice that here we have actually used an equivalent form of DRS
described in \cite{slides} with change of variables. In our notation,
the fixed-point mapping $f$ is
\[ f(u,v)=
\left[\begin{array}{l}
\Pi_{\mathcal{C}}((I+Q)^{-1}(u+v)-v)\\
v-(I+Q)^{-1}(u+v)+u
\end{array}\right],
\] which is non-expansive (\cf, the appendix in \cite{SCS}).

Notice that with the transformations made, the equivalence and
convergence properties of DRS can not be directly applied here as in the
previous examples. Nevertheless, the equivalence between the fixed-point
problem and the SDHE system here can be seen directly by  noticing that
$f(u,v)=(u,v)^T$ if and only if
\[ (I+Q)^{-1}(u+v)=u,\quad \Pi_{\mathcal{C}}((I+Q)^{-1}(u+v)-v)=u,
\]
\ie, $Qu=v$ and $\Pi_{\mathcal{C}}(u-v)=u$. By Moreau decomposition
\cite{Proximal}, we have
\[
\Pi_{\mathcal{C}}(u-v)+\Pi_{-\mathcal{C}^*}(u-v)=u-v,
\] and hence
\[
\Pi_{\mathcal{C}}(u-v)=u\Leftrightarrow
\Pi_{-\mathcal{C}^*}(u-v)=-v\Leftrightarrow \Pi_{\mathcal{C}^*}(v-u)=v.
\] Hence we see that $f(u,v)=(u,v)^T\Rightarrow Qu=v, ~(u,v)^T\in
\mathcal{C}\times\mathcal{C}^*$. On the other hand, when $Qu=v$ and
$(u,v)\in \mathcal{C}\times\mathcal{C}^*$, we have $u^Tv=u^TQu=0$ by the
skew-symmetry of $Q$, and hence for any $w\in \mathcal{C}$,
\[
\|u-v-w\|_2^2=\|u-w\|_2^2+\|v\|_2^2-2v^T(u-w)=\|u-w\|_2^2+\|v\|_2^2+2v^Tw\geq
\|v\|_2^2,
\] where the last inequality comes from the fact that $v^Tw\geq 0$ as
$v\in \mathcal{C}^*$ and $w\in \mathcal{C}$, and the equality is
achieved if and only if $u=w$. Hence we have $\Pi_{\mathcal{C}}(u-v)=u$,
from which we conclude that $(u,v)^T$ is a fixed-point of $f$ if and
only if $Qu=v$, $(u,v)^T\in\mathcal{C}\times\mathcal{C}^*$, \ie,
$(u,v)^T$ solves the SDHE system.

\subsubsection{Contractive mappings in different norms}\label{contract_norm}
As we can see from (\ref{ineq2}) in the proof of Theorem
\ref{glb_conv_thm}, which does not hold for general norms, the
$\ell_2$-norm in the definition of non-expansiveness is essential to our
analysis of global convergence. Nevertheless, an expansive mapping in
one norm may be non-expansive or even contractive in another norm, as we
will see in the examples below. When a mapping is actually contractive
in some (arbitrary) norm, the global convergence of Algorithm
\ref{alg:AA-I-safe} can still be guaranteed. Formally, we have the
following theorem.

\begin{theorem}\label{mdp_conv}
Suppose that $\{x^k\}_{k=0}^{\infty}$ is generated by Algorithm 
\ref{alg:AA-I-safe}, but with $\alpha=1$, and instead of $f$ being 
non-expansive (in $\ell_2$-norm) in (\ref{non_eq}), $f$ is $\gamma$-contractive 
in some (arbitrary) norm $\|\cdot\|$ (\eg, $l_{\infty}$-norm) on $\reals^n$, 
where $\gamma\in(0,1)$. Then we still have $\lim_{k\rightarrow\infty}x^k=x^\star$, 
where $x^\star=f(x^\star)$ is a solution to (\ref{non_eq}).
\end{theorem}

The proof can be found in the appendix. Notice that the global
convergence in the above algorithm also holds for $\alpha\in(0,1)$, and
the proof is exactly the same apart from replacing $\gamma$ with
$(1-\alpha)+\alpha\gamma$, which is larger than $\gamma$ but is still
smaller than $1$. The only reason for specifying $\alpha=1$ is that it
gives the fastest convergence speed both in theory and practice for
contractive mappings.

\paragraph{Value iteration (VI).} Consider solving a discounted Markov
decision process (MDP) problem with (expected) reward $R(s,a)$,
transition probability $P(s,a,s')$,  initial state distribution
$\pi(\cdot)$, and discount factor $\gamma\in(0,1)$, where $s,\,s'\in
\{1,\dots,S\}$ and $a\in \{1,\dots,A\}$. The goal is to maximize the
(expected) total reward
$\mathbb{E}_{\pi}[\sum_{t=0}^{\infty}\gamma^tr(s_t,\mu(s_t))]$ over all
possible (stationary) policies
$\mu:\{1,\dots,S\}\rightarrow\{1,\dots,A\}$, where $s_{t+1}\sim
P(s_t,\mu(s_t),\cdot)$.

One of the most basic algorithms to solve this problem is the well-known
value iteration algorithm:
\[ x^{k+1}=Tx^k,
\] where $x^k$ approximates the optimal value function
$V^{\star}(s)=\max_{\mu}\mathbb{E}[\sum_{t=0}^{\infty}\gamma^tr(s_t,\mu(s_t))|s_0=s]$, and
$T:\reals^S\rightarrow\reals^S$ is the Bellman operator:
\[
(Tx)_s=\max_{a=1,\dots,A}R(s,a)+\gamma\sum\nolimits_{s'=1}^SP(s,a,s')x_{s'}.
\]

In our notation, the fixed-point mapping $f(x)=T(x)$. A prominent
property of $T$ is that although not necessarily non-expansive in
$\ell_2$-norm, it is $\gamma$-contractive under the $l_{\infty}$-norm, \ie,
\[
\|Tx-Ty\|_{\infty}\leq\gamma\|x-y\|_{\infty}.
\]

By Theorem \ref{mdp_conv}, the global convergence is still guaranteed
when Algorithm \ref{alg:AA-I-safe} is applied to VI here. We also remark
that it would be interesting to apply the accelerated VI to solving the
MDP subproblems in certain reinforcement learning algorithms (\eg, PSRL
\cite{PSRL}, UCRL2 \cite{UCRL2}), where the rewards $r$ and transitions
$P$ are unknown.

\paragraph{Heavy ball (HB).} Consider the following convex quadratic
program (QP),
\BEQ\label{setconstr}
\begin{array}{ll}
\mbox{minimize} & F(x)=\frac{1}{2}x^TAx+b^Tx+c\\
\end{array}
\EEQ where $A\in\reals^{n\times n}$ is positive definite with its
eigenvalues lying between $\mu$ and $L$, $b$ is a constant vector, and
$c$ is a constant scalar.  Equivalently, we consider solving the
nonsingular linear equation $Ax+b=0$. Notice that this is just a special
case of the optimization problem for gradient descent described above,
and the unique optimizer is simply $x^\star=-A^{-1}b$ and can be
obtained by solving the corresponding linear equation. But here we
instead consider solving it using the heavy-ball method, which enjoys a
faster linear convergence rate than the vanilla gradient descent
\cite{HBconv}.

The heavy ball (HB) method is a momentum-based variant of the usual
gradient descent, which takes the following form of iterations:
\[ x^{k+1}= x^k-\alpha(Ax^k+b)+\beta(x^k-x^{k-1}),
\] where $\alpha=\frac{4}{(\sqrt{L}+\sqrt{\mu})^2}$ and
$\beta=\frac{\sqrt{L}-\sqrt{\mu}}{\sqrt{L}+\sqrt{\mu}}$.

Viewing $(x^k,x^{k-1})^T$ as the iteration variable, the fixed-point
mapping $f$  is
\[ f(x',x)=\left[\begin{array}{c} x'-\alpha(Ax'+b)+\beta(x'-x)\\ x'
\end{array}\right]=T\left[
\begin{array}{c}
x'\\
x
\end{array}\right]+h,
\] where
\[ T=
\left[\begin{array}{cc}
(1+\beta)I-\alpha A & -\beta I\\
I & 0
\end{array}\right],\quad h=\left[
\begin{array}{c}
-\alpha b\\
0
\end{array}\right],
\] in which $z$ lies on the segment between $x$ and $x'$, and $I$ is the
$n$-by-$n$ identity matrix.

It's easy to see that $(x',x)^T$ is a fixed-point of $f$ if and only if
$x=x'$ and $Ax'+b=0$, and hence $x=x'$ are both solutions to the
original problem.

In general, $f$ may not be non-expansive in $\ell_2$-norms. However, for any  norm
$\|\cdot\|$ on $\reals^n$, 
\[
\|f(x',x)-f(y',y)\|\leq
\|T\|\|(x'-y',x-y)^T\|,
\] where we use the same notation for the induced norm of $\|\cdot\|$ on
$\reals^{n\times n}$, \ie, $\|T\|=\sup_{x\neq 0}\|Tx\|/\|x\|$.  By
noticing that the eigenvalues of $A$ all lie between $\mu$ and $L$, we
see that the spectral radius of $T$ is upper bounded by \cite{HBconv}
\[
\rho(T)\leq (\sqrt{\kappa}-1)/(\sqrt{\kappa}+1)<1,
\] where $\kappa=L/\mu$. Hence for any sufficiently small $\epsilon$
satisfying $\frac{\sqrt{\kappa}-1}{\sqrt{\kappa}+1}+\epsilon<1$, we can 
define the norm $\|\cdot\|$ as $\|x\|=\|D(1/\epsilon)S^{-1}x\|_1$,
where $\|\cdot\|_1$ is the $l_1$-norm,
\[ T=S ~\textbf{diag}\left(
J_{n_1}(\lambda_1),J_{n_2}(\lambda_2),\dots,J_{n_k}(\lambda_k)
\right) S^{-1}
\] is the Jordan decomposition of $T$, and
\[
D(\eta)=\textbf{diag}\left(D_{n_1}(\eta),D_{n_2}(\eta),\dots,D_{n_k}(\eta)\right),
\] in which  $D_m(\eta)=\textbf{diag}(\eta,\eta^2,\dots,\eta^m)$. Then we
have $\gamma=\|T\|\leq \rho(T)+\epsilon<1$ \cite{SpecRadBd}, and hence
$f$ is $\gamma$-contractive in the norm $\|\cdot\|$.

We remark that the although the above example seems to be a bit trivial
as $f$ is an affine mapping, there has been no global convergence result
even for these simple cases as the existing analysis for applying AA to
linear equations all require a full memory \cite{WalkerNi, ConvLinear,
ConvDIIS}. This indicates that even for affine mappings, one may not be
able to avoid our analysis based on non-expansiveness or contractivity.
And the HB example here further exemplifies the flexibility of choosing
norms for verifying these properties, and restriction to the $\ell_2$-norm
is unnecessary.

We also remark that similar analysis may be conducted for general
strongly convex and strongly smooth objective $F$, and a convex set
constraint may be included by adding a projection step on top of HB. But
here we restrict to the above toy case for succinctness, and we leave
the more general scenarios for future work.

\subsection{Numerical experiments}
We are now ready to illustrate the performance of the Anderson
Acceleration algorithms with the example problems and (unaccelerated)
algorithms above. All the experiments are run using Matlab 2014a on a
system with two 1.7 GHz cores and 8 GB of RAM, running macOS X El
Capitan.

For each experiment, we show the convergence curves of one
representative run against clock time (seconds) and iteration numbers,
respectively. The instances we show below are slightly biased towards
more difficult ones to better exemplify the improvement of AA-I-S-m
(Algorithm \ref{alg:AA-I-safe}) over the original AA-I-m (Algorithm
\ref{alg:AA-I}) .  However, in fact our modified algorithm outperforms
the original AA-I-m in more than $80\%$ of the tests we tried, and is at
least as good as AA-I-m in almost all cases, both in terms of iteration
numbers and clock time.

The codes for the experiments, including some further comparisons with
other algorithms (\eg, AA-II and its regularized version \cite{RNA},
which are also beaten by our algorithm in most cases, but we only
present results focusing on the comparison within the AA-I algorithms)
can be found in \url{https://github.com/cvxgrp/nonexp_global_aa1}. The random
seeds are all set to $456$, \ie, the one used for producing the plots in
this paper for reproducibility. Code in other languages, including
Python and Julia, is being developed and will soon be posted.

\subsubsection{Implementation details}\label{details}
Before we move on to the numerical results, we first describe in more
details the implementation tricks for better efficiency.
\paragraph{Matrix-free updates.} In line 11 of Algorithm
\ref{alg:AA-I-safe}, instead of computing and storing $H_k$, we actually
first compute
$d_k=H_{k-1}g_k+\frac{(s_{k-1}-H_{k-1}\tilde{y}_{k-1})\hat{s}_{k-1}^T
H_{k-1}g_k}{\hat{s}_{k-1}^TH_{k-1}\tilde{y}_{k-1}}$,
and then update $\tilde{x}^{k+1}=x^k-d_k$. This leads to a much more
efficient matrix-free implementation. Another small trick we use is to
normalize the $\hat{s}_k$ vectors, store them, and keep them transposed
to save the computational overhead.

\paragraph{Termination criteria.} In all our experiments, we simply
terminate the experiment when either the iteration number reaches a
pre-specified maximum $K_{\max}$, or the relative residual norm
$\|g_k\|_2/\|g_0\|_2$ is smaller than some tolerance \textit{tol}.
Accordingly, the residual norms in the plots are all rescaled by
dividing $\|g_0\|_2$, so all of them starts with $1$ in iteration $0$.
The initial residual norm $\|g_0\|_2$ is shown in the title as
\texttt{res0}. Unless otherwise specified (\eg, ISTA for elastic net
regression), we always choose $K_{\max}=1000$ and \textit{tol}
$=10^{-5}$. We remark that although not shown in the plots, the residual
norms actually continue to decrease as iterations proceed in all the
examples below.

\paragraph{Choice of hyper-parameters.} Throughout the experiments, we
use a single set of hyper-parameters to show the robustness of our
algorithm (Algorithm \ref{alg:AA-I-safe}). We choose $\theta=0.01$,
$\tau=0.001$, $D=10^6$, $\epsilon=10^{-6}$, and memory $m=5$ (apart from
the memory effect experiment on VI, in which we vary the memory sizes to
see the performance change against memories). We choose a small
averaging weight $\alpha=0.1$ to make better use of the fact that most
vanilla algorithms already correspond to averaged $f$.

\paragraph{Additional rules-of-thumb.} In our algorithm, in general by
setting a relatively small $D$ and large $\epsilon$, one enforces the
modified algorithm to use safe-guarding steps more often, making it
closer to the original AA-I-m. This may be wanted in case the problems
are relatively easy and safe-guard checking is a slight waste of time.
The Powell regularization parameter should not be set too large, as it
will empirically break down the acceleration effect. For the re-start
checking parameter $\tau$, a choice ranging from $0.001$ to $0.1$ are
all found reasonable in our experiments. A large $\tau$ will force the
algorithm to re-start quite often, making it close to choosing the
memory size $m=1$. A memory size ranging from 2 to 50 are all found to
be reasonable choices, with larger memories leading to more stable
acceleration with slightly larger per-iteration costs. However, when the
memory size becomes too large, especially when it is close to the
variable dimension, our algorithm (as well as the original AA-I-m) will
again become unstable.

In addition, AA algorithms are in general relatively more sensitive to
scaling than the vanilla algorithms. For most of the random instances we
show below, the scaling is unnecessary as expected. However, even for
the synthetic but structural UCI Madelon dataset used in the regularized
logistic regression example below, the AA algorithms will fail if we do
not divide $m$ in the objective. Similar issues occur when we come to
the heavy ball example below with an ill-conditioned linear system.
Hence in practice, the problem data need to be scaled. For examples of
pre-scaling and pre-conditioning, see \cite{SCS}.

\subsubsection{Problem instances}
We consider the following specific problem instances for the algorithms 
listed in Section \ref{prob_alg}, ranging from statistics, control to game 
theory and so on. For each plot, AA-I-m is labeled as \emph{aa1}, AA-I-S-m 
is labeled as \emph{aa1-safe}, and the original (vanilla) algorithm is labeled 
as \emph{origin}. The residual norms are computed in the $\ell_2$-norm, \ie, 
the vertical axis in the plots is $\|g_k\|_2$. In the title of the 
``residual norm versus time'' figures, ``time ratio'' indicates the average 
time per iteration of the specified algorithm divided by that of the vanilla 
algorithm. The average is computed for the single run shown in the figure among
 all the iterations up to $K_{\max}$.

\paragraph{GD: Regularized logistic regression.} We consider the
following regularized logistic regression (Reg-Log) problem:
\BEQ\label{logreg}
\begin{array}{ll}
\mbox{minimize}& \dfrac{1}{m}\sum_{i=1}^m\log(1+y_i\theta^Tx_i)+
\dfrac{\lambda}{2}\|\theta\|_2^2,
\end{array}
\EEQ
where $y_i=\pm1$ are the labels, and $x_i\in\reals^n$ are the features 
and attributes. The minimization is over $\theta\in \reals^n$. We use 
UCI Madelon dataset, which contains $2000$ samples (\ie, $m=2000$) and 
$500$ features (\ie, $n=500$) . We choose $\lambda=0.01$, and initialize 
$x^0$ with independent normally distributed entries, \ie, using \texttt{randn.m}.
 To avoid numerical overflow, we normalize $x^0$ to have a $\ell_2$-norm equal 
 to $0.001$. The step size $\alpha$ is chosen as $2/(L+\lambda)$, 
 where $L=\|X\|_2^2/4m$ is an upper bound on the largest eigenvalues 
 of the objective Hessians \cite{RNA}, and $X=[x_1,\dots,x_m]$. 
 The results are shown in Figure \ref{gd_logreg}. 
\begin{figure}[h]
\centering
\includegraphics[trim={1cm 6cm 1cm 6cm},clip, width=7.5cm]{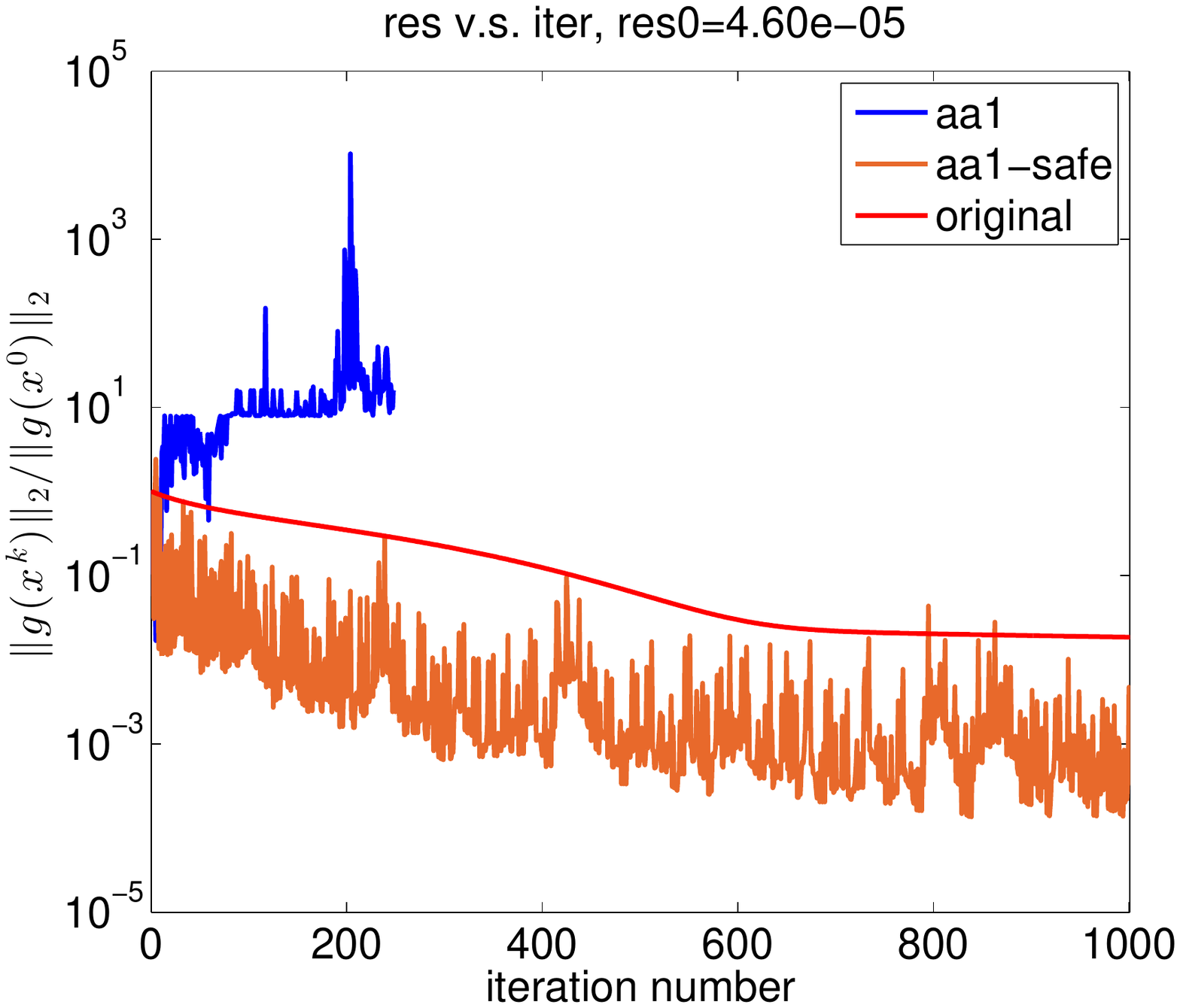}
\includegraphics[trim={1cm 6cm 1cm 6cm},clip,width=7.5cm]{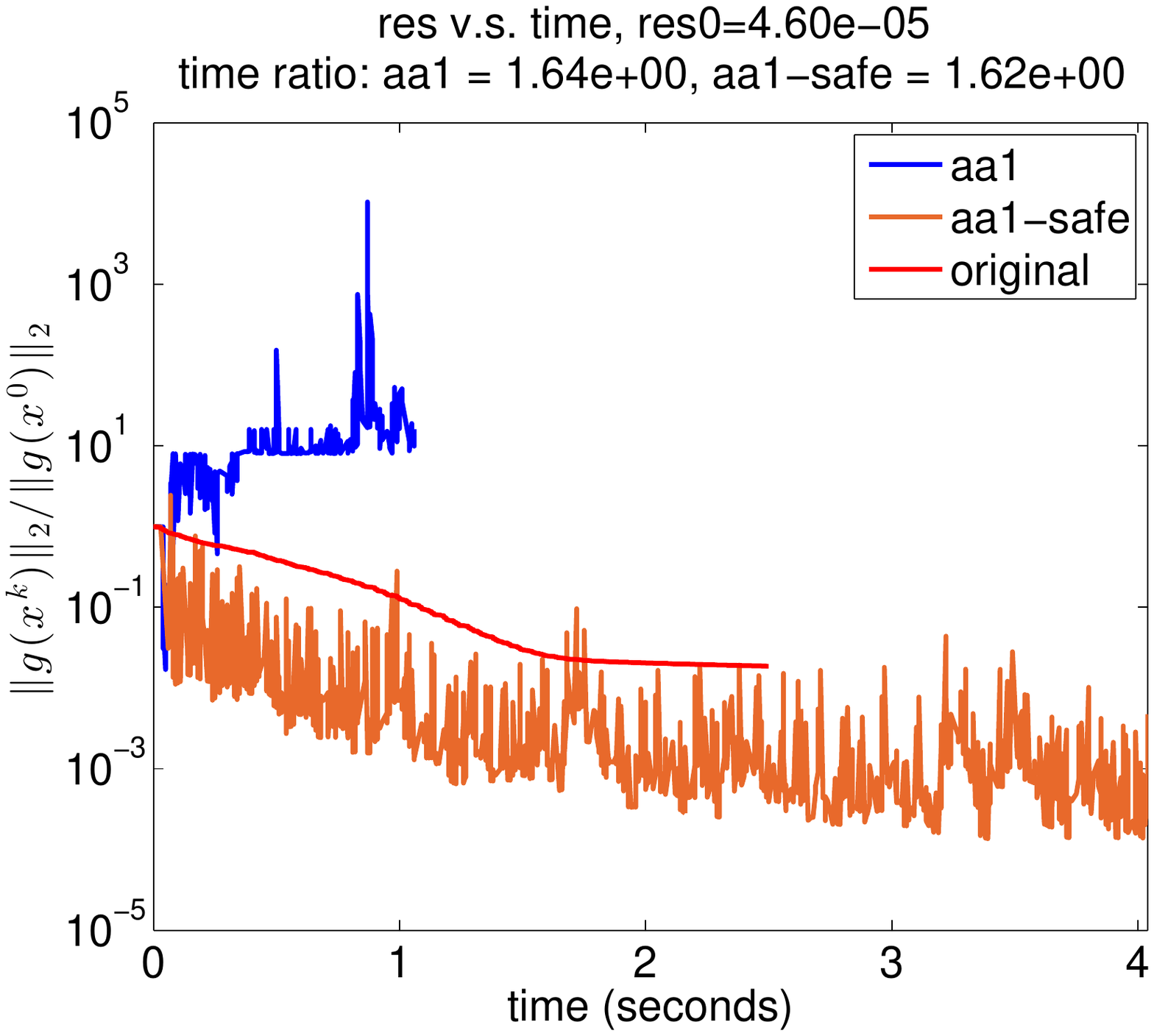}
\caption{GD: Reg-Log. Left: residual norm versus iteration. 
Right: residual norm versus time (seconds).}
\label{gd_logreg}
\end{figure}

In this example, the original AA-I-m completely fails, and our modified
AA-I-S-m obtains a 100x-1000x improvement over the original gradient
descent algorithm in terms of the residual norms. Interestingly,
although the residual norms of AA-I-S-m oscillate above the vanilla
algorithm at several points, the improvement in terms of objective
values is much more stable, as shown in Figure \ref{gd_logreg_obj}. Here
the vertical axis is the objective value minus the smallest objective
value found among all three algorithms. The objective value of the
original AA-I-m is mostly not even plotted as plugging its corresponding
iterates into (\ref{logreg}) yields $\infty$.
\begin{figure}[h]
\centering
\includegraphics[trim={1cm 6cm 1cm 6cm},clip, width=7.5cm]{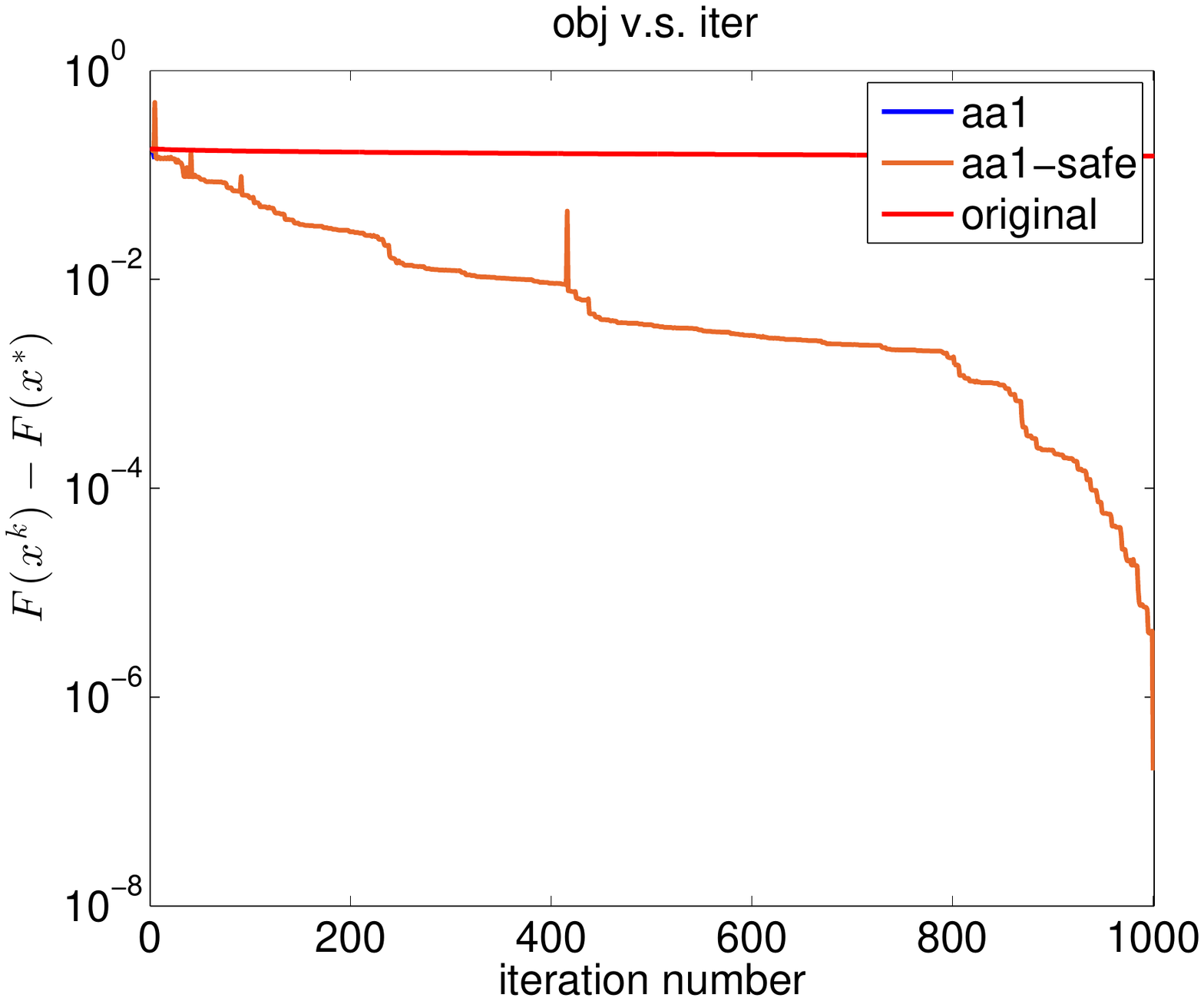}
\includegraphics[trim={1cm 6cm 1cm 6cm},clip,width=7.5cm]{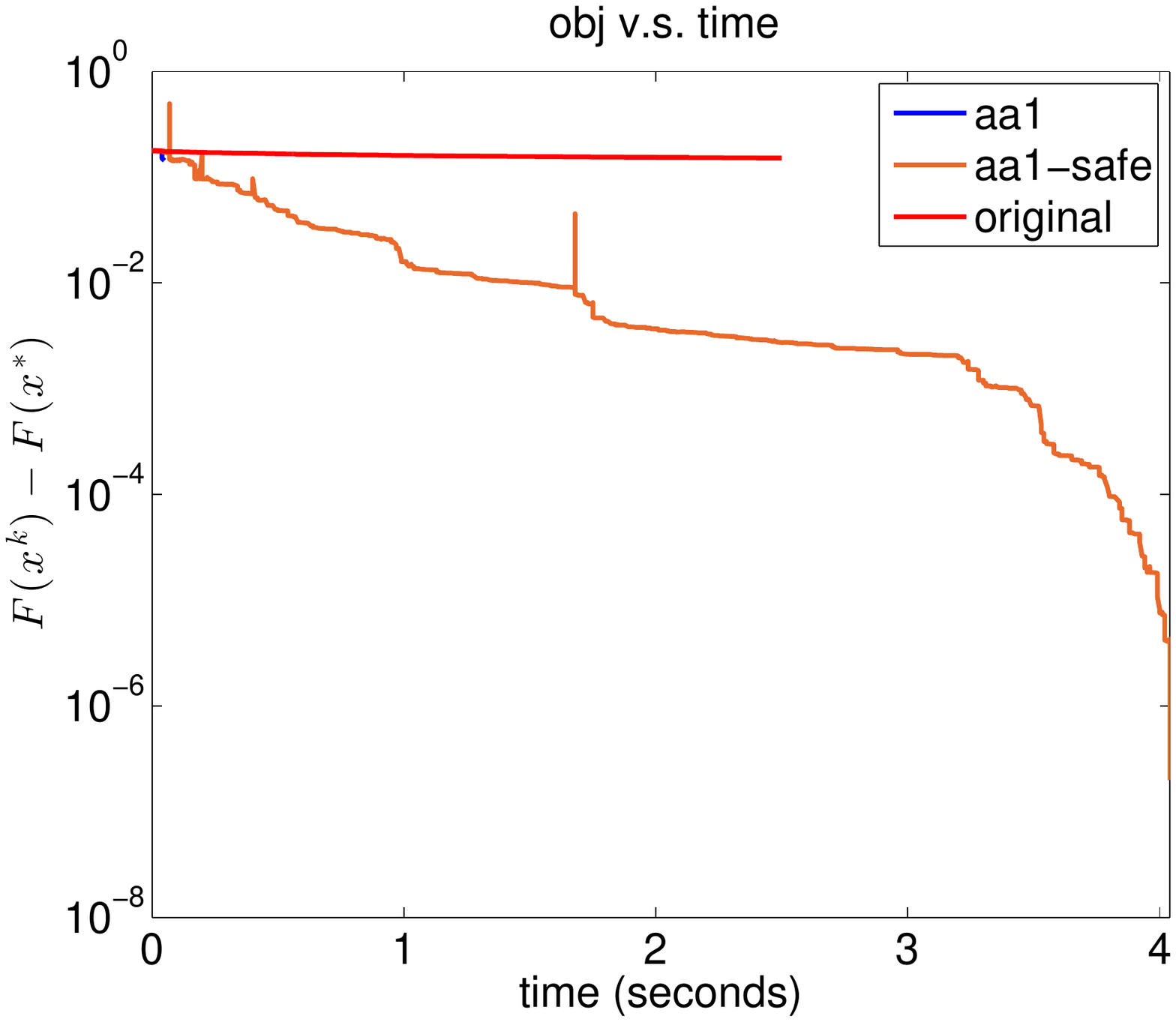}
\caption{GD: Reg-Log. Left: objective value versus iteration. 
Right: objective value versus time (seconds).}
\label{gd_logreg_obj}
\end{figure}

\paragraph{HB: Linear system.}  As described in \S \ref{contract_norm},
we consider the simple problem of solving the nonsingular linear system
$Ax+b=0$, where $A\in \reals^{n\times n}$ is positive definite and $b$
is a constant vector. We generate $A=B^TB+0.005I$, where $I$ is the
$n$-by-$n$ identity matrix, and $B\in \reals^{\lfloor n/2\rfloor \times
n}$ is generated by \texttt{randn.m}. The vector $b$ is also generated
with \texttt{randn.m}. We choose $n=1000$ in our experiments. To compute
the step sizes $\alpha$ and $\beta$, we choose $\mu$ as $0.005$ and
$L=\|A\|_F$, which avoids the expensive eigenvalue decomposition.

Notice that here we deliberately choose $B$ to be a ``fat'' matrix so
that $A$ is ill-conditioned. In our example, the condition number is
$\textbf{cond}(A)\approx 6.4629\times 10^5$. And with $\kappa=L/\mu\geq
\textbf{cond}(A)$, the convergence of the vanilla HB algorithm will be
rather slow, as can be seen from the theoretical convergence rate
$(\sqrt{\kappa}-1)/(\sqrt{\kappa}+1)$ (which is super close to $1$).

To remedy this, we adopt a simple diagonal scaling strategy that scales
$A$ and $b$ by the row and column absolute value sums of
$A=(a_{ij})_{n\times n}$. More explicitly, we compute $\hat{A}=D^{-1}A$
and $\tilde{b}=D^{-1}b$, where
\[
D=\textbf{diag}\left(\sum_{j=1}^n|a_{1j}|,\dots,\sum_{j=1}^n|a_{nj}|\right).
\]  We then further right diagonalize $\hat{A}=(\hat{a}_{ij})_{n\times
n}$ as $\tilde{A}=\hat{A}E^{-1}$, where
\[
E=\textbf{diag}\left(\sum_{i=1}^n|\hat{a}_{i1}|,\dots,\sum_{i=1}^n|\hat{a}_{in}|\right).
\]  Essentially, this is exactly performing one step of Sinkhorn-Knopp
algorithm to the absolute value matrix $|A|=(|a_{ij}|)_{n\times n}$ of
$A$ for matrix equilibration \cite{Knight-SK}. Obviously, we see that
$\tilde{x}$ is the solution to $\tilde{A}\tilde{x}+\tilde{b}=0$ if and
only if $x=E^{-1}\tilde{x}$ is the solution to $Ax+b=0$. The results are
shown in Figure \ref{hb_qp_precond}, from which we again see the
anticipated improvement.
\begin{figure}[h]
\centering
\includegraphics[trim={1cm 6cm 1cm 6cm},clip, width=7.5cm]
{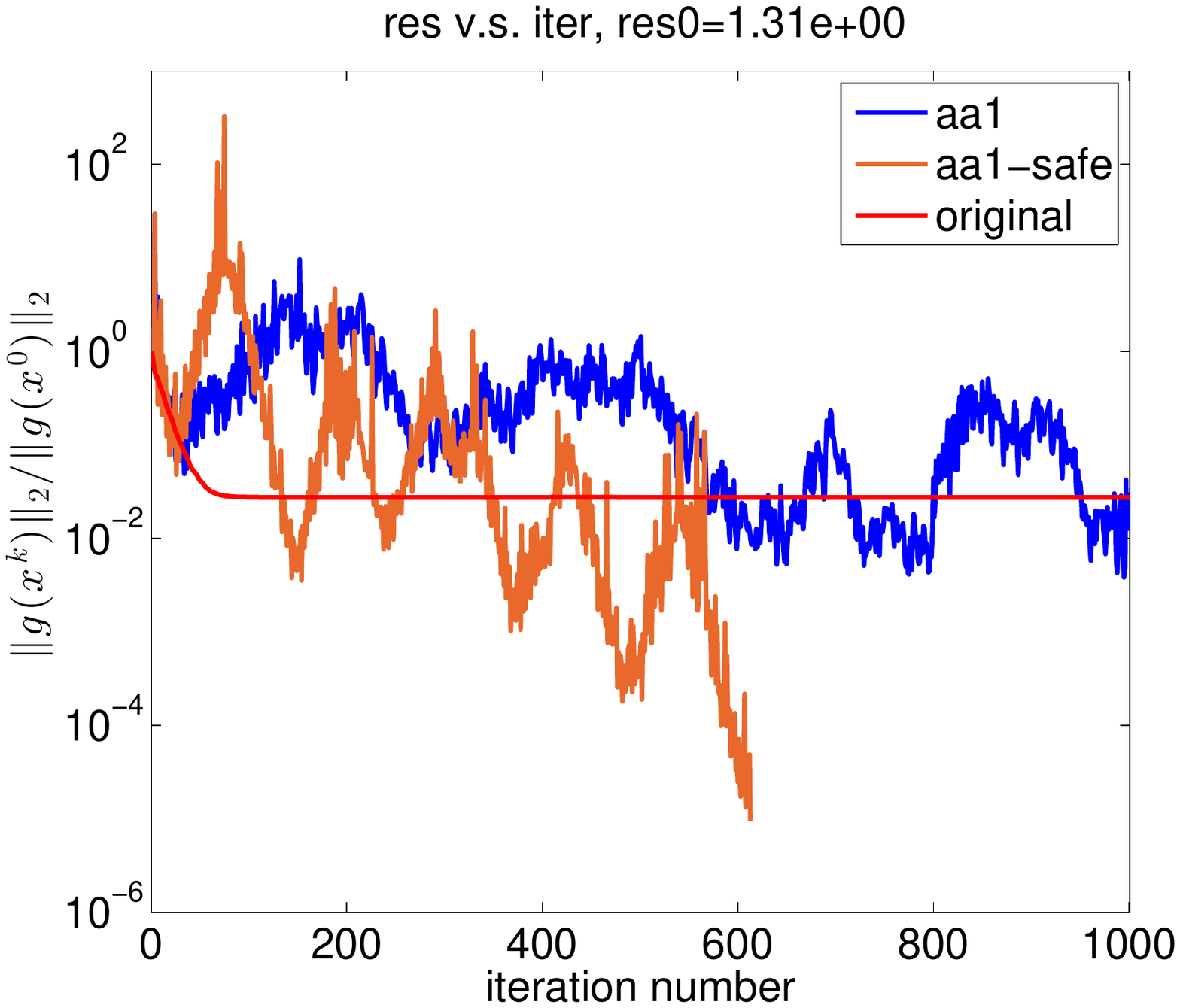}
\includegraphics[trim={1cm 6cm 1cm 6cm},clip,width=7.5cm]
{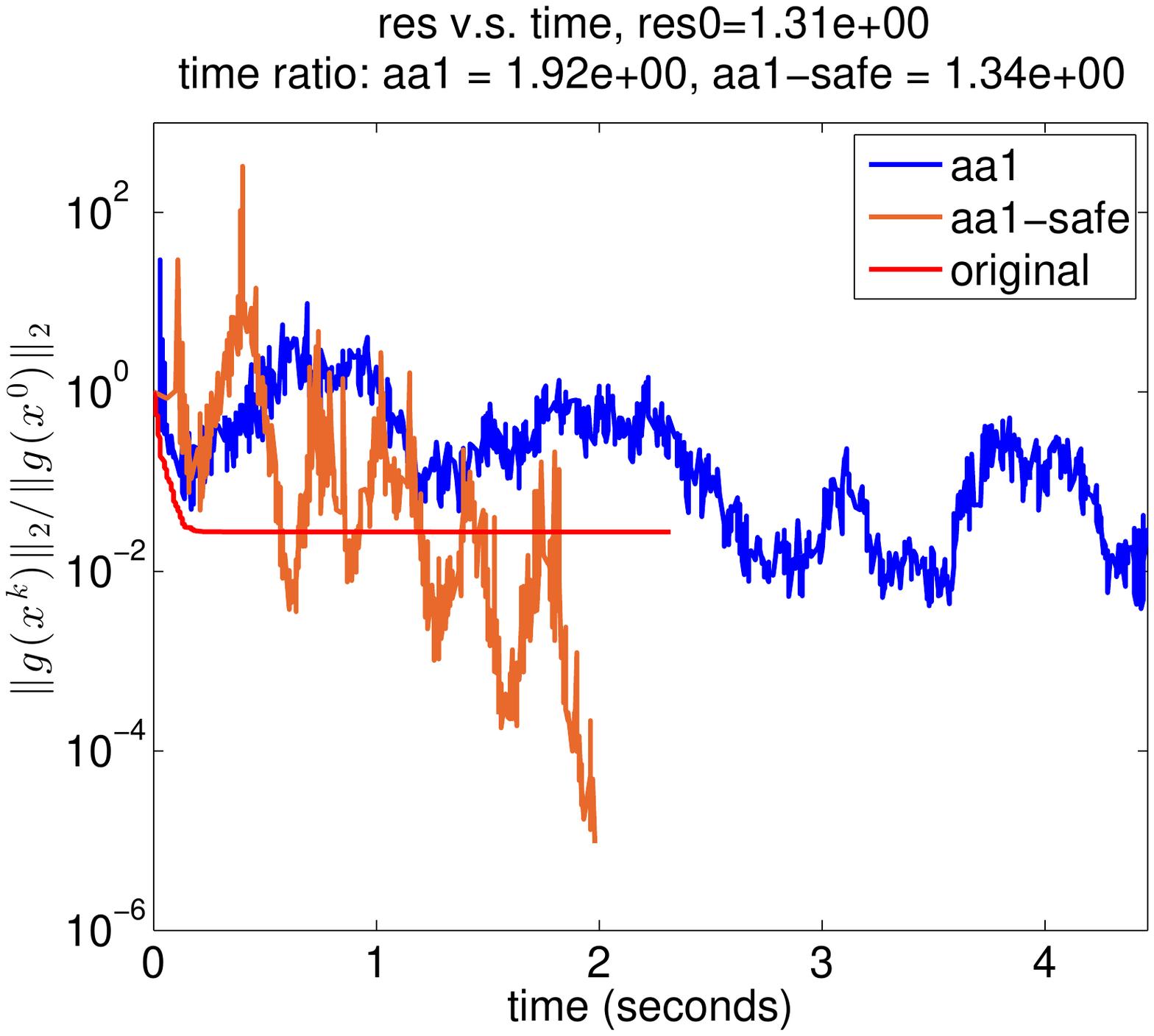}\\
\caption{HB: linear system. Left: residual norm 
versus iteration. Right: residual norm versus time (seconds).}
\label{hb_qp_precond}
\end{figure}

\paragraph{AP: Linear program.}  We consider solving the following
linear program (LP),
\BEQ\label{LP_AP}
\begin{array}{ll}
\mbox{minimize} & c^Tx\\
\mbox{subject to} & Ax=b, \quad x\in\mathcal{K},
\end{array}
\EEQ where $A\in\reals^{m\times n}$, $b\in\reals^m$, $c\in\reals^n$, and
$\mathcal{K}$ is a nonempty, closed and convex cone.  Notice that here
we deliberate choose a different (dual) formulation of (\ref{co}) to
show the flexibility of our algorithm, which can be easily mounted on
top of vanilla algorithms.

As in SCS, (\ref{LP_AP}) can be similarly formulated as the self-dual
homogeneous embedding (SDHE) system (\ref{sdhe}), but now with
\[ 
Q=\left[\begin{array}{ccc} 0 & -A^T & c\\ A & 0 & -b\\ -c^T &b^T & 0
\end{array}\right],\quad \mathcal{C}=\mathcal{K}\times\reals^m\times \reals_+.
\]

Under the notations of AP, solving the SDHE system above reduces to
finding a point in the intersection of $C$ and $D$, with
$C=\{(u,v)\;|\;Qu=v\}$ and $D=\mathcal{C}\times\mathcal{C}^*$, which can
then be solved by AP.

We generate a set of random data ensuring primal and dual feasibility of
the original problem (\ref{LP_AP}), following \cite{SCS}. More
specifically, we first generate $A$ as a sparse random matrix with
sparsity $0.1$ using \texttt{sprandn.m}. We then generate $z^\star$ with
\texttt{randn.m}, and take $x^\star=\max(z^\star,0)$,
$s^\star=\max(-z^\star,0)$ where the maximum is taken component-wisely.
We then also generate $y^\star$ with \texttt{randn.m}, and take
$b=Ax^\star$, $c=A^Ty^\star+s^\star$. In our experiments, we set $m=500$
and $n=1000$, and $x^0$ is simply initialized using \texttt{randn.m} and
then normalized to have a unit $\ell_2$-norm.

In addition, as in the HB example above and SCS \cite{SCS}, we perform
diagonal scaling on the problem data. More explicitly, we compute
$\tilde{A}=D^{-1}AE^{-1}$ exactly as in the HB example, and accordingly
scale $b$ to be $\tilde{b}=D^{-1}b$ and $c$ to be $\tilde{c}=E^{-1}c$.
Again, we see that $\tilde{x}$ is a solution to (\ref{LP_AP}) with $A$,
$b$, $c$ replaced with the scaled problem data $\tilde{A}$, $\tilde{b}$,
$\tilde{c}$, if and only if $x=E^{-1}\tilde{x}$ is a solution to the
original problem.

The results are summarized in Figure \ref{ap-lp}.
\begin{figure}[h]
\centering
\includegraphics[trim={1cm 6cm 1cm 6cm},clip, width=7.5cm]{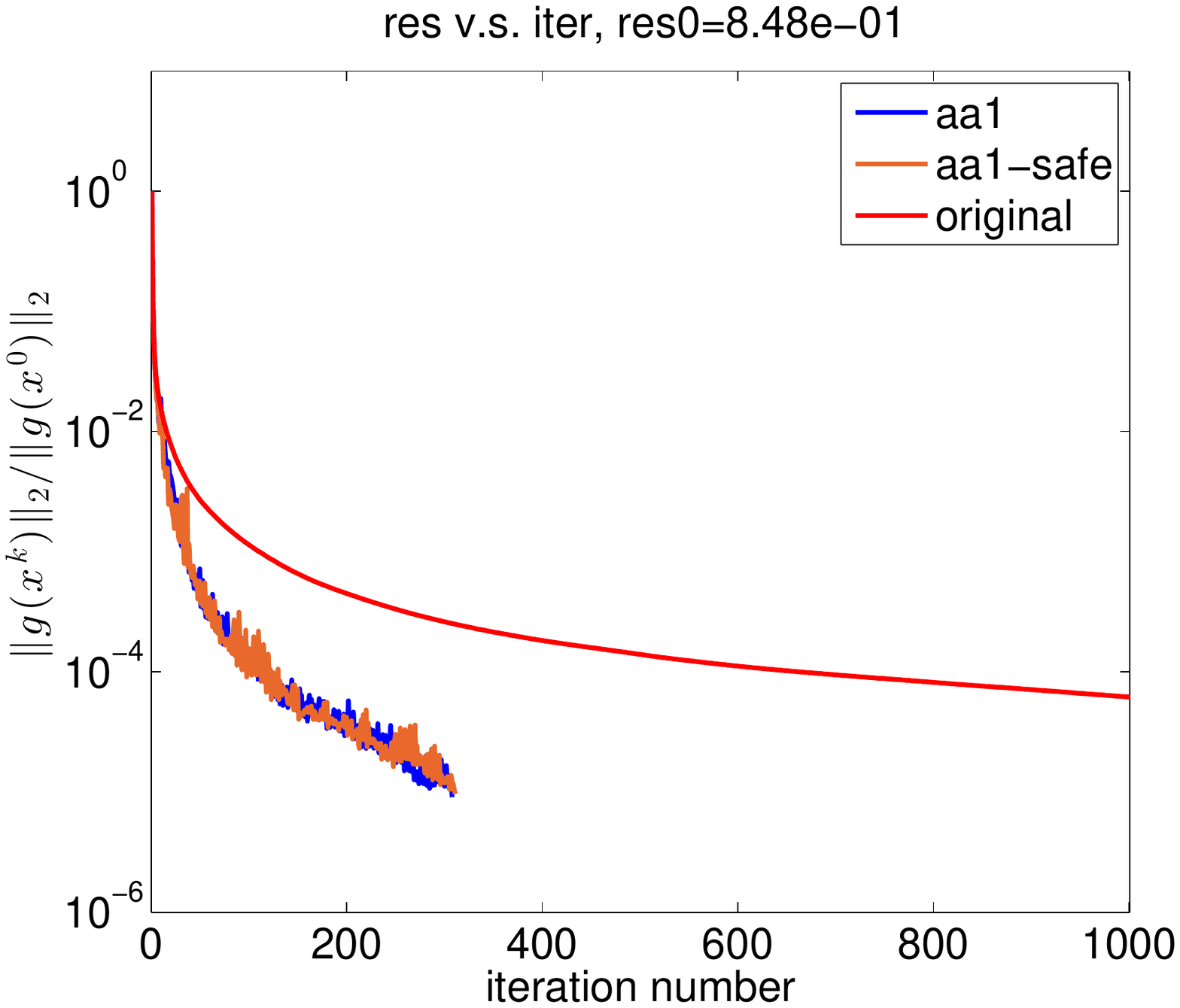}
\includegraphics[trim={1cm 6cm 1cm 6cm},clip,width=7.5cm]{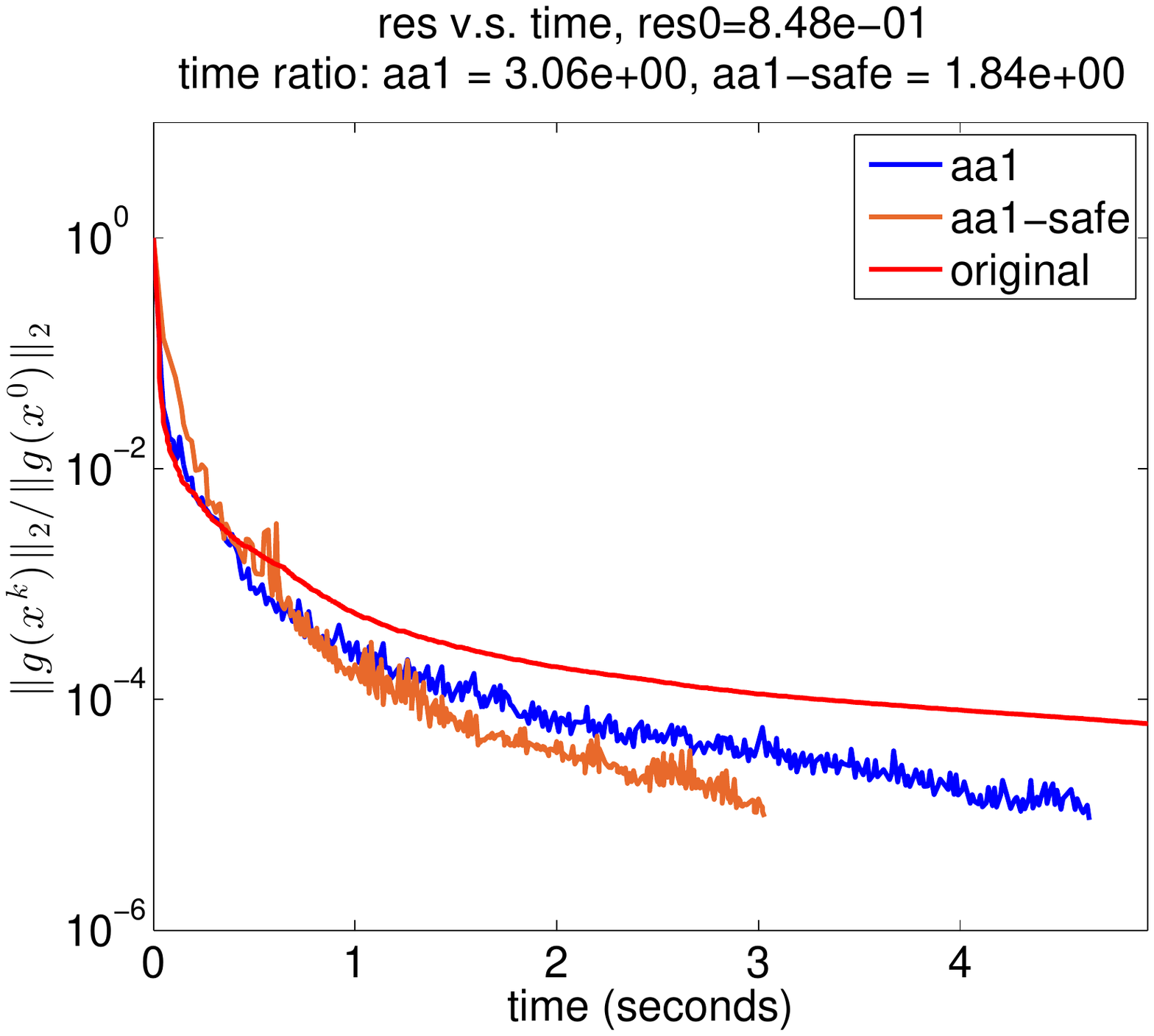}\\
\caption{AP: LP as SDHE. Left: residual norm versus iteration. 
Right: residual norm versus time (seconds).}
\label{ap-lp}
\end{figure}
We can see that our algorithm AA-I-S-m compares favorably with the
original AA-I-m in terms of iteration numbers, and both AA-I-S-m and
AA-I-m outperform the vanilla AP algorithm. In terms of running time, we
can see a further slight improvement over the original AA-I-m.

\paragraph{PGD: Non-negative least squares and convex-concave matrix
game.} We consider the following non-negative least squares (NNLS)
problem:
\BEQ
\begin{array}{ll}
\mbox{minimize}& \frac{1}{2}\|Ax-b\|_2^2\\
\mbox{subject to}& x\geq 0,
\end{array}
\EEQ where $A\in\reals^{m\times n}$ and $b\in\reals^m$.

Such a problem arises ubiquitously in various applications, especially
when $x$ has certain physical interpretation \cite{NNLS}. We consider
the more challenging high dimensional case, \ie, $m<n$ \cite{NNLSHD}.
The gradient of the objective function can be simply evaluated as
$A^TAx-A^Tb$, and hence the PGD algorithm can be efficiently
implemented.

We generate both $A$ and $b$ using \texttt{randn.m}, with $m=500$ and
$n=1000$. We again initialize $x^0$ using \texttt{randn.m} and then
normalize it to have a unit $\ell_2$-norm. The step size $\alpha$ is set to
$1.8/\|A^TA\|_2$. The results are summarized in Figure \ref{pgd-rand}.
\begin{figure}[h]
\centering
\includegraphics[trim={1cm 6cm 1cm 6cm},clip, width=7.5cm]{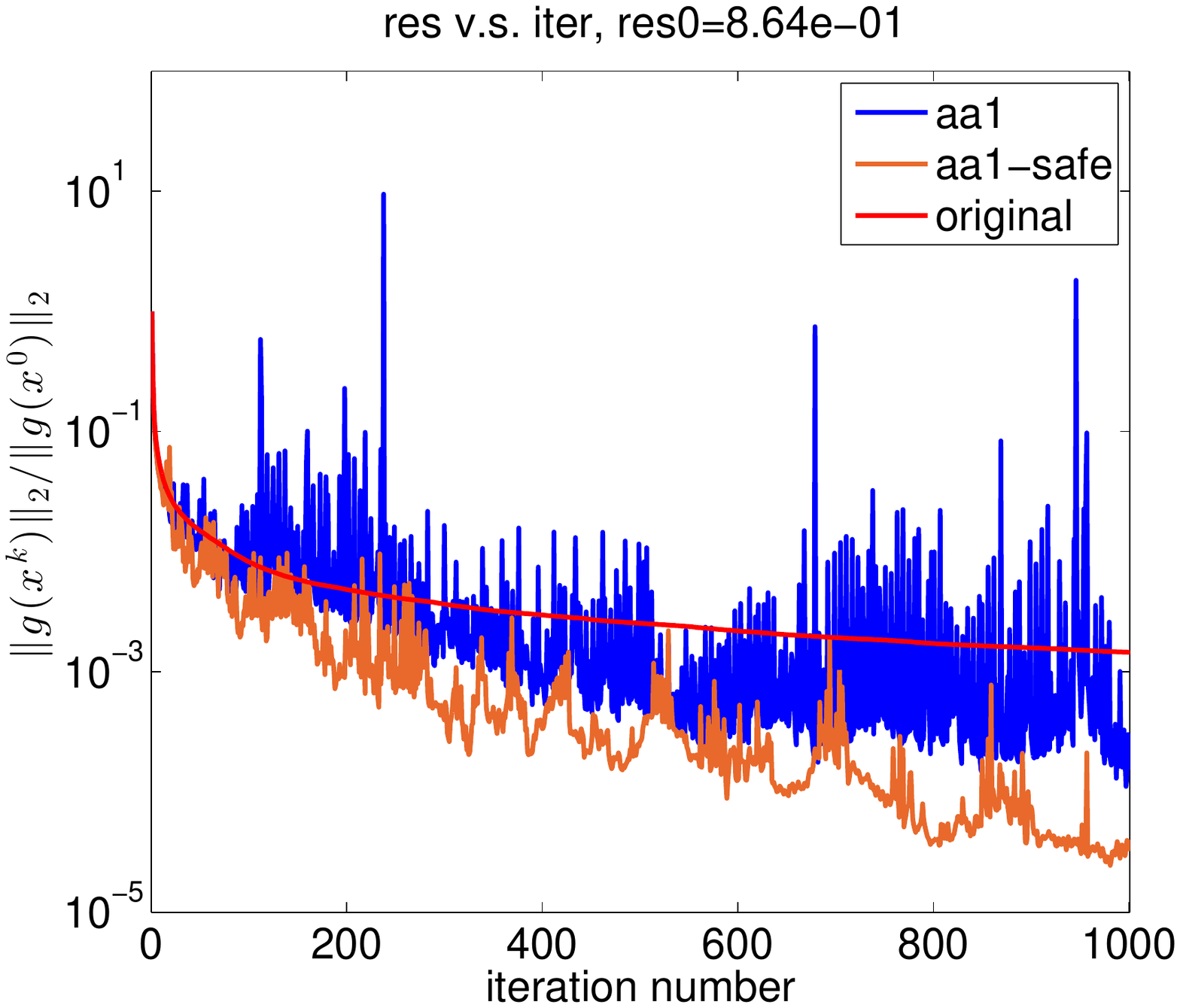}
\includegraphics[trim={1cm 6cm 1cm 6cm},clip,width=7.5cm]{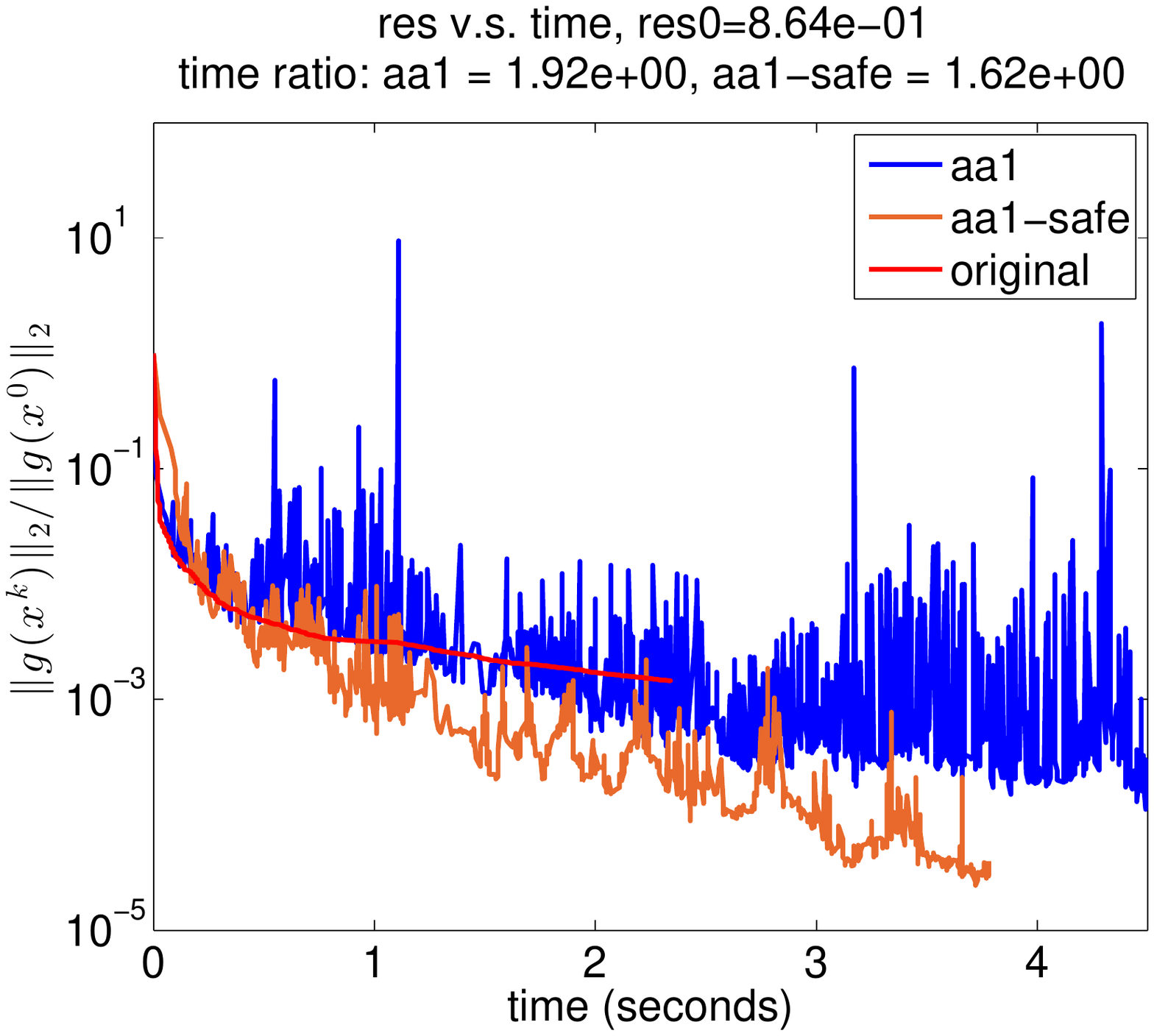}\\
\caption{PGD: NNLS. Left: residual norm versus iteration. 
Right: residual norm versus time (seconds).}
\label{pgd-rand}
\end{figure}

We also consider a more specialized and structured problem:
convex-concave matrix game (CCMG), which can be reformulated into a form
solvable by PGD, as we show below.

A CCMG can be formulated as the following LP \cite{CVXBOOK}:
\BEQ
\begin{array}{ll}
\mbox{minimize}&t\\
\mbox{subject to}&u\geq 0,\quad {\bf1} ^Tu=1,\quad P^Tu\leq t{\bf 1},
\end{array}
\EEQ where $t\in\reals$, $u\in\reals^m$ are variables, and
$P\in\reals^{m\times n}$ is the pay-off matrix. Of course we can again
reformulate it as an SDHE system and solve it by AP as above. But here
we instead consider a different reformulation amenable to PGD.

To do so, we first notice that the above LP is always feasible. This can
be seen by choosing $u$ to be an arbitrary probability vector, and
setting $t=\|P^Tu\|_{\infty}$. Hence the above LP can be further
transformed into
\BEQ
\begin{array}{ll}
\mbox{minimize}&t+\frac{1}{2}\|P^Tu+s-t{\bf1}\|_2^2\\
\mbox{subject to}&u\geq 0,\quad 1^Tu=1,\quad s\geq 0,
\end{array}
\EEQ where we introduce an additional (slack) variable $s\in\reals^n$.
Using the efficient projection algorithm onto the probability simplex
set \cite{SimplexProj, CVXBOOK}, the above problem can be solved
efficiently by PGD.

We generate $P$ using \texttt{randn.m} with $m=500$ and $n=1500$. Again,
$x^0$ is initialized using \texttt{randn.m} and then normalized to have
a unit $\ell_2$-norm. The step size $\alpha$ is set to
$1.8/\|\tilde{A}^T\tilde{A}\|_2$, where $\tilde{A}=[P^T, I, e]$, in
which $I$ is the $n$-by-$n$ identity matrix and $e\in\reals^n$ is an
all-one vector.  The results are summarized in Figure \ref{pgd-ccm}.
\begin{figure}[h]
\centering
\includegraphics[trim={1cm 6cm 1cm 6cm},clip, width=7.5cm]{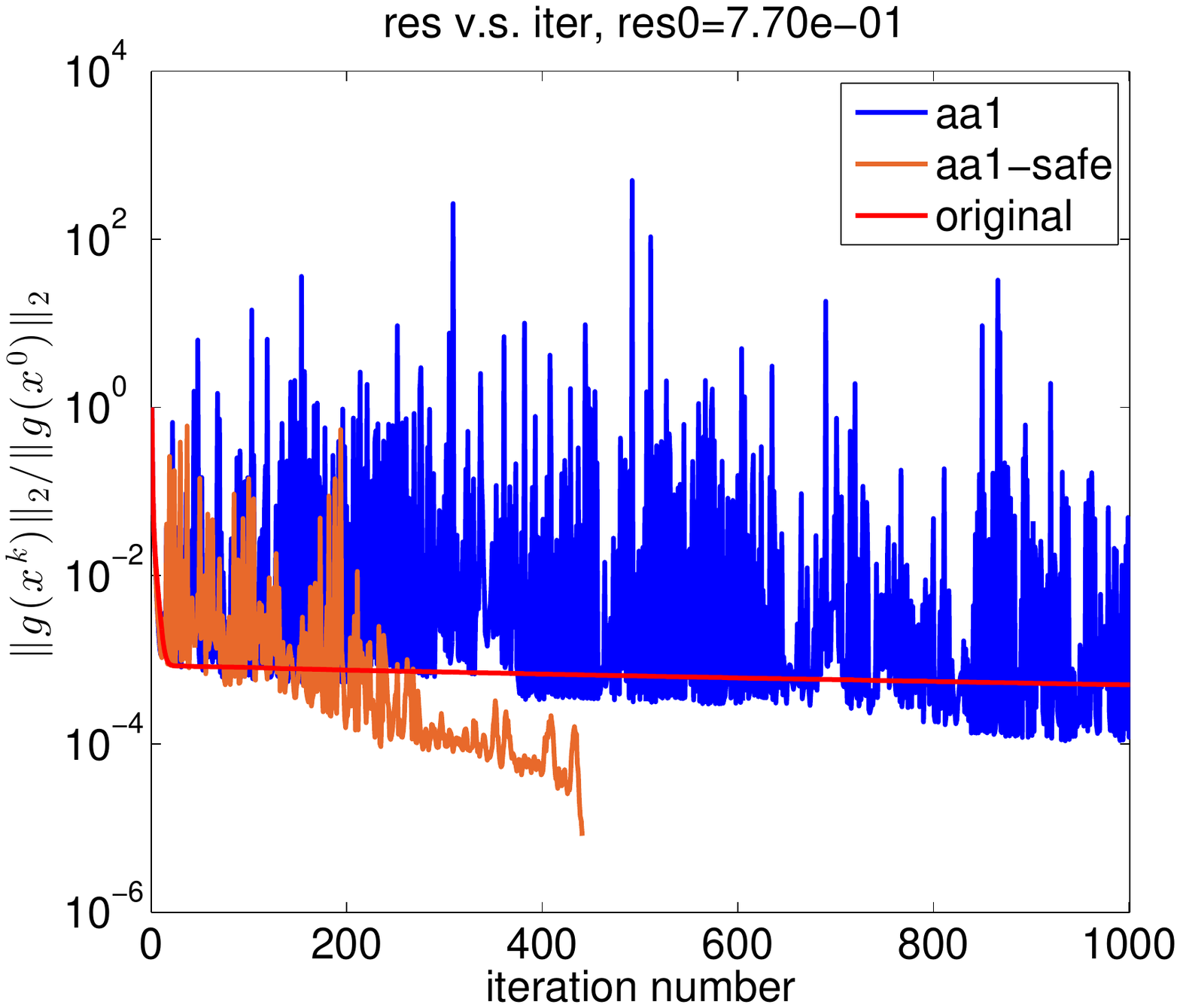}
\includegraphics[trim={1cm 6cm 1cm 6cm},clip,width=7.5cm]{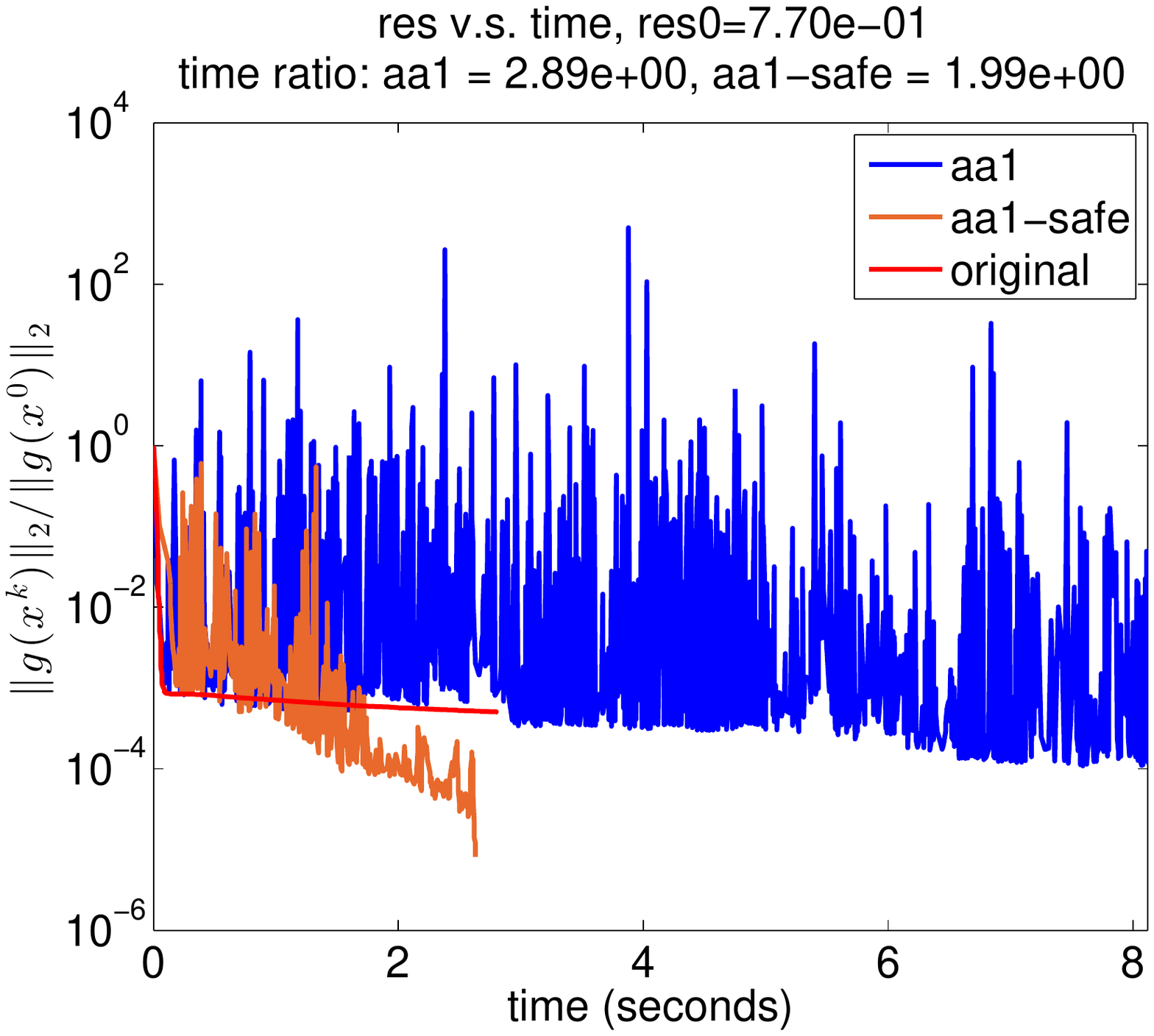}\\
\caption{PGD: CCMG. Left: residual norm versus iteration. 
Right: residual norm versus time (seconds).}
\label{pgd-ccm}
\end{figure}

\paragraph{ISTA: Elastic net regression.} We consider the following
elastic net regression (ENR) problem \cite{ENR}:
\BEQ
\begin{array}{ll}
\mbox{minimize}&\frac{1}{2}\|Ax-b\|_2^2 + \mu\left(\frac{1-\beta}{2}
\|x\|_2^2+\beta\|x\|_1\right),\\
\end{array}
\EEQ where $A\in\reals^{m\times n}$, $b\in\reals^m$. In our experiments,
we take $\beta=1/2$ and $\mu=0.001\mu_{\max}$, where
$\mu_{\max}=\|A^Tb\|_{\infty}$ is the smallest value under which the ENR
problem admits only the zero solution \cite{SCS}. ENR is proposed as a
hybrid of Lasso and ridge regression, and has been widely used in
practice, especially when one seeks both sparsity and overfitting
prevention.

Applying ISTA to ENR, we obtain the following iteration scheme:
\[
x^{k+1}=S_{\alpha\mu/2}\left(x^k-\alpha\left(A^T(Ax-b)+
\frac{\mu}{2}x\right)\right),
\] in which we choose $\alpha=1.8/L$, with
$L=\lambda_{\max}(A^TA)+\mu/2$.

We again consider a harder high dimensional case, where $m=500$ and
$n=1000$. The data is generated similar to the Lasso example in
\cite{SCS}. More specifically, we generate $A$ using \texttt{randn.m},
and then generate $\hat{x}\in\reals^n$ using \texttt{sprandn.m} with
sparsity $0.1$. We then generate $b$ as $b=A\hat{x}+0.1w$, where $w$ is
generated using \texttt{randn.m}. The initial point $x^0$ is again
generated by \texttt{randn.m} and normalized to have a unit $\ell_2$-norm.
The step size is chosen as $\alpha=1.8/L$, where $L=\|A^TA\|_2+\mu/2$.
The results are shown in Figure \ref{ista-enr}. Here we set the
tolerance \textit{tol} to $10^{-8}$ to better exemplify the performance
improvement of our algorithm in a relative long run.
\begin{figure}[h]
\centering
\includegraphics[trim={1cm 6cm 1cm 6cm},clip, width=7.5cm]{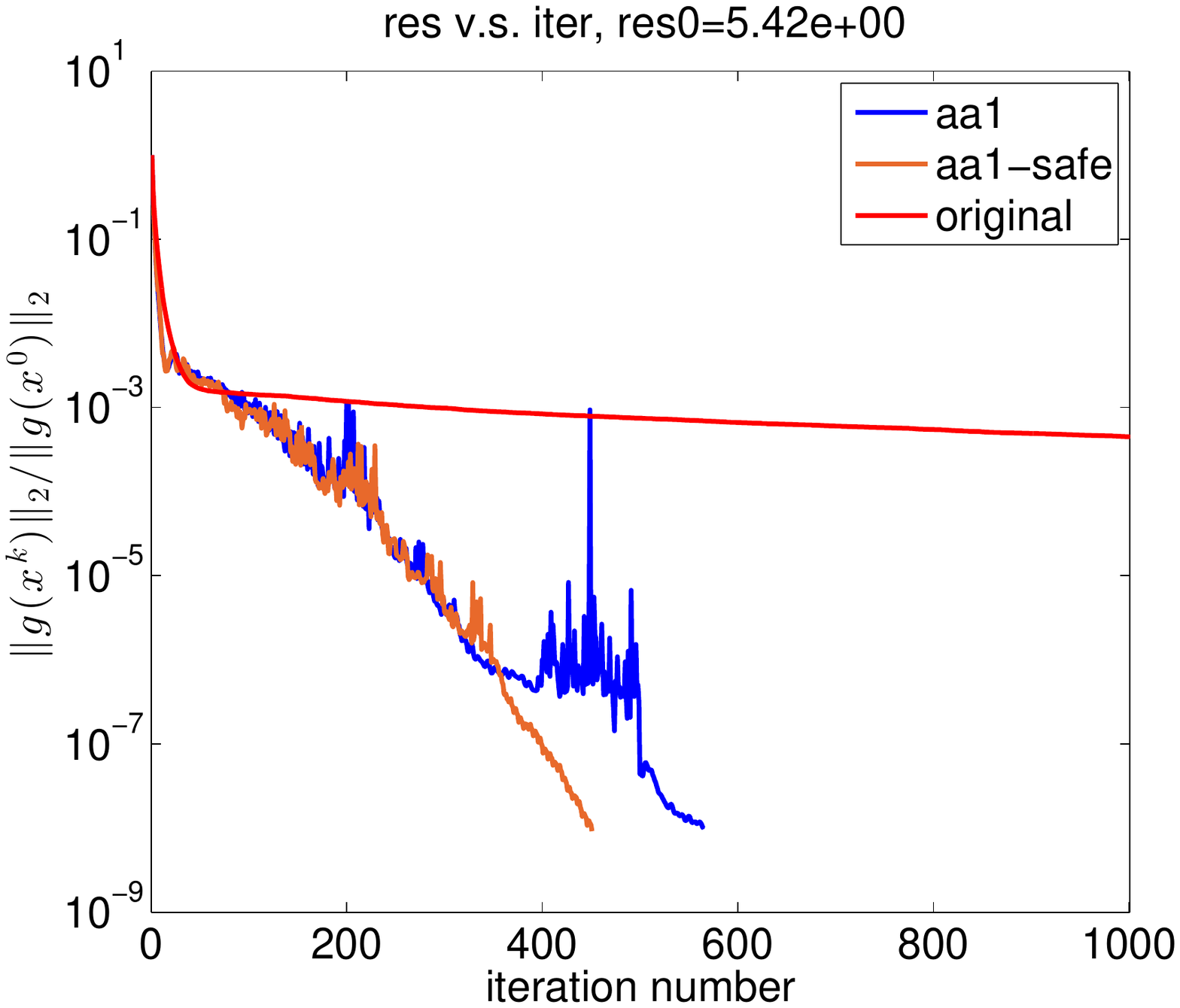}
\includegraphics[trim={1cm 6cm 1cm 6cm},clip,width=7.5cm]{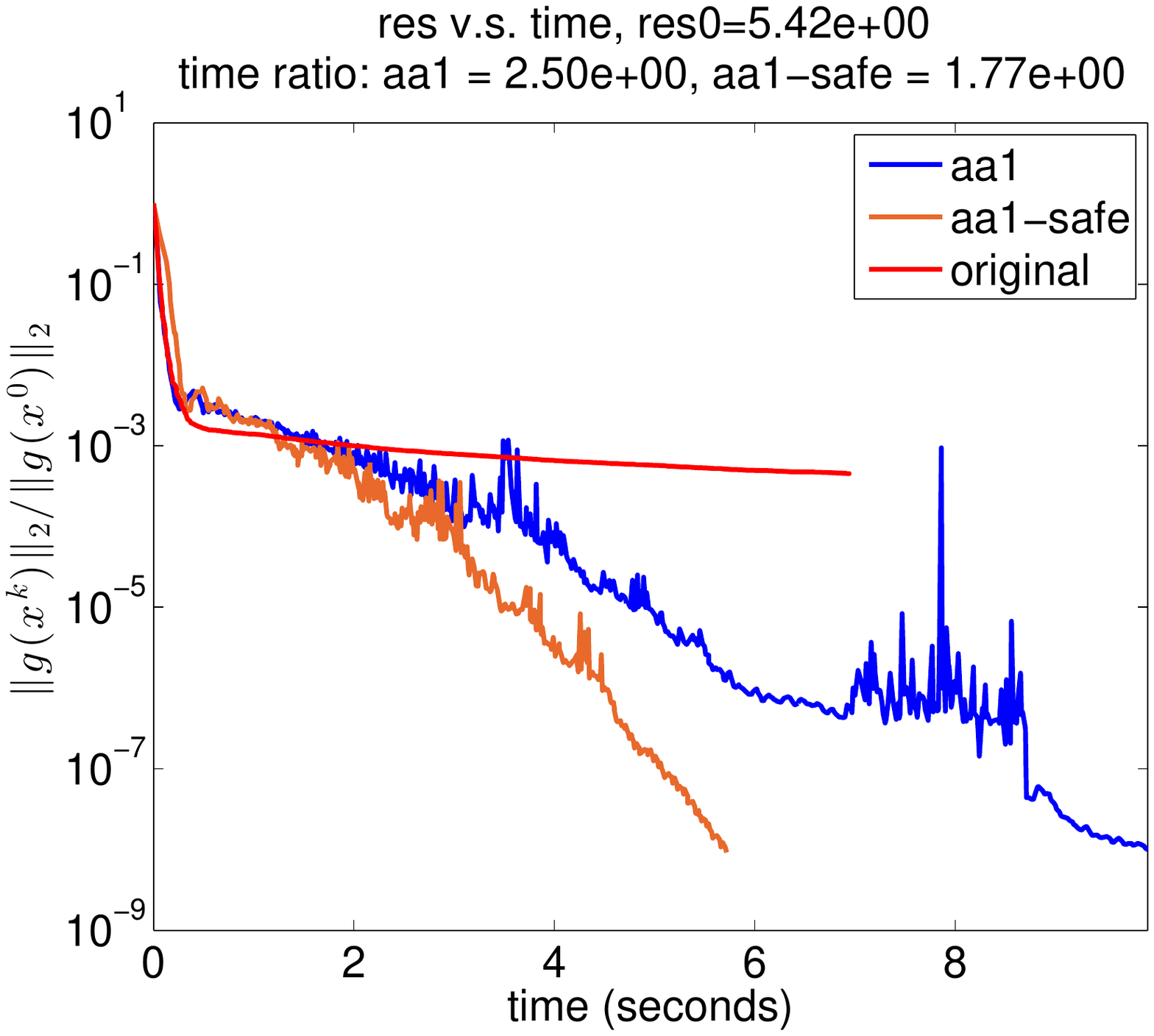}\\
\caption{ISTA: ENR. Left: residual norm versus iteration. 
Right: residual norm versus time (seconds).}
\label{ista-enr}
\end{figure}

\paragraph{CO: Facility location.} Consider the following facility
location problem \cite{FacLoc}:
\BEQ
\begin{array}{ll}
\mbox{minimize}&\sum_{i=1}^m\|x-c_i\|_2,\\
\end{array}
\EEQ where $c_i\in\reals^n$, $i=1,\dots,m$ are locations of the clients,
and the goal is to find a facility location that minimizes the total
distance to all the clients.

Applying CO to this problem with $\alpha=1$, we obtain that
(\cite{Proximal})
\[
\begin{array}{ll}
x_i^{k+1}&=\textbf{prox}_{\|\cdot\|_2}(z_i^k+c_i)-c_i\\
z_i^{k+1}&=z_i^k+2\bar{x}^{k+1}-x_i^{k+1}-\bar{z}^k, \quad i=1,\dots,m,
\end{array}
\] where $\textbf{prox}_{\|\cdot\|_2}(v)=(1-1/\|v\|_2)_+v$.

Notice that all the updates can be parallelized. In particular, in the
Matlab implementation no for loops is needed within one iteration, which
is important to the numerical efficiency.  We generate $c_i$ using
\texttt{sprandn.m}, with $m=500$ and $n=300$ and sparsity $0.01$. The
results are summarized in Figure \ref{co-facloc}. Notice that here we
again set the tolerance \textit{tol} to $10^{-8}$ to better demonstrate
the improvement of our algorithm, and we truncate the maximum iteration
number to $K_{\max}=500$ for better visualization.
\begin{figure}[h]
\centering
\includegraphics[trim={1cm 6cm 1cm 6cm},clip, width=7.5cm]{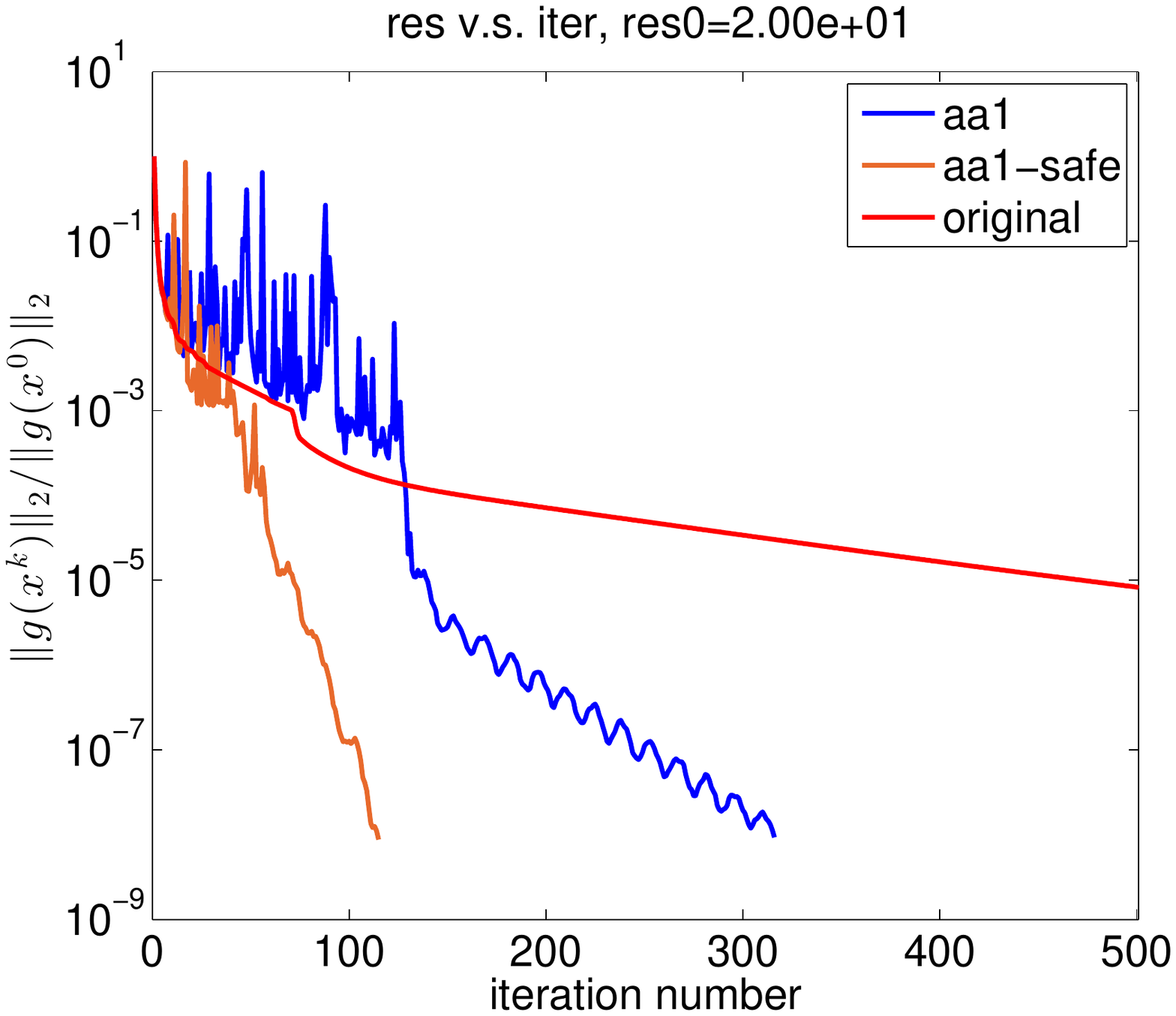}
\includegraphics[trim={1cm 6cm 1cm 6cm},clip,width=7.5cm]{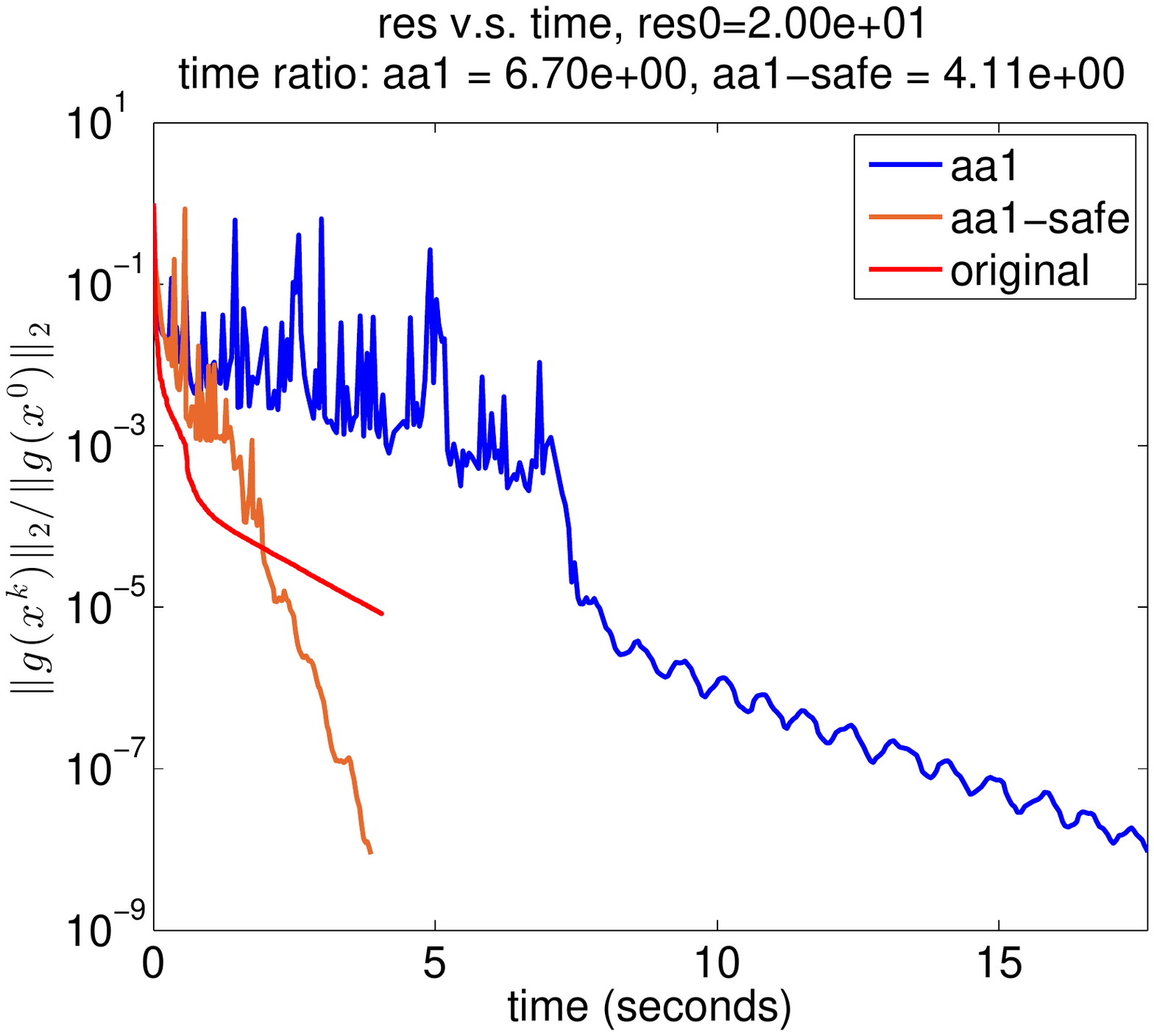}\\
\caption{CO: facility location. Left: residual norm versus iteration. 
Right: residual norm versus time (seconds).}
\label{co-facloc}
\end{figure}

We remark that in general, the $\ell_2$-norm can also be replaced with an
arbitrary $p$-norm, and more generally any function for
which the proximal operators can be easily evaluated.

\paragraph{SCS: Cone programs.} Consider (\ref{co}) with
$\mathcal{K}=\reals^m_+$ (resp.
$\mathcal{K}=\{s\in\reals^m\;|\;\|s_{1:m-1}\|_2\leq s_m\}$), \ie, a
generic LP (resp. SOCP).  We solve it using a toy implementation of SCS,
\ie, one without approximate projection, CG iterations, fine-tuned
over-relaxation and so on.

We make use of the following explicit formula for the projection onto
the second order cone
$\mathcal{K}=\{s\in\reals^m\;|\;\|s_{1:m-1}\|_2\leq s_m\}$
(\cite{Proximal}):
\[
\Pi_{\mathcal{K}}(s)=
\left\{\begin{array}{ll}
s & \text{ if } \|s_{1:n-1}\|_2\leq s_n\\
0 & \text{ if } \|s_{1:n-1}\|_2\leq -s_n\\
\frac{\|s_{1:n-1}\|_2+s_n}{2}\left[\frac{s_{1:n-1}}{\|s_{1:n-1}\|_2},
1\right]^T & \text{ otherwise.}
\end{array}\right.
\]

For both LP and SOCP, we choose $m=500$ and $n=700$, and $x^0$ is
initialized using \texttt{randn.m} and then normalized to have a unit
$\ell_2$-norm. We again follow \cite{SCS} to generate data that ensures
primal and dual feasibility of the original cone programs.

For LP, we generate $A$ as a horizontal concatenation of
\texttt{sprandn(m,$\lfloor$n/2$\rfloor$,0.1)} and identity matrix of
size $m\times \lfloor n/2\rfloor$, added with a noisy term \texttt{1e-3
* randn(m, n)}. We then generate $z^\star$ using \texttt{randn.m}, and
set $s^\star=\max(z^\star,0)$ and $y^\star=\max(-z^\star,0)$, where the
maximum is also taken component-wisely. We then also generate $x^\star$
using \texttt{randn.m}, and take $b=Ax^\star+s^\star$ and
$c=-A^Ty^\star$.

For SOCP, we similarly generate $A$ exactly the same as in LP. We then
generate $z^\star$ using \texttt{randn.m}, and set
$s^\star=\Pi_{\mathcal{K}}(z^\star)$ and $y^\star=s^\star-z^\star$,
where the maximum is also taken component-wisely. We then once again
generate $x^\star$ using \texttt{randn.m}, and take $b=Ax^\star+s^\star$
and $c=-A^Ty^\star$.

The results are summarized in Figure \ref{scs-lp} and Figure
\ref{scs-socp}.
\begin{figure}[h]
\centering
\includegraphics[trim={1cm 6cm 1cm 6cm},clip, width=7.5cm]{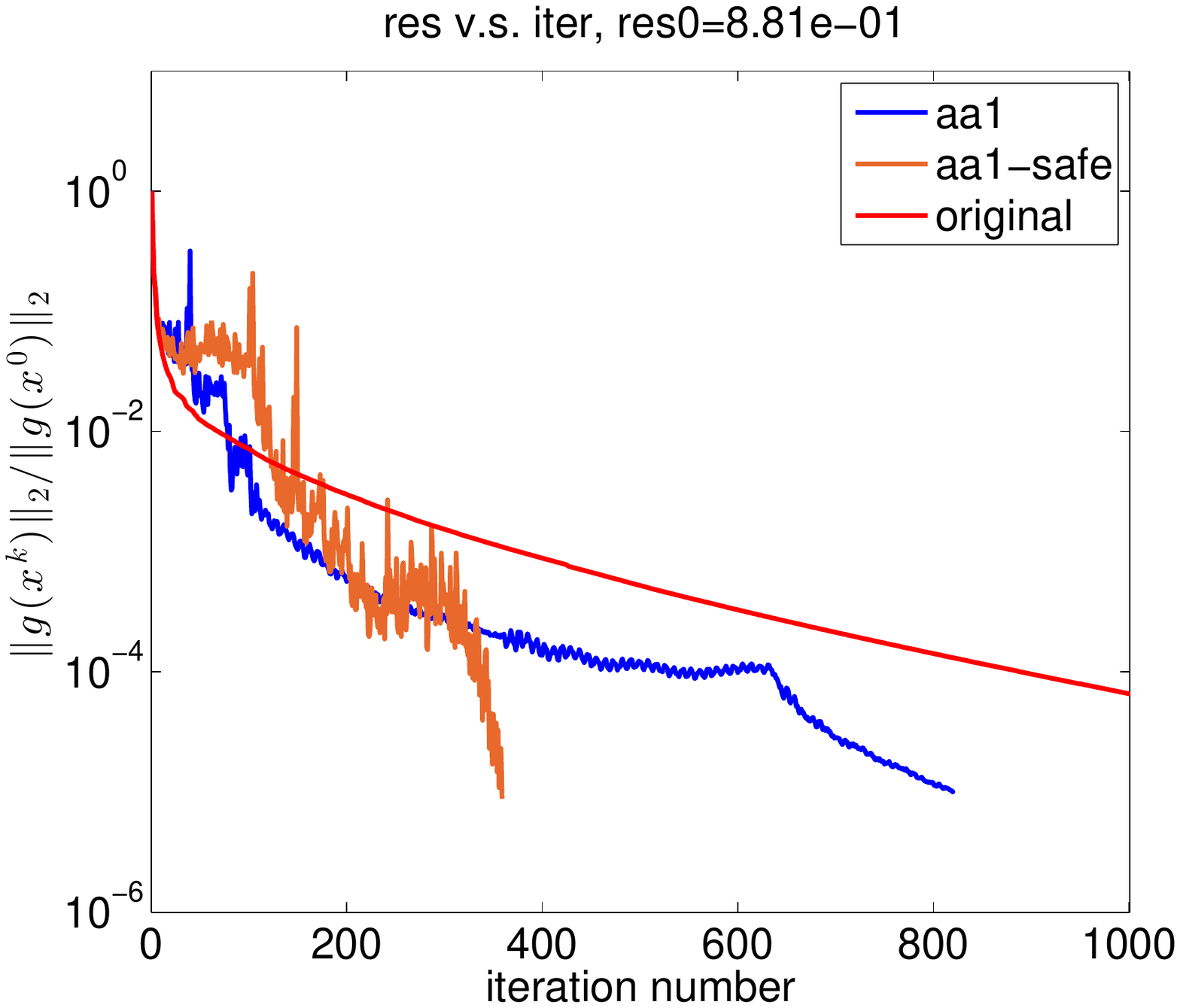}
\includegraphics[trim={1cm 6cm 1cm 6cm},clip,width=7.5cm]{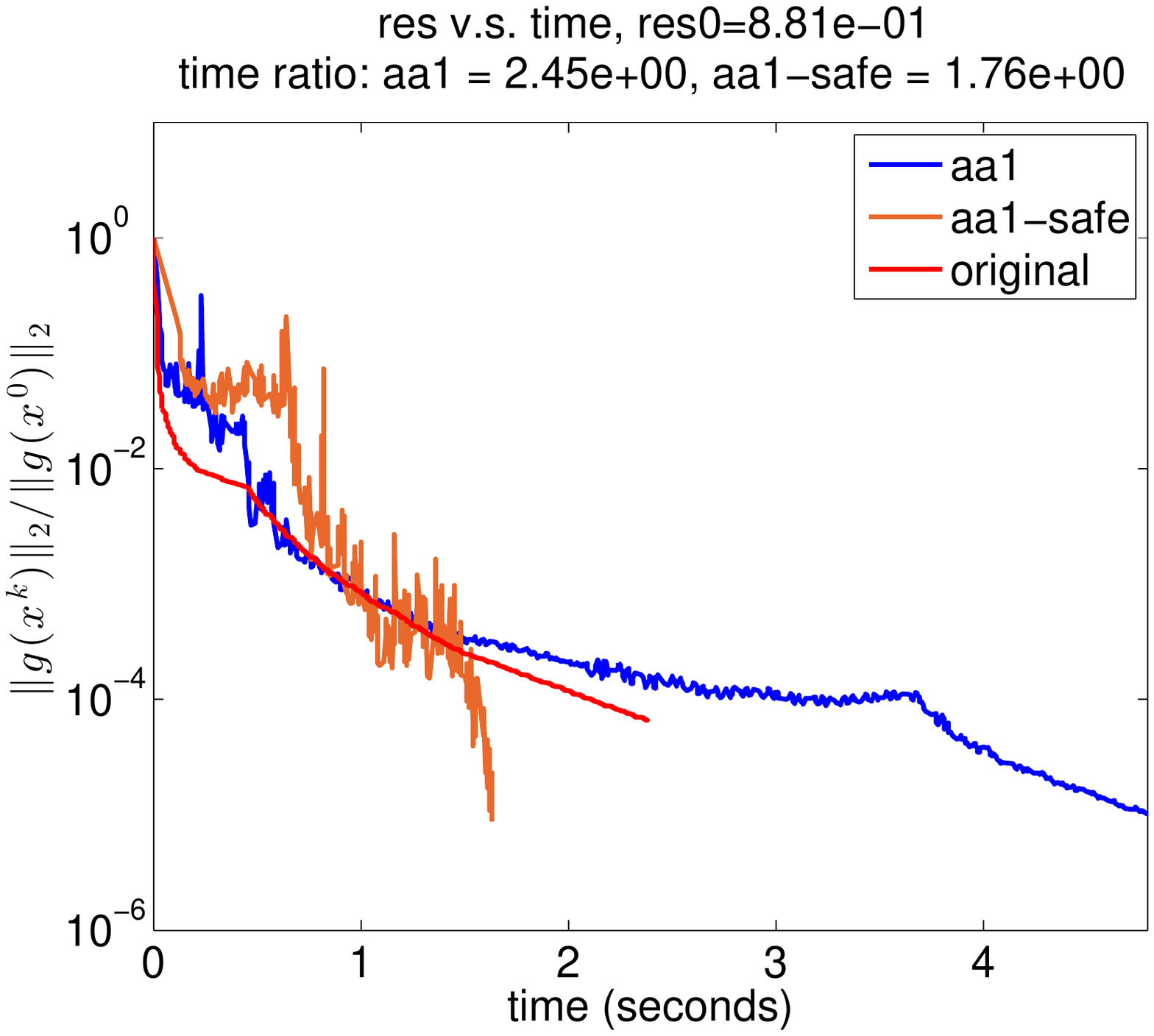}\\
\caption{SCS: LP. Left: residual norm versus iteration.
 Right: residual norm versus time (seconds).}
\label{scs-lp}
\end{figure}

\begin{figure}[h]
\centering
\includegraphics[trim={1cm 6cm 1cm 6cm},clip, width=7.5cm]{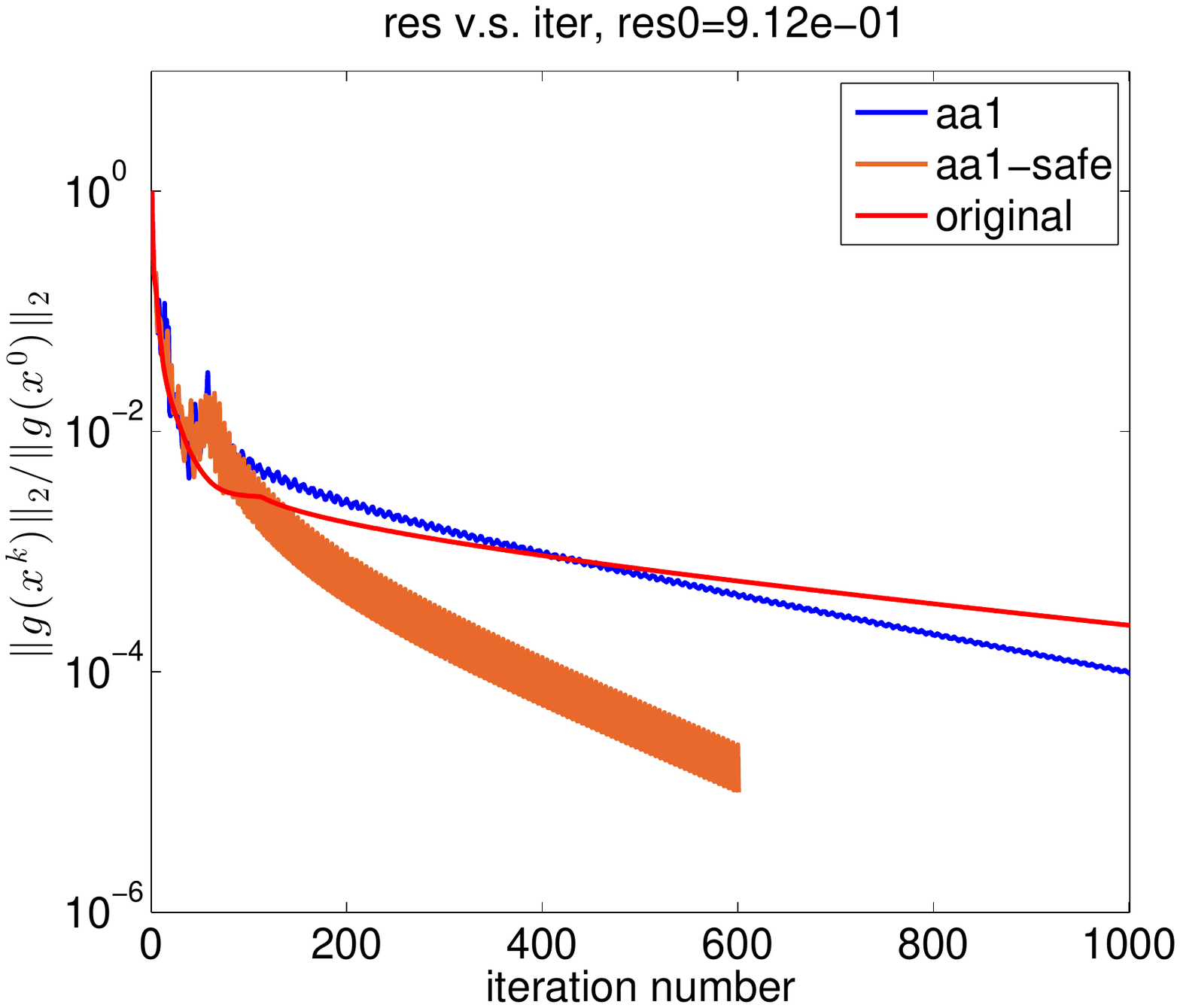}
\includegraphics[trim={1cm 6cm 1cm 6cm},clip,width=7.5cm]{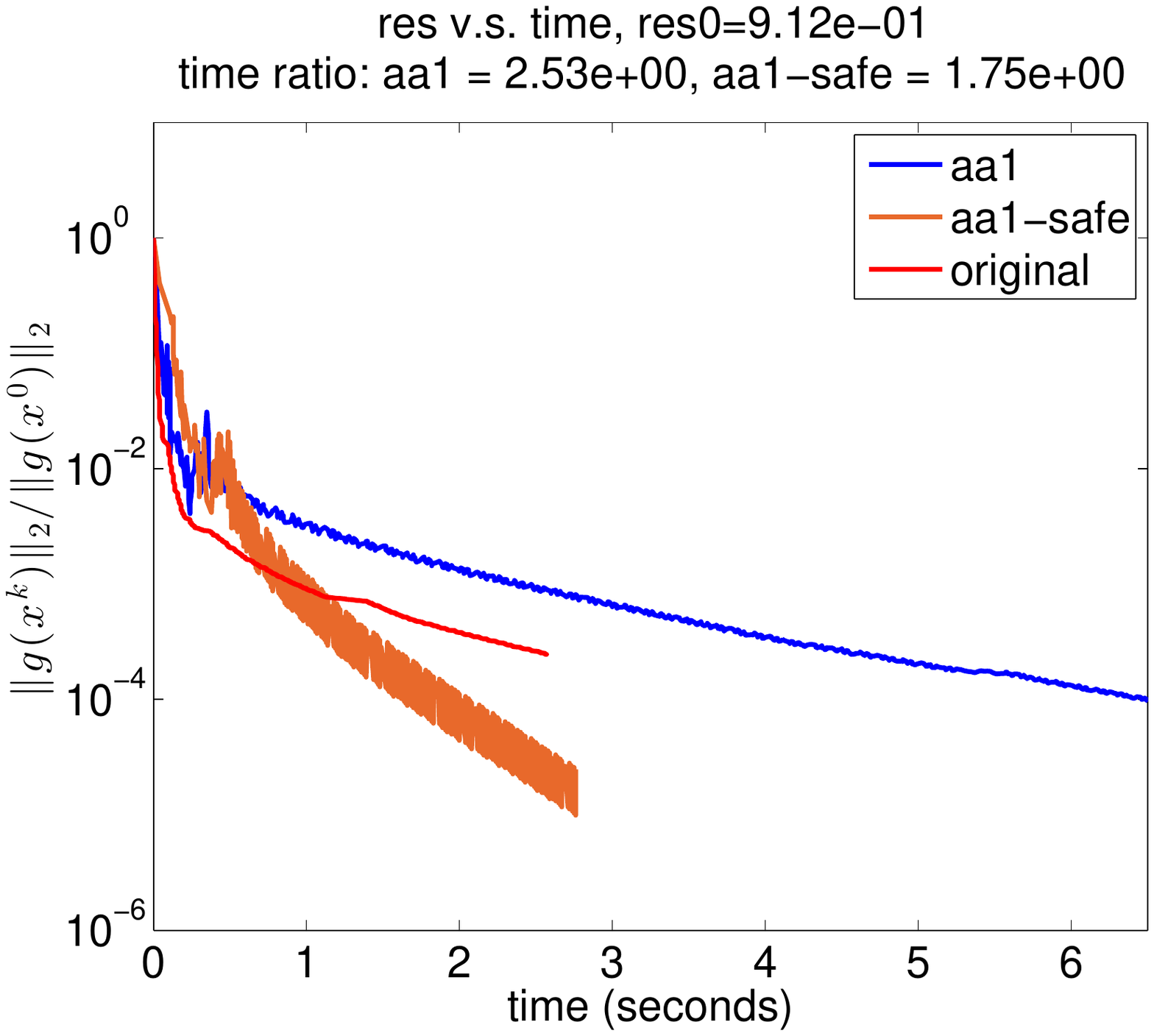}\\
\caption{SCS: SOCP. Left: residual norm versus iteration. 
Right: residual norm versus time (seconds).}
\label{scs-socp}
\end{figure}

\paragraph{VI: Markov decision process.} As described in \S
\ref{contract_norm}, we consider solving a general random Markov
decision process (MDP) using VI. In our experiments, we choose $S=300$
and $A=200$, and we choose a large discount factor $\gamma=0.99$ to make
the problem more difficult, thusly making the improvement of AA more
explicit.

The transition probability matrices $P_a\in\reals^{S\times S}$,
$a=1,\dots,A$ are first generated as \texttt{sprand}
$(S,S,0.01)+0.001I$, where $I$ is the $S$-by-$S$ identity matrix, and
then row-normalized to be a stochastic matrix. Here the addition of
$0.001I$ is to ensure that no all-zero row exists. Similarly, the reward
matrix $R\in \reals^{S\times A}$ is generated by \texttt{sprandn.m} with
sparsity $0.01$. The results are summarized in Figure \ref{vi-rand}.
Notice that the maximum iteration $K_{\max}$ is set to $50$ for better
visualization.
\begin{figure}[h]
\centering
\includegraphics[trim={1cm 6cm 1cm 6cm},clip, width=7.5cm]{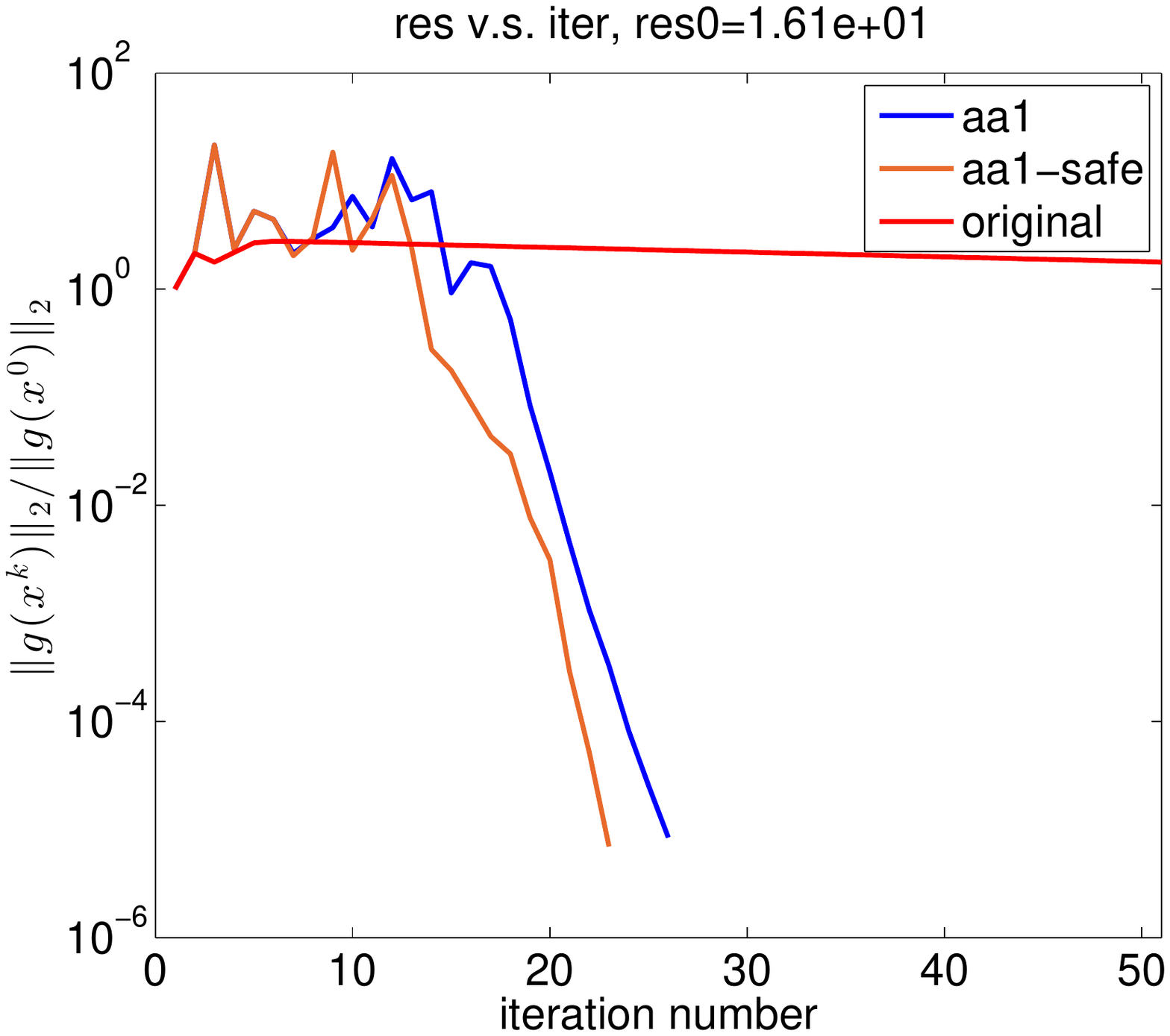}
\includegraphics[trim={1cm 6cm 1cm 6cm},clip,width=7.5cm]{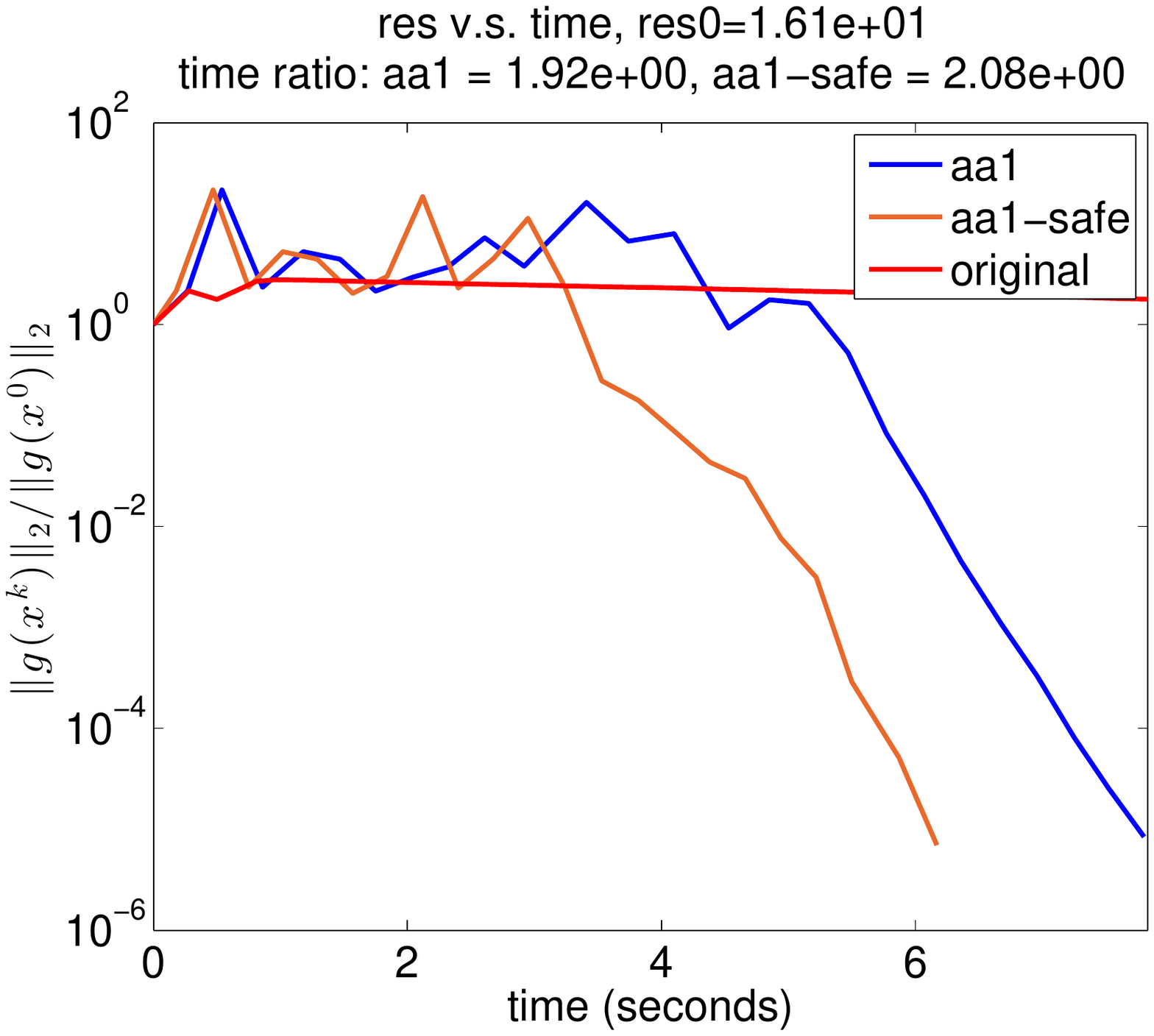}\\
\caption{VI: MDP. Left: residual norm versus iteration. 
Right: residual norm versus time (seconds).}
\label{vi-rand}
\end{figure}

It would be interesting to test the algorithms on more structured MDP,
\eg, the chain MDP, frozen lake, grid world, and more practically an
energy storage problem
(\url{http://castlelab.princeton.edu/html/datasets.htm#storagesalas}).

Moreover, it would also be interesting to see how our algorithm helps as
a sub-solver in other reinforcement learning algorithms.
 
\paragraph{Influence of memory sizes.} Finally, we rerun the VI
experiments above with different memories $m=2,~5,~10,~20,~50$. We
consider a slightly smaller problem size $S=200$ and $A=100$ here for
faster running of a single instance, which facilitates the empirical
verification of the representativeness of the plot we show here. All
other data are exactly the same as in the above example. The results are
summarized in Figure \ref{vi-rand-mem}. Notice that again the maximum
iteration $K_{\max}$ is set to $50$ for better visualization.
 \begin{figure}[h]
\centering
\includegraphics[trim={1cm 6cm 1cm 6cm},clip, width=7.5cm]{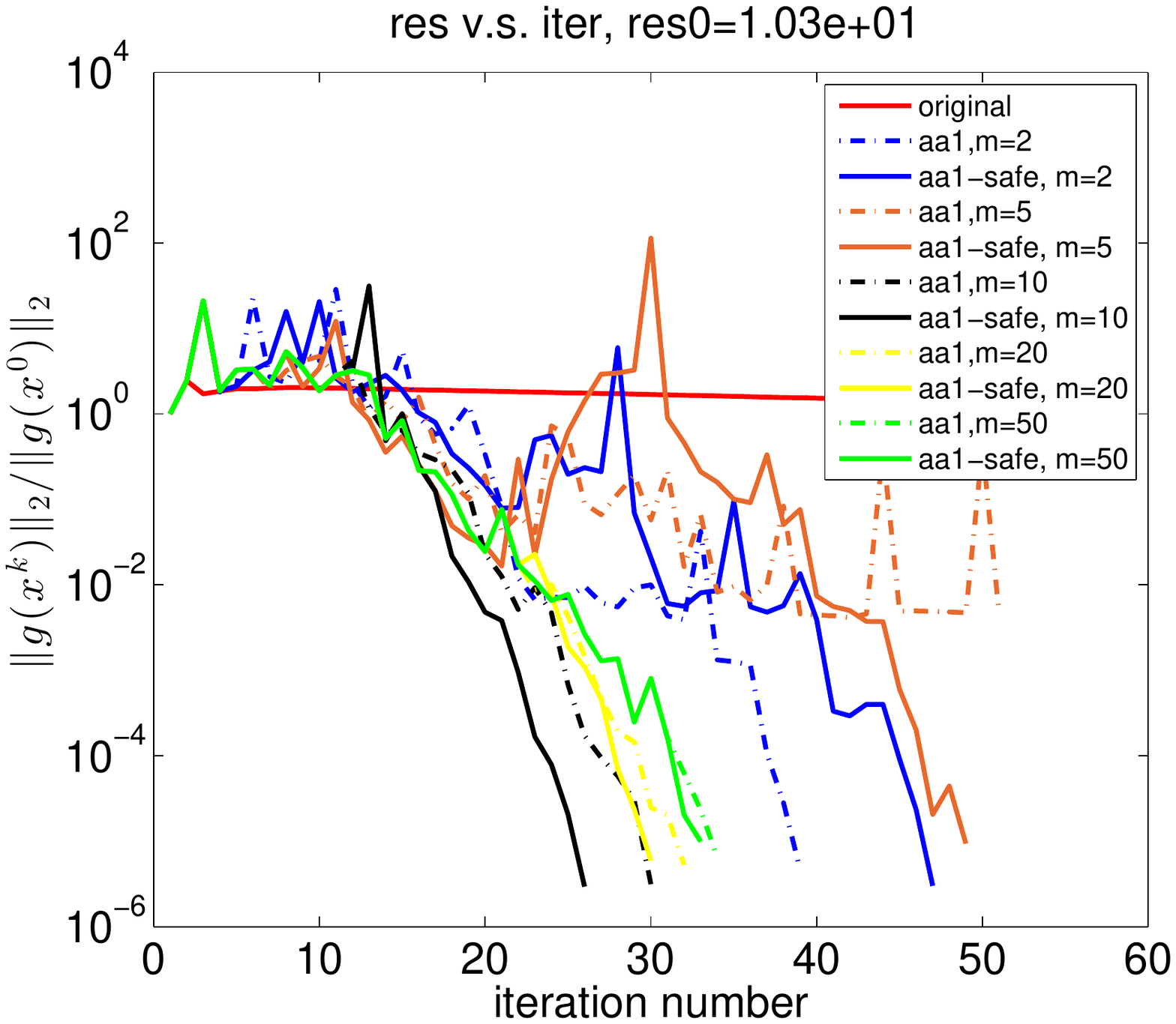}
\includegraphics[trim={1cm 6cm 1cm 6cm},clip,width=7.5cm]{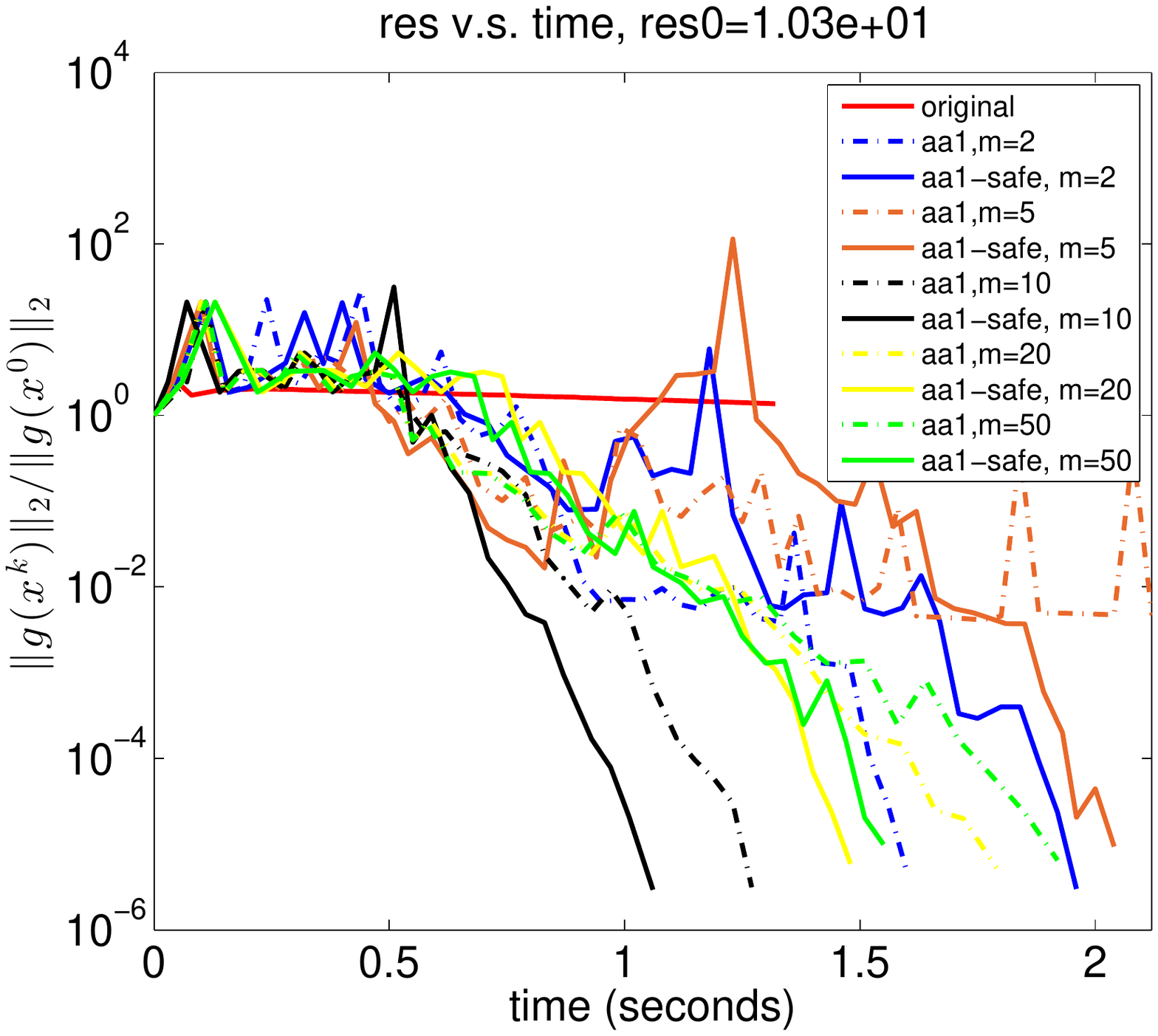}\\
\caption{VI: memory effect. Left: residual norm versus iteration. 
Right: residual norm versus time (seconds).}
\label{vi-rand-mem}
\end{figure}

We can see from the figures that the best performance is achieved (for
both AA-I-S-m and the original AA-I-m) when $m=10$, and deviating from
$10$ in either direction impedes the performance. However, choosing a
reasonably large memory size seems to be a more stable choice compared
to choosing a small one, as can be seen from the case when $m=5$.
 
 \subsubsection{Summary of numerical results}
As we have seen above, surprisingly our algorithm actually performs
better than the original AA-I-m in various cases, sometimes even when
AA-I-m does not seem to suffer much from instability (\eg, SCS for
SOCP), and the improvement is more significant when the latter does
(\eg, GD for regularized logistic regression).

In terms of running time, the safe-guard checking steps does seem to
slightly slow down our algorithm AA-I-S-m, as expected. This is more
obvious for simple problems with larger sizes (\eg, VI for MDP).
Nevertheless, due to the easiness of such problems the extra time is
still quite affordable, making this a minor sacrifice for robustness. In
addition, as in our algorithm the approximate Jacobians are computed
with rank-one updates and re-starts are invoked from time to time, it is
indeed potentially faster than the original fixed-memory AA-I-m, in
which the approximate Jacobian is computed from scratch with $m$
memories. This is also exemplified in most of our numerical experiments
shown above.

Finally, by better parallelization and GPU acceleration, the
per-iteration running time of both AA-I-S-m and  AA-I-m should be
further improved compared to the current single CPU version without any
nontrivial parallelization. This may relieve the contrary result of
acceleration and deceleration of the original AA-I-m in terms of
iteration numbers and time respectively in some examples above (\eg, CO
for facility location), and may further improve the acceleration effect
of our algorithm in terms of computational time.

\section{Extensions to more general settings}\label{ext_var}
In this section, we briefly outline some extended convergence analysis
and results of our algorithm in several more general settings, and
discuss the necessity of potential modifications of our algorithm to
better suit some more challenging scenarios. We then conclude our work
with some final remarks, and shed some light on potential future
directions to be explored.

\paragraph{Quasi-nonexpansive mappings.} A mapping
$f:\reals^n\rightarrow\reals^n$ is called \emph{quasi-nonexpansive} if
for any $y\in X$ a fixed-point of $f$, $\|f(x)-y\|_2\leq \|x-y\|_2$ for
any $x\in \reals^n$. Obviously, non-expansive mappings are
quasi-nonexpansive.

Our convergence theorems actually already hold for these slightly more
general mappings. By noticing that non-expansiveness is only applied
between an arbitrary point and a fixed-point of $f$ in the proof of
Theorem \ref{glb_conv_thm} we immediately see that the same global
convergence result hold if $f$ is only assumed to be quasi-nonexpansive.

Similarly, Theorem \ref{mdp_conv} remain true if the contractivity is
assumed only between an arbitrary point and a fixed-point of $f$, \ie,
$\|f(x)-f(y)\|\leq \gamma\|x-y\|$ for any $x\in\reals^n$ and $y\in X$,
which we term as \emph{quasi-$\gamma$-contractive}.

Formally, we have the following corollary:
\begin{corollary}\label{quasi-nonexp}
Suppose that $\{x^k\}_{k=0}^{\infty}$ is generated by Algorithm 
\ref{alg:AA-I-safe}, and instead of $f$ being non-expansive (in $\ell_2$-norm) 
in (\ref{non_eq}), we only assume that $f$ is either quasi-nonexpansive 
(in $\ell_2$-norm), or quasi-$\gamma$-contractive in some (arbitrary) norm 
$\|\cdot\|$ (\eg, $l_{\infty}$-norm) on $\reals^n$, where $\gamma\in(0,1)$. 
Then we still have $\lim_{k\rightarrow\infty}x^k=x^\star$, where 
$x^\star=f(x^\star)$ is a solution to (\ref{non_eq}). In the latter 
(quasi-$\gamma$-contractive) case, the averaging weight $\alpha$ 
can also be taken as $1$.
\end{corollary}

\paragraph{Iteration-dependent mappings.} Consider the case when the
mapping $f$ varies as iteration proceeds, \ie, instead of a fixed $f$,
we have $f_k:\reals^n\rightarrow\reals^n$ for each $k=0,1,\dots$. The
goal is to find the common fixed-point of all $f_k$ (assuming that it
exists), \ie, finding $x^\star\in \cap_{k\geq 0}X_k$ with $X_k$ being
the fixed-point set of $f_k$. For example, in GD, we may consider a
changing (positive) step size, which will result in a varying mapping
$f$. However, the common fixed-point of all $f_k$ is still exactly the
optimal solution to the original optimization problem. In fact, all
$f_k$ have the same fixed-point set.

Assuming non-expansiveness (actually quasi-nonexpansiveness suffices) of
each $f_k$, $k\geq 0$, and that the fixed-point set $X_k=X$ of $f_k$ is
the same across all $k\geq 0$, both of which hold for GD with positive
varying step sizes described above, we can still follow exactly the same
steps 1 and 2 of the proof for Theorem \ref{glb_conv_thm} to obtain that
$\|g_k\|_2\rightarrow 0$ as $k\rightarrow\infty$, where
$g_k=x^k-f_k(x^k)$, and that $\|x^k-y\|_2$ converges for any fixed-point
$y\in X$.

Unfortunately, in general step 3 does not go through with these changing
mappings. However, if we in addition assume that for any sequence
$x^k\in \reals^n$, $\lim_{k\rightarrow\infty}\|x^k-f_k(x^k)\|_2=0$ and
$x^k\rightarrow\bar{x}$ $\Rightarrow \bar{x}\in X$, then any limit point
of $x^k$ is a common fixed-point of $f_k$'s in $X$. The rest of step 3
then follows exactly unchanged, which finally shows that Theorem
\ref{glb_conv_thm} still holds in this setting.

Formally, we have the following corollary:
\begin{corollary}\label{common_fix}
Suppose that $f_k:\reals^n\rightarrow\reals^n$, $k\geq 0$ are all 
quasi-nonexpansive, and that the fixed-point sets $X_k=\{x\in 
\reals^n\;|\;f_k(x)=x\}$ of $f_k$ are equal to the same set 
$X\subseteq \reals^n$. Assume in addition that for any sequence 
$\{z^k\}_{k=0}^{\infty}\subseteq \reals^n$, if $\lim_{k\rightarrow\infty}
\|z^k-f_k(z^k)\|_2=0$ and $z^k\rightarrow\bar{z}$ for some 
$\bar{z}\in \reals^n$, then $\bar{z}\in X$.
Suppose that $\{x^k\}_{k=0}^{\infty}$ is generated by Algorithm 
\ref{alg:AA-I-safe}, with $f$ replaced with $f_k$ in iteration $k$. 
Then we have $\lim_{k\rightarrow\infty}x^k=x^\star$, where 
$x^\star=f(x^\star)$ is a solution to (\ref{non_eq}). 
\end{corollary}

Although the additional assumption about ``$z^k$'' seems to be a bit
abstract, it does hold if we nail down to the aforementioned specific
case, the GD example with varying step sizes, \ie,
$f_k(x^k)=x^k-\alpha^k\nabla F(x^k)$, and if we assume in addition that
the step size $\alpha^k$ is bounded away from $0$, \ie, $\alpha^k\geq
\epsilon>0$ for some positive constant $\epsilon$ for all $k\geq 0$.

In fact,  by
$\lim_{k\rightarrow\infty}\|g_k\|_2=\lim_{k\rightarrow\infty}\|x^k-f_k(x^k)\|_2=0$,
we have $\lim_{k\rightarrow\infty}\alpha^k\|\nabla F(x^k)\|_2=0$, which
implies that $\lim_{k\rightarrow\infty}\|\nabla F(x^k)\|_2=0$ as
$\alpha^k\geq\epsilon>0$.  In particular, any limit point $\bar{x}$ of
$x^k$ satisfies $\nabla F(\bar{x})=0$ by the continuity of $\nabla F$
assumed in \S \ref{prox_grad}, \ie, $\bar{x}\in X$. Hence we see that
the assumptions made in Corollary \ref{common_fix} all hold in this
example, and hence global convergence of $x^k$ is ensured.

A similar analysis can be carried out to reprove Theorem \ref{mdp_conv}
in this setting.

Nevertheless, it remains open what assumptions are needed in general to
obtain global convergence as in Theorem \ref{glb_conv_thm} and Theorem
\ref{mdp_conv}. In particular, the above analysis fails if $\alpha^k$ is
vanishing, which may arise in many practical cases, \eg, training of
deep neural networks (using stochastic algorithms, which is to be
discussed below). It might be true that some adaptive mechanisms need to
be included in the design of our algorithm to fully fit to this changing
mapping scenario.

\paragraph{Non-expansive mappings in non-Euclidean norms.} Theorem
\ref{mdp_conv} establishes global convergence for contractive mappings
in arbitrary norms. It is hence natural to ask what happens if $f$ is
only non-expansive (instead of contractive) in an arbitrary norm
different from the $\ell_2$-norm. More generally, the norm in which the
mapping $f$ is non-expansive or contractive may also change as the
iterations proceed, which is exactly the case if we perform the same
analysis of HB for a general strongly convex and strongly smooth
objective function. In general, finding out the additional assumptions
needed for the global convergence of the current algorithm in these
settings, or a way to further modify our algorithm to work here if
necessary, may largely contribute to a more flexible algorithm.

\section{Conclusions}\label{conclusion}
In this paper, we modify the type-I Anderson acceleration (AA-I) to propose
a globally convergent acceleration algorithm that works for general
non-expansive non-smooth fixed-point problems, with no additional
assumptions required. We list 9 problem-algorithm combinations, each
supported by one or more concrete problem instances. Our extensive
numerical results show that our modified algorithm is not only more
robust, but also more efficient than the original AA-I. Finally,
extensions to different settings are discussed, and in particular
another theorem is established to ensure global convergence of our
algorithm on value iteration in MDPs and heavy ball methods in QPs.

Despite the success of our algorithm both in theory and practice,
several problems remain open. In particular, the convergence of our
proposed algorithm on general momentum methods (\eg, HB for general
convex constrained optimization, and Nesterov's accelerated gradient
descent), and the potential modifications needed in the absence of such
convergence, deserves a more thorough study.  In addition, it is also
interesting to see how our algorithm performs in stochastic settings,
\ie, when the evaluation of $f$ is noisy. Moreover, some popular
algorithms are still ruled out from our current scenario, \eg,
Frank-Wolfe and Sinkhorn-Knopp. It is thus desired to push our algorithm
more beyond non-expansiveness (apart from quasi-nonexpansiveness and
contractivity in a non-Euclidean norm) to incorporate these interesting
examples, which may also help address the convergence for algorithms in
non-convex optimization settings. In the meantime, a theoretical
characterization of the acceleration effect of our algorithm is still
missing, and in particular no convergence rate has been established for
our algorithm. Although some partial unpublished results relying on
certain differentiability has been obtained by us, it remains super
challenging how we can include all the (non-smooth) mappings listed in
\S \ref{exam_aa} into the assumptions.  Last but not least, numerical
tests with larger sizes and real-world datasets, and more systematic
comparisons with other acceleration methods (\eg, AA-II), are yet to be
conducted, which may finally contribute to a new automatic acceleration
unit for general existing solvers.

\section{Acknowledgements} We thank Ernest Ryu for his comments on the
 possibility and difficulty of analyzing Frank-Wolfe and Nesterov's 
 algorithms in fixed-point mapping frameworks. We thank Tianyi Lin 
 for his advice on generating appropriate random data for LPs, and 
 his comments on the challenge of convergence order analysis in 
 quasi-Newton methods. We thank Qingyun Sun for the general discussions 
 in the early stage of the project, especially those related to 
 interpreting AA as multi-secant methods \cite{FangSecant, WalkerNi}.
  We also thank Michael Saunders for his constant encouragement and 
  positive feedback on the progress of the project on a high level. 
  We are also grateful to Zaiwen Wen for the inspiring discussions on 
  related literature. Last but not least, we also thank Anran Hu for 
  her suggestions on the organization and presentation of the paper, 
  and her suggestion on considering the heavy ball algorithm.

\newpage

\bibliographystyle{alpha}
\bibliography{scs_2.0}

\newpage
\appendix
\section*{Appendices}
\addcontentsline{toc}{section}{Appendices}
\renewcommand{\thesubsection}{\Alph{subsection}}
\paragraph{Proof of Proposition \ref{r1-update}.}
\begin{proof}
Suppose that $Z_k\in \reals^{n\times n-m_k}$ is a basis of 
$\text{span}(S_k)^{\perp}$. 
Then from (\ref{Bk}), we see that $B_kS_k=Y_k$ and $B_kZ_k=Z_k$. 
Now we prove that $B_k^{m_k}$ also satisfies these linear equations. 

Firstly, we show by induction that $B_k^iS_k^i=Y_k^i$, where 
$S_k^i=(s_{k-m_k},\dots,s_{k-m_k+i-1})$ and $Y_k^i=(y_{k-m_k},
\dots,y_{k-m_k+i-1})$, $i=1,\dots,m_k$. 

\textit{Base case.} For $i=1$, we have $B_k^0=I$, $S_k^1=s_{k-m_k}$ 
and $Y_k^1=y_{k-m_k}$, and hence
\[
B_k^1S_k^1=B_k^1s_{k-m_k}=s_{k-m_k}+\dfrac{(y_{k-m_k}-s_{k-m_k})
\hat{s}_{k-m_k}^Ts_{k-m_k}}{\hat{s}_{k-m_k}^Ts_{k-m_k}}=y_{k-m_k}=Y_k^1.
\]

\textit{Induction.} Suppose that we have proved the claim for $i=l$. 
Then for $i=l+1$, 
\[
B_k^{l+1}S_k^{l}=B_k^{l}S_k^{l}+\dfrac{(y_{k-m_k+l}-B_k^{l}s_{k-m_k+l})
\hat{s}_{k-m_k+l}^TS_k^{l}}{\hat{s}_{k-m_k+l}^Ts_{k-m_k+l}}=Y_k^{l},
\]
where we used the hypothesis and the fact that by orthogonalization, 
$\hat{s}_{k-m_k+l}^TS_k^{l}=0$. 

In addition, we also have
\[
B_k^{l+1}s_{k-m_k+l}=B_k^{l}s_{k-m_k+l}+\dfrac{(y_{k-m_k+l}-B_k^{l}
s_{k-m_k+l})\hat{s}_{k-m_k+l}^Ts_{k-m_k+l}}{\hat{s}_{k-m_k+l}^Ts_{k-m_k+l}}
=y_{k-m_k+l},
\]
which shows that $B_k^{l+1}S_k^{l+1}=Y_k^{l+1}$ together with the 
equalities above. This completes the induction, and in particular shows that 
\[
B_k^{m_k}S_k=B_k^{m_k}S_k^{m_k}=Y_k^{m_k}=Y_k.
\]

Secondly, we show that $B_k^{m_k}Z_k=Z_k$. To see this, notice that by
 $\text{span}(\hat{s}_{k-m_k},\dots,\hat{s}_{k-1})=\text{span}(S_k)$, 
 we have $\hat{s}_{k-i}^TZ_k=0$ ($i=1,\dots,m_k$), and hence
\[
B_k^{m_k}Z_k=Z_k+\sum_{i=0}^{m_k-1}\dfrac{(y_{k-m_k+i}-B_k^{i}s_{k-m_k+i})
\hat{s}_{k-m_k+i}^TZ_k}{\hat{s}_{k-m_k+i}^Ts_{k-m_k+i}}=Z_k.
\]

Finally, since $(S_k,Z_k)$ is invertible, we see that the equation 
$B(S_k,Z_k)=(Y_k,Z_k)$ has a unique solution, and hence $B_k=B_k^{m_k}$.
\end{proof}

\paragraph{Proof of Theorem \ref{mdp_conv}.}
\begin{proof} Below we use the same notation for the vector norm $\|\cdot\|$
 on $\reals^n$ and its induced matrix norm on $\reals^{n\times n}$, \ie, 
 $\|A\|=\sup_{x\neq 0}\|Ax\|/\|x\|$ for any $A\in\reals^{n\times n}$. Again, 
 we partition the iteration counts into two subsets accordingly, with 
 $K_{AA}=\{k_0,k_1,\dots\}$ being those iterations that passes line 12, 
 while $K_{KM}=\{l_0,l_1,\dots\}$ being the rest that goes to line 14. 
 Denote as $y$ the unique fixed point of $f$ in (\ref{non_eq}), where 
 the uniqueness comes from contractivity of $f$. 

The proof is completed by considering two scenarios separately. The
first is when $K_{KM}$ is \emph{finite}, in which case the proof is
identical to Theorem \ref{glb_conv_thm}, as neither the contractivity
nor the non-expansiveness comes into play after a finite number of
iterations. The second is when $K_{KM}$ is \emph{infinite}, the proof 
of
which is given below.

Suppose from now on that $K_{KM}$ is an infinite set.

On one hand, for $k_i\in K_{AA}$, by Corollary \ref{algHkbound} and norm
equivalence on $\reals^n$, we have $\|H_{k_i}\|\leq C'$ for some
constant $C'$ independent of the iteration count, and similarly
$\|g_{k_i}\|\leq C^{''}\|g_{k_i}\|_2\leq
C^{''}D\bar{U}(i+1)^{-(1+\epsilon)}$. Hence
\BEQ\label{ineq1'}
\begin{split}
\|x^{k_i+1}-y\|&\leq \|x^{k_i}-y\|+\|H_{k_i}g_{k_i}\|\\
&\leq \|x^{k_i}-y\|+C'\|g_{k_i}\|\leq \|x^{k_i}-y\|+\underbrace{C'C^{''}D
\bar{U}(i+1)^{-(1+\epsilon)}}_{\epsilon_{k_i}'}.
\end{split}
\EEQ

On the other hand, for $l_i\in K_{KM}$ ($i\geq 0$), (\ref{ineq2}) does not 
hold anymore. Instead, we have by $\gamma$-contractivity that 
\BEQ\label{ineq2'}
\|x^{l_i+1}-y\|\leq \gamma \|x^{l_i}-y\|\leq \|x^{l_i}-y\|.
\EEQ

Hence by defining $\epsilon_{l_i}'=0$, we again see from (\ref{ineq1'})
and (\ref{ineq2'}) that
\BEQ\label{fejer'}
\|x^{k+1}-y\|\leq \|x^k-y\|+\epsilon_k',
\EEQ with $\epsilon_k'\geq 0$ and
$\sum_{k=0}^{\infty}\epsilon_k'=\sum_{i=0}^{\infty}\epsilon_{k_i}'<\infty$.
\vspace{0.1cm}

Now define $a_j=\sum_{l_j+1\leq
k<l_{j+1}}\epsilon_k'=\sum\nolimits_{l_j\leq
k_i<l_{j+1}}\epsilon_{k_i}'$. Then we have
$\sum_{j=0}^{\infty}a_j=\sum_{k=0}^{\infty}\epsilon_k'<\infty$, and in
particular $\lim_{j\rightarrow\infty}a_j=0$ and $0\leq a_j\leq E'$ for
some $E'>0$.  Then we have
\BEQ\label{contract_bd}
\begin{split}
\|x^{l_i}-y\|\leq \|x^{l_{i-1}+1}-y\|+\sum\nolimits_{l_{i-1}+1\leq 
k<l_i}\epsilon_k'\leq \gamma\|x^{l_{i-1}}-y\|+a_{i-1}.
\end{split}
\EEQ

By telescoping the above inequality, we immediately see that
\BEQ\label{contract_bd2}
\begin{split}
\|x^{l_i}-y\|&\leq \gamma^i\|x^{l_0}-y\|+\sum_{k'=0}^{i-1}
\gamma^{k'}a_{i-1-k'}\\
&\leq \gamma^i\|x^0-y\|+E'\sum_{k'=\lfloor (i-1)/2\rfloor}^{i-1}
\gamma^{k'}+\sum_{k'=0}^{\lfloor (i-1)/2\rfloor-1}a_{i-1-k'}\\
&\leq \gamma^i\|x^0-y\|+\frac{E'}{1-\gamma}\gamma^{\lfloor (i-1)/2
\rfloor}+\sum_{k'=\lceil (i-1)/2\rceil}^{\infty}a_{k'}
\end{split}
\EEQ where we used the fact that $l_0=0$ by Line 2 of Algorithm
\ref{alg:AA-I-safe}. In particular, by using the fact that
$\gamma^i\rightarrow 0$ as $i\rightarrow \infty$, and that
$\sum_{k'=k}^{\infty}a_{k'}\rightarrow 0$ as $k\rightarrow\infty$, we
see that
\BEQ\label{lim-li}
\lim_{i\rightarrow\infty}\|x^{l_i}-y\|=0.
\EEQ

Finally, for any $k>0$, define $i_k=\argmax_i~\{l_i<k\}$. Then we have
$\lim_{k\rightarrow\infty}i_k=\infty$ as $K_{KM}$ is infinite, and
moreover, $l_{i_k+1}\geq k$. Hence we obtain that
\BEQ\label{reduction}
\begin{split}
\|x^k-y\|&\leq \|x^{l_{i_k}+1}-y\|+\sum\nolimits_{l_{i_k}+1\leq k'\leq 
k-1}\epsilon_{k'}'\\
&\leq \gamma\|x^{l_{i_k}}-y\|+a_{i_k},
\end{split}
\EEQ from which and (\ref{lim-li}) we immediately conclude that
\BEQ\label{lim-k}
\lim_{k\rightarrow\infty}\|x^k-y\|=0,
\EEQ
\ie, $\lim_{k\rightarrow\infty}x^k=y$, where $y$ is the unique solution
of (\ref{non_eq}). This completes our proof.
\end{proof}

\end{document}